\documentclass{amsart}
\usepackage{graphicx}
\usepackage{amsmath, amsfonts,amssymb}
\usepackage{amsthm}
\usepackage{mathrsfs}
\usepackage{etex}
\usepackage[all,cmtip,arrow,2cell]{xy}
\usepackage[colorlinks,
             linkcolor=black,
             citecolor=blue,
             bookmarksnumbered=true]{hyperref}
\usepackage{appendix}
\usepackage{bbold}
\newtheorem{prop}{Proposition}[section]
\newtheorem{thm}{Theorem}[section]
\newtheorem{cor}{Corollary}[section]
\newtheorem{lem}[thm]{Lemma}

\theoremstyle{definition}
\newtheorem{dfn}{Definition}[section]

\theoremstyle{remark}
\newtheorem*{rmk}{Remark}
\theoremstyle{remark}

\theoremstyle{remark}

\newcommand{\G}         {\mathcal{G}}
\newcommand{\h}         {\mathcal{H}}
\newcommand{\K}         {\mathcal{K}}
\newcommand{\Ll}         {\mathcal{L}}
\newcommand{\B}         {\mathcal{B}}
\newcommand{\E}         {\mathcal{E}}
\newcommand{\F}         {\mathcal{F}}

\def\D{\mathscr D}
\def\C{\mathscr C}
\def\X{\mathscr X}
\def\Y{\mathscr Y}
\def\Z{\mathscr Z}
\def\W{\mathscr W}
\def\S{\mathscr S}
\def\g{\mathscr G}
\def\bq{\mathfrak G}
\newcommand{\Top} {\mathfrak{Top}}
\newcommand{\Et} {\mathfrak{Et}}
\newcommand{\Eft} {\mathfrak{EffEt}}
\def\Ha{\mathscr{H}}

\def\Sts{\St\left(S\right)}

\newcommand{\holim}{\operatorname{holim}}

\newcommand{\hl}        {\underset{{\longleftarrow\!\!\!-\!\!\!-\!\!\!-\!\!\!-\!\!\!-\!\!\!-} } \holim  \:}
\newcommand{\hc}        {\underset{{-\!\!\!-\!\!\!-\!\!\!-\!\!\!-\!\!\!-\!\!\!\longrightarrow}} \holim  \:}

\newcommand{\Hom}{\operatorname{Hom}}
\DeclareMathOperator{\St}{St}
\DeclareMathOperator{\Set}{Set}
\DeclareMathOperator{\Sh}{Sh}
\DeclareMathOperator{\sit}{Site}
\newcommand{\Ef}{\operatorname{Eff}}

\DeclareMathAlphabet{\scr}{OT1}{pzc}%
                                 {m}{it}
\def\gb{\scr{Gerbe}\left(\X\right)}

\def\gets{\scr{Gerbed}\left(\Eft\right)}

\newcommand{\rt} { \rotatebox[origin=C]{90}{$\perp$} }

\def\rrrarrow{\hspace{.05cm}\mbox{\,\put(0,-3){$\rightarrow$}\put(0,1){$\rightarrow$}\put(0,5){$\rightarrow$}\hspace{.45cm}}}
\def\rrarrow{  \hspace{.05cm}\mbox{\,\put(0,-2){$\rightarrow$}\put(0,2){$\rightarrow$}\hspace{.45cm}}}
\def\acts{\hspace{.1cm}{\setlength{\unitlength}{.30mm}\linethickness{.09mm}
                        \begin{picture}(8,8)(0,0)\qbezier(7,6)(4.5,8.3)(2,7)\qbezier(2,7)(-1.5,4)(2,1)\qbezier(2,1)(4.5,-.3)(7,2)
                                                 \qbezier(7,6)(6.1,7.5)(6.8,9)\qbezier(7,6)(5,6.1)(4.2,4.4)
                        \end{picture}\hspace{.1cm}}}

\def\longlongrightarrow{-\!\!\!-\!\!\!-\!\!\!-\!\!\!-\!\!\!-\!\!\!\longrightarrow}

\begin{document}

\title{Sheaf Theory for \'Etale Geometric Stacks}
\author{David Carchedi}

\begin{abstract}
We generalize the notion of a small sheaf of sets over a topological space or manifold to define the notion of a small stack of groupoids over an \'etale topological or differentiable stack. We then provide a construction analogous to the \'etal\'e space construction in this context, establishing an equivalence of $2$-categories between small stacks over an \'etale stack and local homeomorphisms over it. We go on to characterize small sheaves and gerbes. We show that ineffective data of \'etale stacks is completely described by the theory of small gerbes. Furthermore, it is shown that \'etale stacks (and in particular orbifolds) induce a small gerbe over their effective part, and all gerbes arise in this way. It follows that ineffective orbifolds, sometimes called non-reduced orbifolds, encode a canonical gerbe over their effective (or reduced) part. For nice enough classes of maps, for instance submersions, we show that \'etale stacks are equivalent to a $2$-category of gerbed effective \'etale stacks. Along the way, we also prove that the $2$-category of topoi is a full reflective sub-2-category of localic stacks.
\end{abstract}

\maketitle
\markboth{David Carchedi}{Sheaf Theory for \'Etale Geometric Stacks}


\tableofcontents

\section{Introduction}
The purpose of this article is to extend the theory of small sheaves of sets over spaces to a theory of small stacks of groupoids over \'etale topological, differentiable, and localic stacks. We provide a construction analogous to the \'etal\'e space construction in this context and establish an equivalence of $2$-categories between small stacks over an \'etale stack and local homeomorphisms over it. This theory provides an interpretation of the ineffective data of any \'etale stack as a small gerbe over its effective part. Moreover, every small gerbe over an effective \'etale stack $\Y$ arises from some \'etale stack $\Z$ whose effective part is equivalent to $\Y$. In particular, this applies to orbifolds, showing ineffective orbifolds, sometimes called non-reduced orbifolds, encode a canonical gerbe over their effective (or reduced) part.

\'Etale stacks model quotients of spaces by certain local symmetries, and their points can posses intrinsic (discrete) automorphism groups. A more or less direct consequence of the existence of points with non-trivial automorphism groups is that \'etale stacks form not only a category, but a bicategory. A widely studied class of such stacks are orbifolds, which have a wide range of uses in foliation theory, string theory, and conformal field theory. More generally, \'etale stacks are an important class of stacks as they include not only all orbifolds, but more generally, all stacky leaf spaces of foliated manifolds. The passage from spaces to \'etale stacks is a natural one as such a passage circumvents many obstructions to geometric problems. For example, it is not true that every foliation of a manifold $M$ arises from a submersion $f:M \to N$ of manifolds, however, it is true that every foliation on $M$ arises from a submersion $M \to \X$, where $\X$ is allowed to be an \'etale differentiable stack \cite{Ie}. Similarly, it is not true that every Lie algebroid over a manifold $M$ integrates to a Lie groupoid $\G \rightrightarrows M$, \cite{bracket}, however it is true when the arrow space $\G$ is allowed to be an \'etale differentiable stack \cite{stacklie}. \'Etale stacks are also a natural setting to consider small sheaves (and more generally small stacks), as the results of \cite{Dorette} imply that \'etale stacks are faithfully represented by their topos of small sheaves.

Recall that for a topological space $X$, a \emph{small sheaf} over $X$ is a sheaf over its category of open subsets, $\mathcal{O}\left(X\right)$, where the arrows are inclusions. The corresponding topos is denoted as $\Sh\left(X\right).$  For small (pre-)sheaves over $X,$ there is an \'etal\'e space construction:

Given a presheaf $F$ over $X,$ there exists a space $\underline L\left(F\right)$ and a local homeomorphism $$L(F):\underline L\left(F\right) \to X$$ such that for every open subset $U$ of $X,$ sections of the map $L\left(F\right)$ over $U$ are in bijection with elements of $aF\left(U\right),$ where $aF$ is the sheaf associated to $F$. The space $\underline L\left(F\right)$ (together with its map down to $X$) is called the \emph{\'etal\'e space} of $F$.
More precisely there is a pair of adjoint functors
$$\xymatrix{\Set^{\mathcal{O}\left(X\right)^{op}} \ar@<-0.5ex>[r]_-{L} & \mathbb{TOP}/X \ar@<-0.5ex>[l]_-{\Gamma}},$$ such that $$L\left(U\right)=U \hookrightarrow X.$$ Here, $L$ takes a presheaf to its \'etal\'e space and $\Gamma$ takes a space over $X$ to its sheaf of sections. This adjunction restricts to an equivalence

$$\xymatrix{\Sh\left(X\right) \ar@<-0.5ex>[r]_-{L} & Et\left(X\right) \ar@<-0.5ex>[l]_-{\Gamma}},$$
between the category of small sheaves over $X$ and the category of local homeomorphisms over $X$.

Similarly, a \emph{small stack} over a space $X$ is a stack (of groupoids) $\Z$ over the category of open subsets of $X$. It is not reasonable to hope to construct an \'etal\'e \emph{space} for $\Z$ whose sections over an open subset $U$ are equivalent to $\Z\left(U\right)$, unless each of the groupoids $\Z\left(U\right)$ are (equivalent to) sets, since sections of a map of spaces can only form a set. Hence, one can only find an \'etal\'e space for sheaves. If there were to be an \'etal\'e ``space'' associated to a stack, this ``space'' would need to actually be an object of a bicategory, so that sections of the map $$\underline L\left(\Z\right) \to X$$ could form a genuine non-discrete groupoid. In this paper, we show that this can be accomplished if we, instead of searching for an \'etal\'e \emph{space}, find an \'etal\'e \emph{\'etale stack}, which we less awkwardly name the \emph{\'etal\'e realization} of $\Z$. In fact, we extend this result to the setting of small stacks of groupoids over \'etale stacks.

We define the notion of a small sheaf and stack over an \'etale topological, differentiable, or localic stack in much the same way as for topological spaces, by finding an appropriate substitute for a Grothendieck site of open subsets. Sheaves over this site are what we call small sheaves over $\X,$ and similarly for stacks. For example, if $G$ is a discrete group acting on a space $X,$ the stacky quotient $X//G$ is an \'etale topological stack, and a small sheaf over $M//G$ is the same as a $G$-equivariant sheaf over $M,$ which can be described as a space $E$ equipped with an action of $G$ and a local homeomorphism $E \to M$ which is equivariant with respect to the two $G$-actions. If $\X$ happens to be an orbifold, then there is an existing notion of sheaf over $\X,$ and it agrees with the definition of a small sheaf over $\X$ in the sense of this paper.

Small sheaves and stacks need to be distinguished from their large counterparts. The $2$-topos of $\emph{large stacks}$ over an \'etale stack $\X$  is the slice $2$-topos $$\St\left(\mathbb{TOP}\right)/\X,$$ in the case of topological stacks, and similarly for the localic and smooth setting. This distinction is highlighted in \cite{Metzler}. A small sheaf or stack over a space or stack should be thought of as algebraic data attached to that space or stack, whereas a large sheaf or stack should be thought of as a geometric object sitting over it. In particular, the study of small gerbes in this setting seems to be a relatively recent endeavor. It should be noted that nearly all applications in the literature of gerbes in differential geometry are applications of \emph{large} gerbes, moreover large gerbes with band $U\!\!\left(1\right)$, so-called bundle-gerbes (see e.g. \cite{bunger1,diffg}). Not every large gerbe is a small gerbe, nor is every large gerbe a bundle gerbe. To the author's knowledge, there has been, as of yet, little application of small gerbes in differentiable geometry or topology. However, the classification of extensions of regular Lie groupoids given in \cite{reg} may be interpreted in terms of small gerbes over \'etale stacks. Nonetheless, there are plenty of examples of small gerbes right under everyone's noses, in the disguise of ineffective data, e.g. every ineffective orbifold gives rise to a small gerbe as does any almost-free action of a Lie group on a manifold. One aim of this paper is to establish the technical tools necessary to begin the study of these objects.

\subsection{Small gerbes and ineffective isotropy data}
Besides establishing a theory of small sheaves and stacks over \'etale stacks, this paper unravels the mystery behind \emph{ineffective data} of \'etale stacks. Suppose that $G$ is a finite group acting on a manifold $M.$ The stacky-quotient $M//G$ is an \'etale differentiable stack, and in particular, an orbifold. Points of this stacky-quotient are the same as points of the naive quotient, that is, orbits of the action. These are precisely images of points of $M$ under the quotient map $M \to M//G.$ For a particular point $x \in M$, if $\left[x\right]$ denotes the point in $M//G$ which is its image, then $$Aut\left(\left[x\right]\right)\cong G_x.$$
If this action is not faithful, then  there exists a non-trivial kernel $K$ of the homomorphism
\begin{equation}\label{eq:actionmap}
\rho:G \to \mathit{Diff}\left(M\right).
\end{equation}
In this case, any element $k$ of $K$ acts trivially and is tagged-along as extra data in the automorphism group $$Aut\left(\left[x\right]\right)\cong G_x$$ of each point $\left[x\right]$ of the stack $M//G$. In fact, $$\bigcap\limits_{x \in M} G_x=Ker\left(\rho\right).$$ In particular, $\rho$ restricted to $Aut\left(\left[x\right]\right)$ becomes a homomorphism
\begin{equation}\label{eq:locact}
\rho_x:Aut\left(\left[x\right]\right) \to \mathit{Diff}\left(M\right)_x
\end{equation}
to the group of diffeomorphisms of $M$ which fix $x$. This homomorphism is injective for all $x$ if and only if the kernel of $\rho$ is trivial. The kernel of each of these homomorphisms is the ``inflated'' part of each automorphism group, and is called the \emph{ineffective isotropy group} of $\left[x\right]$. Up to the identification $$Aut\left(\left[x\right]\right)\cong G_x,$$ each of these ineffective isotropy groups is $K.$ This extra information is stripped away when considering the stacky-quotient $$M//\left(G/K\right),$$ that is to say, $M//\left(G/K\right)$ is the \emph{effective part} of $M//G.$

Hence, having a kernel to the action (\ref{eq:actionmap}) artificially inflates each automorphism group. As an extreme example, suppose the action $\rho$ is trivial, and consider the stacky quotient $M//G.$ It is the same thing as $M$ except each point $x$, has the group $G,$ rather the the trivial group, as an automorphism group. These automorphisms are somehow artificial, since the action $\rho$ sees nothing of $G.$ In this case, the entire automorphism group of each point is its ineffective isotropy group, and this is an example of a purely ineffective orbifold. Since these arguments are local, the situation when $\X$ is an \'etale stack formed by gluing together stacks of the form $M_\alpha//G_\alpha$ for actions of finite groups, i.e. a general orbifold, is completely analogous. 

For a more general \'etale stack, for example a stack of the form $M//G$ where $G$ is discrete but not finite, there is no such local action of the automorphisms groups as in (\ref{eq:locact}), but the situation can be mimicked at the level of germs. There exists a manifold $V$ and a (representable) local homeomorphism $$V \to \X$$ such that for every point $$x:* \to \X,$$
\begin{itemize}
\item[i)] the point $x$ factors (up to isomorphism) as $* \stackrel{\tilde x}{\longrightarrow} V\stackrel{p}{\longrightarrow} \X,$ and
\item[ii)] there is a canonical homomorphism $\tilde\rho_x: Aut\left(x\right) \to \mathit{Diff}_{\tilde x}\left(V\right),$
\end{itemize}
where $\mathit{Diff}_{\tilde x}\left(V\right)$ is the group of germs of locally defined diffeomorphisms of $V$ that fix $\tilde x.$ The kernel of each of these maps is again the inflated part of the  automorphism group. In the case where $\X$ is of the form $M//G$ for a finite group $G$ (or more generally, when $\X$ is an orbifold) the kernel of $\tilde \rho_x$ is the same as the kernel of (\ref{eq:locact}), for each $x$. In general, each $Ker\left(\tilde \rho_x\right)$ is called an \emph{ineffective isotropy group}. Unlike in the case of a global quotient $M//G,$ these groups need not be isomorphic for different points of the stack. However, these kernels may be killed off to obtain the so-called \emph{effective part} of the \'etale stack.

There is another way of trying to artificially inflate the automorphism groups, and this is through gerbes. As a starting example, if $M$ is a manifold, a gerbe over $M$ is a stack $\g$ over $M$ such that over each point $x$ of $M,$ the stalk $\g_x$ is equivalent to a group. From such a gerbe, one can construct an \'etale stack which looks just like $M$ except each point $x,$ now instead of having a trivial automorphism group, has (a group equivalent to) $\g_x$ as its automorphism group. This construction was alluded to in \cite{pres2}. One can use this construction to show that \'etale stacks whose effective parts are manifolds are the same thing as manifolds equipped with a gerbe. In this paper, we show that this result extends to general \'etale stacks, namely that any \'etale stack $\X$ encodes a small gerbe (in the sense of Definition \ref{dfn:smgb}) over its effective part $\Ef\left(\X\right)$, and moreover, every small gerbe over an effective \'etale stack $\Y$ arises uniquely from some \'etale stack $\Z$ whose effective part is equivalent to $\Y$. The construction of an \'etale stack  $\Z$ out of an effective \'etale stack $\Y$ equipped with a small gerbe $\g,$ is precisely the \'etal\'e realization of the gerbe $\g$. In such a situation, there  is a natural bijection between the points of $\Z$ and the points of $\Y$, the only difference being that points of $\Z$ have more automorphisms. For $x$ a point $\Z$, its ineffective isotropy group, i.e. the kernel of $$Aut\left(x\right) \to \mathit{Diff}_{\tilde x}\left(V\right),$$ is equivalent to the stalk $\g_x.$

\subsection{Organization and main results}
Section \ref{sec:small} starts by briefly recalling the basic definitions of \'etale stacks. It is then explained how to associate to any stack a canonical topos of small sheaves in a functorial way. In case the stack in question is presented by a spatial groupoid $\G$, this topos is equivalent to the classifying topos $\B\G$ as defined in \cite{cont}. It is then shown how the results of \cite{Dorette} imply that \'etale stacks are faithfully represented by their topos of small sheaves. Following \cite{pres}, we associate to every (atlas for an) \'etale stack a canonical small site of definition for its topos of small sheaves. We define small stacks to be stacks over this site. We then give an abstract description of a generalized \'etal\'e space construction in this setting, which we call the \'etal\'e realization construction.

As a demonstration of the abstract machinery developed in this section, we also prove a tangential (yet highly interesting) theorem to the effect that, in some sense, topological stacks subsume Grothendieck topoi, once we replace the role of topological spaces with that of locales:

\begin{thm}
There is a $2$-adjunction
$$\xymatrix@C=1.5cm{\Top \ar@<-0.5ex>[r]_{\S} & \mathfrak{LocSt} \ar@<-0.5ex>[l]_{\Sh},}$$
exhibiting the bicategory of topoi (with only invertible $2$-cells) as a reflective subbicategory of localic stacks (stacks coming from localic groupoids).
\end{thm}

Section \ref{sec:conc} aims at giving a concrete description of the abstract construction given in section \ref{sec:small}. For this, we choose to represent small stacks by groupoid objects in the topos of small sheaves. We then show how a generalized action-groupoid construction gives us a concrete model for the \'etal\'e realization of small stacks. As a consequence, we prove:

\begin{thm}
For any \'etale topological, differentiable, or localic stack $\X$, there is an adjoint-equivalence of $2$-categories
$$\xymatrix{\St\left(\X\right)  \ar@<-0.5ex>[r]_-{L}  & Et\left(\X\right)\ar@<-0.5ex>[l]_-{\Gamma}},$$
between small stacks over $\X$ and the $2$-category of \'etale stacks over $\X$ via a local homeomorphism.
\end{thm}

Here $L$ is the \'etal\'e realization functor, and $\Gamma$ is the ``stack of sections'' functor. We also determine which local homeomorphisms over $\X$ correspond to sheaves:

\begin{thm}
A local homeomorphism $f:\Z \to \X$ over an \'etale stack $\X$ is equivalent to the \'etal\'e realization of a small sheaf $F$ over $\X$ if and only if it is a representable map.
\end{thm}

Section \ref{sec:section} provides a concrete model for the ``stack of sections'' functor $\Gamma$ in terms of groupoid objects in the topos of small sheaves.

In section \ref{sec:effective}, we introduce the concept of an effective \'etale stack and show how to associate to every \'etale stack $\X$ an effective \'etale stack $\Ef\left(\X\right)$, which we call its effective part. Although this construction is not functorial with respect to all maps, we show that it is functorial with respect to any category of open maps which is \'etale invariant (see Definition \ref{dfn:local}). Examples of open \'etale invariant classes of maps include open maps, local homeomorphisms, and submersions.

The subject of section \ref{sec:gerbe} is the classification of small gerbes. For $\X$ an effective \'etale stack, the answer is quite nice:

\begin{thm}
For an effective \'etale stack $\X$, a local homeomorphism $f:\g \to \X$ is equivalent to the \'etal\'e realization of a small gerbe over $\X$ if and only if $$\Ef\left(f\right):\Ef\left(\g\right) \to \Ef\left(\X\right) \simeq \X$$ is an equivalence.
\end{thm}

For a general \'etale stack, the theorem is as follows:

\begin{thm}
For an \'etale stack $\X$, a local homeomorphism $f:\g \to \X$ is equivalent to the \'etal\'e realization of a small gerbe over $\X$ if and only if
\begin{itemize}
\item[i)] $\Ef\left(f\right):\Ef\left(\g\right) \to \Ef\left(\X\right) \simeq \X$ is an equivalence, and
\item[ii)] for every space $T$, the induced functor $\g\left(T\right) \to \X\left(T\right)$ is full.
\end{itemize}
\end{thm}
We also prove in this section that the \'etal\'e realization of any small gerbe over an \'etale differentiable stack is, in particular, a differentiable gerbe in the sense of \cite{diffg}.

In section \ref{sec:groth}, we introduce the $2$-category of gerbed effective \'etale stacks. The objects of this $2$-category are effective \'etale stacks equipped with a small gerbe. We then show that when restricting to open \'etale invariant maps, this $2$-category is equivalent to \'etale stacks. In particular, we prove:

\begin{cor}
There is an equivalence of $2$-categories between gerbed effective \'etale differentiable stacks and submersions, $\gets_{subm}$, and the $2$-category of \'etale differentiable stacks and submersions, $\Eft_{subm}.$
\end{cor}

\vspace{0.2in}\noindent{\bf Acknowledgment:} I would like to thank Andr\'e Henriques for his help and encouragement during the beginning of this project, as well as his patience and useful comments concerning the proof-reading of the first draft of this article. I would also like to thank my Ph.D. advisor, Ieke Moerdijk, both for his guidance and for useful mathematical discussion. I am also grateful to Urs Schreiber for taking the time to brainstorm with me about how to prove a technical lemma. Furthermore, I extend thanks to Dorette Pronk for her fruitful e-mail correspondence. Finally, I would like to thank Philip Hackney for his useful comments about the introduction of this paper. The first draft of this article (under the title \emph{Shall Sheaves Stacks and Gerbes over \'Etale Topological and Differentiable Stacks}) was written during my studies as a Ph.D. student at Utrecht University. I am currently (as of the date of this article) a postdoc at the Max Planck Institute for Mathematics, in Bonn, Germany.

 \newpage

\section{Small Sheaves and Stacks over \'Etale Stacks}\label{sec:small}

\subsection{Conventions and notations concerning stacks}
Throughout this article, $S$ shall denote a fixed category whose objects we shall call ``spaces''. The category $S$ shall always be assumed to be either (sober) topological spaces, smooth manifolds, or locales unless otherwise noted. The machinery developed should apply to a larger class of spaces, but a more systematic treatment will be given in a subsequent paper. We will employ a minimalist definition of smooth manifold in that manifolds will neither be assumed paracompact nor Hausdorff. This is done in order to consider the \'etal\'e space (espace \'etal\'e) of a sheaf over a manifold as a manifold itself. We will also sometimes argue point-set theoretically about objects of $S,$ however, in all such cases, the arguments can be extended to locales in the usual way. In this article, the term (local) homeomorphism will mean (local) \emph{diffeo}morphism if $S$ is the category of manifolds. Similarly, for terms such as \emph{continuous}.

\begin{dfn}
A \textbf{topological groupoid} is a groupoid object in $\mathbb{TOP}$, the category of topological spaces. Explicitly, it is a diagram

$$\xymatrix{ {\G_1 \times _{\G_0}\G_1} \ar[r]^(0.6){m} & \G_1 \ar@<+.7ex>^s[r]\ar@<-.7ex>_t[r] \ar@(ur,ul)[]_{\hat i}  & \G_0 \ar@/^1.65pc/^{\mathbb{1}} [l] }$$
of topological spaces and continuous maps satisfying the usual axioms. Forgetting the topological structure (i.e. applying the forgetful functor from $\mathbb{TOP}$ to $\Set$), one obtains an ordinary (small) groupoid. Throughout this article, we shall denote the source and target maps of a groupoid by $s$ and $t$ respectively. Similarly, a $\textbf{localic groupoid}$ is a groupoid object in locales.

A \textbf{Lie groupoid} is a groupoid object in smooth manifolds such that the source and target maps are submersions.

Topological groupoids, localic groupoids, and Lie groupoids form $2$-categories with continuous functors as $1$-morphisms and continuous natural transformations as $2$-morphisms, respectively. (Recall that when $S$ is smooth manifolds, by continuous, we mean smooth.) We will denote any of these $2$-category by $S-Gpd$, and call an object thereof an $S$-groupoid.

\end{dfn}

\begin{rmk}
Traditionally, Lie groupoids are required to have a Hausdorff object space, however, as every manifold is locally Hausdorff, any Lie groupoid in the sense we defined is Morita equivalent to one that meets this requirement. (See Definition \ref{dfn:Morita}.)
\end{rmk}

Consider the $2$-category $Gpd^{S^{op}}$\ of weak presheaves in groupoids over $S$, that is contravariant (possibly weak) $2$-functors from the category $S$ into the 2-category of (essentially small) groupoids $Gpd$\footnote{Technically speaking, when $S$ is topological spaces or locales, we should restrict ourselves to a Grothendieck universe of such spaces. If $S$ is smooth manifolds, we may avoid this by replacing $\Sts$ with stacks on Cartesian manifolds, i.e., manifolds of the form $\mathbb{R}^n$, which forms a small site.}.

We recall the $2$-Yoneda Lemma:

\begin{lem}\cite{FGA}
If $C$ is an object of a category $\C$ and $\X$ a weak presheaf in $Gpd^{\C^{op}}$, then there is a natural equivalence of groupoids
$$\Hom_{Gpd^{\C^{op}}}\left(C,\X\right) \simeq \X\left(C\right).$$
\end{lem}
If $G$ is a topological group or a Lie group, then a standard example of a weak presheaf is the functor that assigns to each space the category of principal $G$-bundles over that space (this category is a groupoid). More generally, let $\G$ be an $S$-groupoid. Then $\G$ determines a weak presheaf on $S$ by the rule

\begin{equation*}
X \mapsto \Hom_{S-Gpd}\left(\left(X\right)^{(id)},\G\right),
\end{equation*}
where $\left(X\right)^{(id)}$ is the $S$-groupoid whose object space is $X$ and has only identity morphisms. This defines an extended Yoneda $2$-functor $\tilde y:S-Gpd \to Gpd^{S^{op}}$ and we have the obvious commutative diagram

$$\xymatrix{S  \ar[d]_{\left(\mspace{2mu} \cdot \mspace{2mu}\right)^{(id)}} \ar[r]^{y} & \Set^{S^{op}} \ar^{\left(\mspace{2mu} \cdot \mspace{2mu}\right)^{(id)}}[d]\\
S-Gpd \ar_{\tilde y}[r] & Gpd^{S^{op}}.}$$
We denote by $\left[\G\right]$ the associated stack on $S$, $a \circ \tilde y\left(\G\right)$, where $a$ is the stackification $2$-functor (with respect to the Grothendieck topology generated by open covers). $\left[\G\right]$ is called the \textbf{stack completion} of the groupoid $\G$.

\begin{rmk}
There is a notion of principal bundle for topological groupoids and Lie groupoids, and $\left[\G\right]$ is in fact the functor that assigns to each space the category of principal $\G$-bundles over that space.
\end{rmk}

\begin{dfn}
A stack $\X$ on $\mathbb{TOP}$ is a \textbf{topological stack} if it is equivalent to $\left[\G\right]$ for some topological groupoid $\G$. A stack $\X$ on $\mathit{Mfd}$, the category of smooth manifolds, is a \textbf{differentiable stack} if it is equivalent to $\left[\G\right]$ for some Lie groupoid $\G$. Similarly, one can define a \textbf{localic stack.}
\end{dfn}

\begin{dfn}
An $S$-groupoid $\G$ is \textbf{\'etale} if its source-map $s$ (and therefore also its target map $t$) is a local homeomorphism.
\end{dfn}

\begin{dfn}
A topological, differentiable, or localic stack $\X$ is \textbf{\'etale} if it is equivalent to $\left[\G\right]$ for some \'etale $S$-groupoid $\G$.
\end{dfn}

\begin{dfn}
A morphism $f:\Y \to \X$ of stacks is said to be \textbf{representable} if for any map from a space $T \to \X$, the weak $2$-pullback $T \times_{\X} \Y$ is (equivalent to) a space.
\end{dfn}

\begin{rmk}
A morphism $\varphi:\X \to \Y$ between stacks is an \textbf{epimorphism} if it is locally essentially surjective in the following sense:

For every space $X$ and every morphism $f:X \to \Y,$ there exists an open cover  $\mathcal{U}=\left(U_i \hookrightarrow X\right)_i$ of $X$ such that for each $i$ there exists a map $\tilde f_i:U_i \to \Y,$ such that the following diagram $2$-commutes:

$$\xymatrix@M=5pt{U_i \ar@{^(->}[d] \ar[r]^-{\tilde f_i} & \Y \ar[d]^-{\varphi}\\
X \ar[r]^-{f} & \X.}$$
In words, this just means any map $X \to \Y$ from a space $X$ locally factors through $\varphi$ up to isomorphism.
\end{rmk}

\begin{dfn}
An \textbf{atlas} for a stack $\X$ is a representable epimorphism $X \to \X$ from a space $X$.
\end{dfn}

\begin{rmk}
A stack $\X$ comes from an $S$-groupoid if and only if it has an atlas. If $X \to \X$ is an atlas, then $\X$ is equivalent to the stack-completion of the groupoid $X \times_\X X \rightrightarrows X$. Conversely, for any $S$-groupoid $\G$, the canonical morphism $\G_0 \to \left[\G\right]$ is an atlas.
\end{rmk}

\begin{dfn}\label{dfn:local}
Let $P$ be a property of a map of spaces. It is said to be \textbf{invariant under change of base} if for all $$f: Y \to X$$ with property $P$, if $$g:Z \to X$$ is any representable map, the induced map $$Z \times_X Y \to Z$$ also has property $P$. The property $P$ is said to be \textbf{invariant under restriction}, in the topological setting if this holds whenever $g$ is an embedding, and in the differentiable setting if and only if this holds whenever $g$ is an \emph{open} embedding. Note that in either case, being invariant under change of base implies being invariant under restriction. A property $P$ which is invariant under restriction is said to be \textbf{local on the target} if any $$f: Y \to X$$ for which there exists an open cover $\left(U_\alpha \to X\right)$  such that the induced map $$\coprod\limits_\alpha  {U_\alpha  }  \times_{X} Y \to \coprod\limits_\alpha  {U_\alpha  } $$ has property $P$, must also have property $P$.
\end{dfn}

Examples of such properties are being an open map, \'etale map, proper map, closed map etc.

\begin{prop}
A stack $\X$ over $S$ is \'etale if and only if it admits an \'etale atlas $p:X \to \X$, that is a representable epimorphism which is also \'etale.
\end{prop}

\begin{proof}
This follows from the fact that if $\G$ is any $S$-groupoid, the following diagram is $2$-Cartesian:
$$\xymatrix{\G_1 \ar[d]_{s} \ar[r]^{t} & \G_0 \ar[d]\\
\G_0 \ar[r] & \left[\G\right],}$$
where the map $\G_0 \to \left[\G\right]$ is induced from the canonical map $\G_0 \to \G$.
\end{proof}

\begin{rmk} Traditionally speaking, a \textbf{differentiable stack} is a stack $\X$ equivalent to $\left[\G\right]$ where $\G$ is a Lie groupoid. This is equivalent to it having an atlas which is a representable submersion.
\end{rmk}

\begin{dfn}\label{dfn:Morita}
An internal functor $\varphi:\h \to \G$ of $S$-groupoids is a \textbf{Morita equivalence} if the following two properties hold:

\begin{itemize}
\item[i)] (Essentially Surjective)\\
The map $t \circ pr_1:\G_1 \times_{\G_0} \h_0 \to \G_0$ admits local sections, where $\G_1 \times_{G_0} \h_0$ is the fibred product

$$\xymatrix{\G_1 \times_{G_0} \h_0 \ar[r]^-{pr_2} \ar[d]_-{pr_1} & \h_0 \ar[d]^-{\varphi} \\
\G_1 \ar[r]^-{s} & \G_0.}$$
\item[i)] (Fully Faithful)
The following is a fibered product:

$$\xymatrix{\h_1 \ar[r]^-{\varphi} \ar[d]_-{\left(s,t\right)} &\G_1 \ar[d]^-{\left(s,t\right)}  \\
\h_0 \times \h_0 \ar[r]^-{\varphi \times \varphi} & \G_0 \times \G_0.}$$
\end{itemize}
Two $S$-groupoids $\Ll$ and $\K$ are \textbf{Morita equivalent} if there is a chain of Morita equivalences $\Ll \leftarrow \h \rightarrow \K$. Moreover, $\Ll$ and $\K$ are Morita equivalent if and only if $\left[\Ll\right] \simeq \left[\K\right]$
\end{dfn}

Every internal functor $\h \to \G$ induces a map $\left[\h\right] \to \left[\G\right]$ and the induced functor $$\Hom\left(\h,\G\right) \to \Hom\left(\left[\h\right],\left[\G\right]\right)$$ is full and faithful, but not essentially surjective. However, any morphism $$\left[\h\right] \to \left[\G\right]$$ arises from a chain $$\h \leftarrow \K \rightarrow \G,$$ with $\K \to \h$ a Morita equivalence. In fact, the class of Morita equivalences admits a calculus of fractions and topological (differentiable) stacks are equivalent to the bicategory of fractions of $S$-groupoids with inverted Morita equivalences. For details see \cite{Dorette}.

\begin{dfn}
By an \textbf{\'etale cover} of a space $X$, we mean a surjective local homeomorphism $U \to X$. In particular, for any open cover $\left(U_\alpha\right)$ of $X$, the canonical projection $$\coprod\limits_{\alpha} U_\alpha \to X$$ is an \'etale cover.
\end{dfn}

\begin{dfn}\label{dfn:cech}
Let $\h$ be an $S$-groupoid. If $\mathcal{U}=U \to \h_0$ is an \'etale cover of $\h_0$, then one can define the \textbf{\v{C}ech-groupoid} $\h_{\mathcal{U}}$. Its objects are $U$ and the arrows fit in the pullback diagram

$$\xymatrix{\left(\h_{\mathcal{U}}\right)_1 \ar[r] \ar[d]_-{\left(s,t\right)} & \h_1 \ar[d]^-{\left(s,t\right)} \\
U \times U \ar[r] & \h_0 \times \h_0,}$$
and the groupoid structure is induced from $\h$. There is a canonical map $\h_{\mathcal{U}} \to \h$ which is a Morita equivalence. Moreover,
\end{dfn}
\begin{equation}\label{eq:pizza}
\Hom\left(\left[\h\right],\left[\G\right]\right) \simeq \underset{\mathcal{U} \in Cov\left(\h_0\right)} \hc \Hom_{S-Gpd} \left(\h_{\mathcal{U}},\G\right),
\end{equation}
where the weak 2-colimit above is taken over a suitable 2-category of \'etale covers. For details see \cite{Andre}.

\begin{rmk}
We could restrict to open covers, and a similar statement would be true. However, it will become convenient to work with \'etale covers later.
\end{rmk}

Applying equation (\ref{eq:pizza}) to the case where $\left[\h\right]$ is a space $X$, by the Yoneda Lemma we have

\begin{equation*}
\left[\G\right] \left(X\right) \simeq \underset{\mathcal{U} \in Cov\left(X\right)} \hc \Hom_{S-Gpd} \left(X_{\mathcal{U}},\G\right).
\end{equation*}

\begin{dfn}
Let $\C$ be a $2$-category, and $C$ an object. The \textbf{slice $2$-category} $\C/C$ has as \textbf{objects} morphisms $\varphi:D \to C$ in $\C$. The \textbf{morphisms} are $2$-commutative triangles of the form
$$\xymatrix@R=0.6cm@C=0.6cm{D \ar[rrdd]_-{\varphi}
 \ar[rrrr]^{f} \ar@{} @<-5pt>  [rrrr]| (0.6) {}="a"
 							& &  & & E \ar[lldd]^-{\psi}\\
						&\ar@{} @<5pt>  [rd]| (0.4) {}="b"
\ar @{=>}^{\alpha}  "b";"a"&	& & \\
						& & C, & &}$$
with $\alpha$ invertible. A \textbf{$2$-morphism} between a pair of morphisms  $\left(f,\alpha\right)$ and $\left(g,\beta\right)$ going between $\varphi$ and $\psi$ is a $2$-morphism in $\C$ $$\omega:f \Rightarrow g$$ such that the following diagram commutes:
$$\xymatrix@R=0.6cm@C=0.6cm{\psi f \ar@{=>}[rr]^-{\psi\omega}& & \psi g\\ & \varphi \ar@{=>}[lu]^-{\alpha} \ar@{=>}[ru]_-{\beta}.}$$
\end{dfn}

We end by a standard fact which we will find useful later:

\begin{prop}\label{prop:stand}
For any stack $\X$ on $S$, there is a canonical equivalence of $2$-categories $\St\left(S/\X\right)\simeq \St\left(S\right)/\X$.
\end{prop}

The construction is as follows:

Given $\Y \to \X$ in $\St\left(S\right)/\X$, consider the stack $$\tilde \Y(T \to \X):=\Hom_{\St\left(S\right)/\X}\left(T \to \X,\Y \to \X\right).$$ Given a stack $\W$ in $\St\left(S/\X\right)$, consider it as a fibered category $\int{\W} \to S/\X$. Then since $S/\X \simeq \int{\X}$ (as categories), the composition $\int{\W} \to \int{\X} \to S$ is a category fibered in groupoids presenting a stack $\tilde W$ over $S$, and since the diagram

$$\xymatrix{\int{\W} \ar[rd] \ar[d] & \\
\int{\X} \ar[r] & S}$$
commutes,  $\int{\W} \to \int{\X}$ corresponds to a map of stacks $\tilde \W \to \X$.

We leave the rest to the reader.

\subsection{Grothendieck topoi}\label{sec:topoi}
A concise definition of a Grothendieck topos is as follows:

\begin{dfn}
A category $\E$ is a Grothendieck topos if it is a reflective subcategory of a presheaf category $Set^{\C^{op}}$ for some small category $\C$,

\begin{equation}\label{eq:topos}
\xymatrix@C=1.5cm{\E \ar@<-0.5ex>[r]_{j_*} & Set^{\C^{op}} \ar@<-0.5ex>[l]_{j^*}},
\end{equation}
with $j^* \rt \mspace{2mu} j_*$, such that the left-adjoint $j^*$ preserves finite limits. From here on in, topos will mean Grothendieck topos.
\end{dfn}

\begin{rmk}
It is standard that this definition is equivalent to saying that $\E$ is equivalent to $\Sh_J\left(\C\right)$ for some Grothendieck topology $J$ on $\C$, see for example \cite{sheaves}.
\end{rmk}

\begin{dfn}
A \textbf{geometric morphism} from a topos $\E$ to a topos $\F$ is a an adjoint-pair

$$\xymatrix@C=1.5cm{\E \ar@<-0.5ex>[r]_{f_*} & \F \ar@<-0.5ex>[l]_{f^*}},$$
with $f^* \rt \mspace{1mu} f_*$, such that $f^*$ preserve finite limits. The functor $f_*$ is called the \textbf{direct image} functor, whereas the functor $f^*$ is called the \textbf{inverse image} functor.
\end{dfn}

In particular, this implies, somewhat circularly, that equation (\ref{eq:topos}) is an example of a geometric morphism.

Topoi form a $2$-category. Their arrows are geometric morphisms. If $f$ and $g$ are geometric morphisms from $\E$ to $\F$, a $2$-cell $$\alpha:f \Rightarrow g$$ is given by a natural transformation $$\alpha:f^* \Rightarrow g^*.$$ In this paper, we will simply ignore all non-invertible $2$-cells to arrive at a $(2,1)$-category of topoi, $\Top$.

\subsection{Locales and frames}
Our conventions on locales and frames closely follow \cite{loc1}. Recall that the category of \textbf{frames} has as objects complete Heyting algebras, which are complete lattices of a certain kind, and morphism are given by functions which preserve finite meets and arbitrary joins. The category of \textbf{locales} is dual to that of frames. Locales are generalized spaces and find their home in the domain of so-called pointless topology. See for example \cite{pointless}.

\begin{dfn}\label{dfn:open}
Given a topological space $X$, we denote its poset of open subsets by $\mathcal{O}\left(X\right)$. The poset $\mathcal{O}\left(X\right)$ together with intersection and union, forms a complete Heyting algebra, hence a locale.
\end{dfn}

Notice that a continuous map $f:X \to Y$ induces a map

\begin{eqnarray*}
\mathcal{O}\left(Y\right) &\to& \mathcal{O}\left(X\right)\\
U &\mapsto& f^{-1}\left(U\right).
\end{eqnarray*}
It is easy to see that this is a map of frames, hence, is a map $$\mathcal{O}\left(f\right):\mathcal{O}\left(X\right) \to \mathcal{O}\left(Y\right)$$ in the category of locales. This makes $\mathcal{O}$ into a functor $$\mathcal{O}:\mathbb{TOP} \to Locales.$$ In fact, this functor has a right-adjoint $$pt:Locales \to Top.$$ The adjoint-pair $\mathcal{O} \rt \mspace{2mu} pt$ restricts to an equivalence between sober topological spaces, and locales with enough points (both ``sober'' and ``with enough points'' have a precise mathematical meaning). This result is known as Stone duality. The class of sober spaces is quite large in practice. It includes many highly non-Hausdorff topological spaces such as the prime spectrum with the Zariski topology, $Spec\left(A\right)$, for a commutative ring $A$.

Note that the open-cover Grothendieck topology on topological spaces naturally extends to locales. We make the following definition:


\subsection{Small sheaves as a Kan-extension}

Let $\Top$ denote the bicategory of Grothendieck topoi, geometric morphisms, and \emph{invertible} natural transformations, as in \ref{sec:topoi}. There is a canonical functor $$S \to \Top,$$ which assigns each space $X$ its topos of sheaves $\Sh(X)$. By (weak) left-Kan extension, we obtain a 2-adjoint pair $\Sh \rt \S$
$$\xymatrix@C=1.5cm{Gpd^{S^{op}} \ar@<-0.5ex>[r]_{\Sh} & \Top \ar@<-0.5ex>[l]_{\S},}$$
where $Gpd^{S^{op}}$ denotes the bicategory of weak presheaves in groupoids. In fact, the essential image of $\S$ lies entirely within the bicategory of stacks over $S$, $\Sts$, where $S$ is equipped with the standard ``open cover'' Grothendieck topology \cite{bunge}. So, by restriction, we obtain an adjoint pair

\begin{equation}\label{eq:sheaves}
\xymatrix@C=1.5cm{\St(S) \ar@<-0.5ex>[r]_{\Sh} & \Top \ar@<-0.5ex>[l]_{\S}.}
\end{equation}

\begin{dfn}

For $\X$ a stack over $S$, we define the topos of \textbf{small sheaves} over $\X$ to be the topos $\Sh(\X)$.

\end{dfn}

\begin{rmk}
Suppose that $\X \simeq \left[\G\right]$ for some $S$-groupoid $\G$. Then we may consider the nerve $N\left(\G\right)$ as a simplicial $S$-object

$$\xymatrix{ \G_0& \G_1  \ar@<+.7ex>[l] \ar@<-.7ex>[l]& {\G_2 \cdots}  \ar@<0.9ex>[l] \ar@<0.0ex>[l] \ar@<-0.9ex>[l]}.$$

By composition with the Yoneda embedding, we obtain a simplicial stack $$y \circ N\left(\G\right):\Delta^{op} \to \St(S).$$ The weak colimit of this diagram is the stack $\left[\G\right]$. Since $\Sh$ is a left-adjoint, it follows that $\Sh\left(\left[\G\right]\right)$ is the weak colimit of the simplicial-topos

$$\xymatrix{ \Sh\left(\G_0\right)& \Sh\left(\G_1\right)  \ar@<+.7ex>[l] \ar@<-.7ex>[l]& {\Sh\left(\G_2\right) \cdots}  \ar@<0.9ex>[l] \ar@<0.0ex>[l] \ar@<-0.9ex>[l]}.$$

From \cite{cont}, it follows that $\Sh\left(\left[\G\right]\right) \simeq \B\G$, the classifying topos of $\G$. We will return to a more concrete description of the classifying topos later.
\end{rmk}

For the rest of this subsection, we will assume that $S$ is sober topological spaces, or locales, unless otherwise noted.

The adjoint pair $\Sh \rt \S$ restricts to an equivalence between, on one hand, the subbicategory of $\St(S)$ on which the unit is an equivalence, and, on the other hand, the subbicategory of $\Top$ on which the co-unit is an equivalence.

\begin{prop}\label{prop:unit}
If $\X$ is an \'etale stack, then the unit is an equivalence.
\end{prop}

\begin{proof}
Let $T$ be a space, then $$\S(\Sh(\X))(T)=\Hom(\Sh(T),\Sh(\X)),$$ and the latter is the groupoid of geometric morphisms from $\Sh(T)$ to $\B\G$, where $\G$ is some groupoid representing $\X$. From, \cite{cont} this in turn is equivalent to $\X(T)$.
\end{proof}

Let $\Et$ denote the full subbicategory of $\St(S)$ consisting of \'etale stacks. Then, since the unit restricted to $\Et$ is an equivalence, $Sh$ restricted to $\Et$ is 2-categorically fully faithful. We now identify its essential image.

\begin{dfn}
A topos $\E$ is an \textbf{\'etendue} if there exists a well-supported object $E \in \E$ (i.e. $E \to 1$ is an epimorphism) such that the slice topos $\E/E$ is equivalent to $\Sh(X)$ for some space $X$.
\end{dfn}

\begin{thm}
A topos $\E$ is an \'etendue if and only if $\E \simeq \B\G$ for some \'etale groupoid $\G$ \cite{sga4}.
\end{thm}

\begin{cor}\label{cor:etendue}
$\Sh$ induces an equivalence between the bicategory of \'etale stacks and the bicategory of \'etendues.
\end{cor}

\begin{rmk}
This result was original proven in \cite{Dorette}.
\end{rmk}

This corollary should be interpreted as evidence that for \'etale groupoids $\G$, $\Sh\left(\left[\G\right]\right) = \B\G$ is the correct notion for the topos of sheaves over $\left[\G\right]$ since just as for spaces, morphisms between \'etale stacks are the same as geometric morphisms between their topoi of sheaves.

\begin{cor}
Let $\X \simeq \left[\G\right]$ and $\Y \simeq \left[\h\right]$ be two stacks with $\Y$ \'etale. Then $$\Hom\left(\X,\Y\right) \simeq \Hom\left(\B\G,\B\h\right).$$
\end{cor}

The adjoint pair $\Sh \rt \S$ allows us also to prove another interesting result, which we shall now do, for completeness.

\begin{dfn}
An $S$-groupoid $\G$ is \textbf{\'etale-complete} if the diagram

$$\xymatrix@=16pt@M=6pt{ \Sh\left(\G_1\right) \ar[r]^{t}  \ar@{} @<-4pt> [r]| (0.4){}="a"
               \ar[d]_{s} \ar@{} @<4pt> [d]| (0.4){}="b"
                                   & \Sh\left(\G_0\right) \ar[d]^{p}     \ar @{=>}^{\mu}  "a";"b" \\
                                        \Sh\left(\G_0\right) \ar[r]_{p}        & \B\G           }$$
is a (weak) pullback-diagram of topoi, where $\B\G$ is the classifying topos of $\G$, $p$ is induced from the inclusion $\G_0 \to \G$, and $\mu$ is induced by the obvious action of $\G$ on sheaves over $\G_0$. For details, see \cite{cont}.

A stack $\X$ over $S$ is \textbf{\'etale-complete} if it is equivalent to $\left[\G\right]$ for some \'etale-complete $\G$.
\end{dfn}

\begin{rmk}
Every \'etale-groupoid is \'etale-complete \cite{cont}.
\end{rmk}

\begin{rmk}
Proposition \ref{prop:unit} and its proof remains valid if \'etale is replaced with \'etale-complete.
\end{rmk}

Let $\mathfrak{EtC}$ denote the full subbicategory consisting of \'etale-complete stacks. $\Sh$ restricted to $\mathfrak{EtC}$ is also 2-categorically fully faithful. For $S$ sober-topological spaces, to the author's knowledge, there is no nice description for the essential image. However, for $S$ locales, the answer is quite nice indeed:

\begin{thm}
For $S$ locales, $\Sh$ induces an equivalence between the bicategory of \'etale-complete stacks and the bicategory $\Top$ of topoi. In particular, $$\S:\Top \to St\left(S\right)$$ exhibits Grothendieck topoi as a reflective full subbicategory of stacks on locales.
\end{thm}

\begin{proof}
It suffices to show that $\Sh$ is essentially surjective. Every topos is equivalent to $\B\G$ for some localic groupoid $\G$ \cite{loc1}, and hence to $\Sh\left(\X\right)$ for some localic stack $\X$ over locales. The result now follows from the fact that every localic groupoid $\G$ has an \'etale-completion $\hat \G$ such that $\B\G \simeq \B \hat \G$ \cite{cont}.
\end{proof}

\begin{cor}
The adjunction $\left(\ref{eq:sheaves}\right)$ restricts to a adjunction
$$\xymatrix@C=1.5cm{\Top \ar@<-0.5ex>[r]_{\S} & \mathfrak{LocSt} \ar@<-0.5ex>[l]_{\Sh},}$$
exhibiting the $2$-category of topoi as a reflective subbicategory of localic stacks.
\end{cor}

\begin{rmk}
In light of the fact that every topos $\E$ with enough points is equivalent to $\B\G$ for some topological groupoid $\G$ \cite{IekeButz}, one may be tempted to claim that \'etale-complete topological stacks are equivalent to topoi with enough points. However, the proof just given does not work for the topological case as a topological groupoid's \'etale-completion may not be a topological groupoid, but only a groupoid object in locales.
\end{rmk}

\begin{rmk}
Most of what has been done in this subsection caries over for smooth manifolds if we use ringed-topoi rather than just topoi. In particular, the result of Pronk that \'etale differentiable stacks and smooth-\'etendue are equivalent can be proven along these lines.
\end{rmk}

\subsection{The classifying topos of a groupoid}

\begin{dfn}
Given an $S$-groupoid $\h$, a (left) \textbf{$\h$-space} is a space $E$ equipped with a \textbf{moment map} $\mu:E \to \h_0$ and an \textbf{action map} $$\rho:\h_1 \times_{\h_0} E \to E,$$
where
$$\xymatrix{\h_1 \times_{\h_0} E \ar[r]  \ar[d] & E \ar^-{\mu}[d] \\
\h_1 \ar^-{s}[r] & \h_0\\}$$
is the fibred product, such that the following conditions hold:

\begin{itemize}
\item[i)] $\left(gh\right) \cdot e = g \cdot \left(h \cdot e\right)$ whenever $e$ is an element of $E$ and $g$ and $h$ elements of $\h_1$ with domains such that the composition makes sense,
\item[ii)] $\mathbb{1}_{\mu\left(e\right)} \cdot e =e$ for all $e \in E,$ and
\item[iii)] $\mu\left(g \cdot e\right) = t \left(g\right)$ for all $g \in \h_1$ and $e \in E.$
\end{itemize}
A map of $\h$-spaces is simply an equivariant map, i.e., a map $$(E,\mu,\rho) \to (E',\mu',\rho')$$ is map $f:(E,\mu,) \to (E',\mu')$ in $S/\h_0$ such that

$$f(he)=hf(e)$$
whenever this equation makes sense.
\end{dfn}

\begin{rmk}
This definition extends for localic groupoids in the obvious (diagrammatic) way.
\end{rmk}

\begin{dfn}
An $\h$-space $E$ is an $\h$-\textbf{equivariant sheaf} if the moment map $\mu$ is a local homeomorphism. The category of $\h$-equivariant sheaves and equivariant maps forms the \textbf{classifying topos} $\B\h$ of $\h$.
\end{dfn}

\subsection{The small-site of an \'etale stack}

\begin{dfn}\label{dfn:site}
Let $\h$ be an \'etale $S$-groupoid. Let $\sit\left(\h\right)$ be the following category: The objects are the open subsets of $\h_0$. An arrow $U \to V$ is a section $\sigma$ of the source-map $s:\h_1 \to \h_0$ over $U$ such that $t \circ \sigma: U \to V$ as a map in $S$. Composition is by the formula $\tau \circ \sigma(x):=\tau\left(t\left(\sigma(x\right)\right).$

There is a canonical functor $i:\mathcal{O}\left(\h_0\right) \hookrightarrow \sit\left(\h\right)$ which sends an inclusion $U \hookrightarrow V$ in $\mathcal{O}\left(\h_0\right)$ to $\mathbb{1}|_U$, where $\mathbb{1}$ is the unit map of the groupoid, and $\mathcal{O}$ is as in Definition \ref{dfn:open}.

This functor induces a Grothendieck pre-topology on $\sit\left(\h\right)$ by declaring covering families to be images under $i$ of covering families of $\mathcal{O}\left(\h_0\right)$. The Grothendieck site $\sit\left(\h\right)$ equipped with the induced topology is called the \textbf{small site} of the groupoid $\h$.
\end{dfn}

\begin{rmk}
Given an \'etale stack $\X$ with an \'etale atlas $X \to \X$, we can describe $\sit\left(X \times_\X X \rightrightarrows X\right)$ in terms of this stack and atlas. Denote the $S$-groupoid $$X \times_\X X \rightrightarrows X$$ by $\h$. Let $\sit\left(\X,X\right)$ denote the following category. The objects of are open subsets of $X=\h_0$ and the arrows are pairs $\left(f,\alpha\right)$, such that

$$\xymatrix@R=0.4cm@C=0.4cm{U \ar@{^{(}->}[rd]
 \ar@<+0.5ex>[rrrr]^{f} \ar@{} @<-5pt>  [rrrr]| (0.6) {}="a"
 							& &  & & V \ar@{^{(}->}[ld]\\
						&X \ar[rd] \ar@{} @<5pt>  [rd]| (0.4) {}="b"
\ar @{=>}^{\alpha}  "b";"a"&	& X \ar[ld]& \\
						& & \X & &}.$$
In other words, it is the full subcategory of $\St\left(S\right)/\X \simeq \St\left(S/\X\right)$ (Proposition \ref{prop:stand}) spanned by objects of the form $U \hookrightarrow X \to \X$, with $U  \subseteq X$ open. We claim that this category is canonically equivalent to $\sit\left(\h\right).$
To see this, suppose $\sigma$ is a section of $s$ over $U$ whose image lies in $t^{-1}\left(V\right)$. We can associate to it the map
\begin{eqnarray*}
\alpha\left(\sigma\right):U &\to& \h_1\\
x &\mapsto& \sigma\left(x\right).\\
\end{eqnarray*}
Then, letting $$f:=t\circ \sigma:U \to V,$$ $\alpha\left(\sigma\right):U \to \h_1$ is a continuous natural transformation
$$\xymatrix@R=0.4cm@C=0.4cm{U \ar@{^{(}->}[rd]
 \ar@<+0.5ex>[rrrr]^{f} \ar@{} @<-5pt>  [rrrr]| (0.6) {}="a"
 							& &  & & V \ar@{^{(}->}[ld]\\
						&\h_0 \ar[rd] \ar@{} @<5pt>  [rd]| (0.4) {}="b"
\ar @{=>}^{\alpha\left(\sigma\right)}  "b";"a"&	& \h_0 \ar[ld]& \\
						& & \h. & &}$$
Applying stack-completion, one arrives at an arrow in $\sit\left(\X,X\right).$ Conversely, if one has a diagram of the form, $$\xymatrix@R=0.4cm@C=0.4cm{U \ar@{^{(}->}[rd]
 \ar@<+0.5ex>[rrrr]^{f} \ar@{} @<-5pt>  [rrrr]| (0.6) {}="a"
 							& &  & & V \ar@{^{(}->}[ld]\\
						&X \ar[rd] \ar@{} @<5pt>  [rd]| (0.4) {}="b"
\ar @{=>}^{\alpha}  "b";"a"&	& X \ar[ld]& \\
						& & \X & &},$$
then since the canonical map $$\tilde y\left(\h\right) \to \left[\h\right]$$ is object-wise full and faithful, this must correspond to continuous natural transformation as in the previous diagram. But such a natural transformation, by definition, is a continuous map $$\alpha:U \to \h_1$$ such that $$s\circ\alpha=id_U$$ and $$t\circ\alpha=f.$$ Spelling this out, one arrives at an equivalence of categories.
\end{rmk}

\begin{dfn}\label{dfn:mu}
Given an object $U \subset \h_0$ of $\sit\left(\h\right)$, the space $s^{-1}\left(U\right)$ comes equipped with a canonical left $\h$-action along the target map $t$. Since the target map is a local homeomorphism, this $\h$-space is in fact an equivariant sheaf. We denote it by $m_U$.
\end{dfn}

Extend this to a functor as follows:

Given $\sigma:U \to V$ in $\sit\left(\h\right)$, define a map $$f:s^{-1}\left(U\right) \to s^{-1}\left(V\right)$$ by sending $$x\stackrel{h}{\longrightarrow}y$$ to $$t\left(\sigma(x\right))\stackrel{\sigma(x)^{-1}}{-\!\!\!-\!\!\!-\!\!\!-\!\!\!-\!\!\!\longrightarrow} x\stackrel{h}{\longrightarrow}y,$$ which is clearly $\h$-equivariant. Conversely, given and $\h$-equivariant map $$f:s^{-1}\left(U\right) \to s^{-1}\left(V\right),$$ let $\sigma:=\hat i \circ F \circ \mathbb{1}|_U,$ where $\hat i$ denotes the morphism $$\h_1 \to \h_1$$ which sends an arrow to its inverse. The map $\sigma$ is an object of $\sit\left(\h\right)$ and it is easy to check that this defines a natural bijection

$$\Hom_{\sit\left(\h\right)}\left(U,V\right) \cong \Hom_{\B\h}\left(m_U,m_V\right).$$
Hence we get a full and faithful functor $m:\sit\left(\h\right) \to \B\h$.

\begin{prop}
The left Kan extension of $m$ along the Yoneda embedding $$y:\sit\left(\h\right) \to \Sh(\left(\sit\left(\h\right)\right)$$ is an equivalence between the topos of sheaves for the Grothendieck site $\sit\left(\h\right),$ and the classifying topos $\B\h$ \cite{pres}.
\end{prop}

\begin{dfn}
By a \textbf{small stack} over an \'etale stack $\X\simeq\left[\h\right]$, we mean a stack $\Z$ over $\sit\left(\h\right)$. We denote the 2-category of small stacks over $\X$ by $\St\left(\X\right)$.
\end{dfn}

\begin{rmk}
This definition does not depend on the choice of presenting groupoid since, if $\G$ is another groupoid such that $\left[\G\right]\simeq \X$, then $$\Sh\left(\sit\left(\G\right)\right) \simeq \B\G \simeq \B\h \simeq \Sh\left(\sit\left(\h\right)\right)$$ and hence $\St\left(\sit\left(\G\right)\right) \simeq \St\left(\sit\left(\h\right)\right)$ by the Comparison Lemma for stacks \cite{sga4}. A more intrinsic equivalent definition is that a small stack over $\X$ is a stack over the topos $\Sh\left(\X\right)$ in the sense of Giraud in \cite{Giraud}, that is a stack over $\Sh\left(\X\right)$ with respect to the canonical Grothendieck topology, which in this case is generated by jointly epimorphic families. Even better, since we are dealing with \'etale stacks, in light of Corollary \ref{cor:etendue}, we may instead work with the bicategory of \'etendues. Then, a small stack over an etendue $\E$ is precisely a stack over $\E,$ with its canonical site.
\end{rmk}

\subsection{The \'etal\'e realization of a small stack}

Recall that for a sheaf $F$ over a space $X$, the \'etal\'e space (espace \'etal\'e) is a space $E \to X$ over $X$ via a local homeomorphism (\'etale map), such that the sheaf of sections of $E \to X$ is isomorphic to $F$. In fact, the \'etal\'e space can be constructed for any presheaf, and the corresponding sheaf of sections is isomorphic to its sheafification. As a set, $E$ is the disjoint union of the stalks of $F$ and the topology is induced by local sections.

Abstractly, this construction may be carried out as follows:

Consider the category of open subsets of $X$, $\mathcal{O}\left(X\right)$, where the arrows are inclusions, as in Definition \ref{dfn:open}. This category, equipped with its natural Grothendieck topology, is of course the site over which ``sheaves over $X$'' are sheaves. There is a canonical functor $j:\mathcal{O}\left(X\right) \to S/X$ which sends an open $U \subseteq X$ to $U \hookrightarrow X$. Hence, there is an induced adjunction

$$\xymatrix{\Set^{\mathcal{O}\left(X\right)^{op}} \ar@<-0.5ex>[r]_-{L} & S/X \ar@<-0.5ex>[l]_-{\Gamma}}.$$
Here, $L$ takes a presheaf to its \'etal\'e space and $\Gamma$ takes a space $T \to X$ over $X$ to its sheaf of sections. The composite $\Gamma \circ L$ is isomorphic to the sheafification functor $a:\Set^{\mathcal{O}\left(X\right)} \to \Sh\left(X\right)$, and the image of $L$ lies completely in the subcategory $Et\left(X\right)$ of $S/X$ spanned by spaces over $X$ via a local homeomorphism. When restricted to $\Sh\left(X\right)$ and $Et\left(X\right)$, the adjoint pair $L \rt\ \Gamma$ is an equivalence of categories

$$\xymatrix{\Sh\left(X\right) \ar@<-0.5ex>[r]_-{L} & Et\left(X\right) \ar@<-0.5ex>[l]_-{\Gamma}}.$$
This construction can be done even more topos-theoretically as follows:

The canonical functor $j:\mathcal{O}\left(X\right) \to S/X$ produces three adjoint functors $j_! \rt j^* \rt j_*$

$$\xymatrix{\Sh\left(X\right) \ar@<-0.8ex>[r] \ar@<0.8ex>[r]  & \Sh\left(S/X\right) \ar[l]},$$
where the Grothendieck topology on $S/X$ is induced from the open cover topology on $S$. For a sheaf $F$ over $X$, $j_!\left(X\right)=y\left(L\left(F\right)\right)$, where $y$ denotes the Yoneda embedding $y:S/X \hookrightarrow \Sh\left(S/X\right)$.

Hence, $$y \circ L:\Set^{\mathcal{O}\left(X\right)^{op}} \to \Sh\left(S/X\right)$$ can be identified with the left Kan extension of $$\mathcal{O}\left(X\right) \stackrel{j}{\longrightarrow} S/X \stackrel{y}{\hookrightarrow} \Sh\left(S/X\right)$$ along Yoneda.

We now turn our attention to generalizing this construction to work when both $X$ and $F$ are stacks. Let $\h$ be an \'etale groupoid and let $\X \simeq  \left[\h\right].$ In light of the remark after Definition \ref{dfn:site}, there is a canonical fully faithful functor $$j_\h:\sit\left(\h\right) \to S/\X$$ which sends $U \subseteq \h_0$ to $U \hookrightarrow \h_0 \to \X$. This produces three adjoint functors $j_! \rt j^* \rt j_*$

$$\xymatrix{Gpd^{\sit\left(\h\right)^{op}} \ar@<-0.8ex>[r] \ar@<0.8ex>[r]  & \St\left(S/\X\right) \ar[l]}.$$
We denote $j_!$ by $L$ and $j^*$ by $\Gamma$.

More explicitly, $j_!$ is the weak left Kan extension of $j_\h$ along Yoneda, and $$\Gamma(\Y)(U)=\Hom_{ \St\left(S/\X\right)}\left(y\left(U \hookrightarrow \h_0 \to \X\right),\Y\right).$$

\begin{rmk}

Under the equivalence given in Proposition \ref{prop:stand}, $\Y$ may be viewed as stack $\bar\Y$ in $\St\left(S\right)$ together with a map $$f:\bar \Y \to \X.$$ From this point of view, $\Gamma(f:\bar \Y \to \X)$ assigns an open subset $U$ of $\h_0$ the groupoid of ``sections of $f$ over $U$,'' which can be described explicitly as the groupoid whose objects are pairs $\left(\sigma,\alpha\right)$ which fit into a $2$-commutative diagram
$$\xymatrix@R=0.6cm{& & \bar \Y \ar[dd]^-{f} & &\\
& & \ar@{} @<3pt>  [r]| (-0.1) {}="a"  &\\
U \ar@{^{(}->}[r] \ar@<+.7ex>[rruu]^-{\sigma} & \h_0\ar[r] \ar@{} @<+12pt>  [r]  | (-.1) {}="b" \ar @{=>}^{\alpha}  "b";"a"& \X,&}$$
and whose morphisms $\left(\sigma,\alpha\right) \to \left(\sigma',\alpha'\right)$ are $2$-cells
$$ \xygraph{!{0;(2,0):(0,.5)::}
{U}="a" [r] {\bar \Y}="b"
"a":@/^{1.5pc}/"b"^-{\sigma}|(.4){}="l"
"a":@/_{1.5pc}/"b"_-{\sigma'}
"l" [d(.3)]  [r(0.1)] :@{=>}^{\omega} [d(.5)]} $$
such that the following diagram commutes:
$$\xymatrix{j_{\h}\left(U\right) \ar@{=>}[r]^-{\alpha} \ar@{=>}[rd]_-{\alpha'} & f \circ \sigma \ar@{=>}[d]^-{f\omega}\\
& f\circ \sigma'.}$$
\end{rmk}

\begin{dfn}
Let $\Z$ be a weak presheaf in groupoids over $\sit\left(\h\right)$. Then $L\left(\Z\right)$ is the \textbf{\'etal\'e realization} of $\Z$.
\end{dfn}

\begin{prop}
Let $\Y$ be any stack in $\St\left(S/\X\right)$. Then $\Gamma\left(\X\right)$ is a stack.
\end{prop}

\begin{proof}
This is immediate from the fact that $\Y$ satisfies descent.
\end{proof}

In fact, we can say more:

\begin{thm}
The 2-functor $\Gamma \circ L$ is equivalent to the stackification 2-functor $a:Gpd^{\sit\left(\h\right)^{op}} \to \St\left(\sit\left(\h\right)\right)\simeq \St(\X).$
\end{thm}

\begin{proof}

Suppose $\Z$ is a weak presheaf in groupoids over $\sit\left(\h\right)$. Then $$\Gamma\left(\Z\right)\left(V\right) \simeq L\left(\Z\right)\left(V \hookrightarrow \h_0 \to \X\right).$$ Let $G\left(\Z\right)$ be the weak presheaf in groupoids over $S/\X$ given by $$G\left(\Z\right) \simeq  \underset{U \to \Z} \hc y\left(U \hookrightarrow \h_0 \to \X\right).$$ Then $\Gamma L\left(\Z\right)\left(V\right) \simeq a\left(G\left(\Z\right)\right)\left(V\hookrightarrow \h_0 \to \X\right)$, where $a$ is stackification.\\

Note:

\begin{eqnarray*}
\resizebox{1.5in}{!}{$G\left(\Z\right)\left(W \hookrightarrow \h_0 \to \X\right)$} &\simeq& \resizebox{3.5in}{!}{$\underset{U \to \Z} \hc \Hom_{\St\left(S/\X\right)}\left(y\left(W \hookrightarrow \h_0 \to \X\right),y\left(U \hookrightarrow \h_0 \to \X\right)\right)$}\\
&\simeq& \underset{U \to \Z} \hc \Hom_{S/\X}\left(W \hookrightarrow \h_0 \to \X,U \hookrightarrow \h_0 \to \X\right)\\
&\simeq& \underset{U \to \Z} \hc \Hom_{\sit\left(\h\right)}\left(W,U\right)\\
&\simeq& \left(\underset{U \to \Z} \hc y\left(U\right)\right) \left(W\right)\\
&\simeq& \Z\left(W\right).\\
\end{eqnarray*}
Given any weak presheaf in groupoids $\W$ over a Grothendieck site $\left(\C,J\right)$, we define $\W^+$ by
$$\W^{+}\left(C\right) =  \underset{\left(C_i \to C\right)_i} \hc \hl \left[ {\prod \limits_i{\W\left(C_i\right)}} \rrarrow {\prod \limits_{i,j}{\W\left(C_{ij}\right)}} \rrrarrow {\prod\limits_{i,j,k} {\W\left(C_{ijk}\right)}}\right].$$
Then $a\left(\W\right)=\W^{+++}$ (See for instance \cite{htt}, section 6.5.3). Now,

\begin{eqnarray*}
\resizebox{1in}{!}{$G\left(\Z\right)^{+}\left(j_{\h}\left(V\right)\right)$}&=& \resizebox{3.88in}{!}{$\underset{\left(V_i \hookrightarrow V\right)_i} \hc \hl \left[{\prod \limits_i{G\left(j_{\h}\left(V_i\right)\right)}} \rrarrow {\prod \limits_{i,j}{G\left(j_{\h}\left(V_{ij}\right)\right)}} \rrrarrow {\prod\limits_{i,j,k} {G\left(j_{\h}\left(V_{ijk}\right)\right)}}\right]$}\\
&\simeq& \underset{\left(V_i \hookrightarrow V\right)_i}\hc \hl \left[ {\prod \limits_i{\Z\left(V_i\right)}} \rrarrow {\prod \limits_{i,j}{\Z\left(V_{ij}\right)}} \rrrarrow {\prod\limits_{i,j,k} {\Z\left(V_{ijk}\right)}}\right]\\
&\simeq&\Z^{+}\left(V\right).\\
\end{eqnarray*}
Hence

\begin{eqnarray*}
\Gamma L\left(\Z\right)\left(V\right) &\simeq& a\left(G\left(\Z\right)\right)\left(V\hookrightarrow \h_0 \to \X\right)\\
&\simeq& \left(G\left(\Z\right)\right)^{+++}\left(V\hookrightarrow \h_0 \to \X\right)\\
&\simeq& \Z^{+++}\left(V\right)\\
&\simeq& a\left(\Z\right)\left(V\right).\\
\end{eqnarray*}

\end{proof}

\begin{cor}\label{cor:rest}
The adjunction $L$  $\rt$  $\Gamma$ restricts to an adjunction
$$\xymatrix{\St\left(\X\right)  \ar@<-0.5ex>[r]_-{\tilde L}  & \St\left(S/\X\right)\ar@<-0.5ex>[l]_-{\tilde \Gamma}},$$
where $\tilde L$ and $\tilde \Gamma$ denote the restrictions. This furthermore restricts to an adjoint-equivalence
$$\xymatrix{\St\left(\X\right)  \ar@<-0.5ex>[r]_-{\bar L}  & \mathscr{E}ss\left(L\right)\ar@<-0.5ex>[l]_-{\bar \Gamma}},$$
equivalence between $\St\left(\X\right)$ and its essential image under $L$.
\end{cor}

The first part of this Corollary is clear. In general, a $2$-adjunction restricts to an equivalence between, on one hand, those objects for which the component of the unit is an equivalence, and on the other hand, those objects for which the component of the co-unit is an equivalence. Hence, it suffices to prove that the essential image of $L$ is the same as the essential image of $\bar L$. In fact, we will prove more, namely:

\begin{prop}\label{relp}
Suppose $\Z$ is a weak presheaf of groupoids over $\sit\left(\h\right)$. Then $L\left(\Z\right)\simeq L\left(a\left(\Z\right)\right)).$
\end{prop}

\begin{proof}
$\tilde L \circ a$ and $L$ are both weak colimit preserving and agree on representables.
\end{proof}

\begin{rmk}
If $\X$ is equivalent to a space $X$, then this construction generalizes the \'etal\'e space construction from sheaves over $X$ to stacks over $X$ (in the ordinary sense). In the particular case when the stack over $X$ is a sheaf of sets, then its \'etal\'e realization is its (Yoneda-embedded) \'etal\'e space.
\end{rmk}

\section{A Concrete Description of \'Etal\'e Realization}\label{sec:conc}
The construction given for the \'etal\'e realization of a small stack over an \'etale stack, as of now, is rather abstract, since it is given as a weak left Kan extension. In order to work with this construction, we wish to give a more concrete description of it. To accomplish this, it is useful first to have a more concrete hold on how to represent these small stacks themselves.

For a general Grothendieck site $\left(\C,J\right)$, one way of representing stacks is by groupoid objects in sheaves. Given a groupoid object $\mathbb{G}$ in $\Sh\left(\C\right)$, it defines a strict presheaf of groupoids by assigning an object $C$ of $\C$ the groupoid $$\Hom_{Gpd\left(\Sh\left(\C\right)\right)}\left(y\left(C\right)^{id},\mathbb{G}\right),$$ where $y\left(C\right)^{id}$ is the groupoid object in sheaves with objects $y\left(C\right)$ and with only identity arrows, where $y$ denotes the Yoneda embedding. This strict presheaf is a sheaf of groupoids. In fact, there is an equivalence of $2$-categories between groupoid objects in sheaves, and sheaves of groupoids. Moreover, every stack on $\left(\C,J\right)$ is equivalent to the stackification of such a strict presheaf arising from a groupoid object in sheaves. For details see Appendix \ref{App:sheavess}.

In our case, we have a nice description of sheaves on $\sit\left(\h\right)$, namely, it is the classifying topos $\B\h$ of equivariant sheaves. Hence, we can model small stacks over $\left[\h\right]$ by groupoid objects in $\h$-equivariant sheaves. In the following subsection, we will describe a way to construct from a given groupoid object $\K$ in equivariant sheaves, an \'etale stack over $\left[\h\right]$ which will turn out to be equivalent to the \'etale realization of the stack over $\sit\left(\h\right)$ associated to $\K$.

\subsection{Generalized action groupoids}\label{subsec:action}
\begin{dfn}
Let $\h$ be any $S$-groupoid and let $\K$ be a groupoid object in $\h$-spaces. In particular we have two $\h$-spaces $\left(\K_0,\mu_0,\rho_0\right)$ and $\left(\K_1,\mu_1,\rho_1\right)$ which are the underlying objects and arrows of $\K$. Note that the source map $$s:\left(\K_1,\mu_1,\rho_1\right) \to \left(\K_0,\mu_0,\rho_0\right)$$ and target map $$t:\left(\K_1,\mu_1,\rho_1\right) \to \left(\K_0,\mu_0,\rho_0\right)$$ are maps $s,t:\left(\K_1,\mu_1,\right) \to \left(\K_0,\mu_0,\right)$ in $S/\h_0$, hence $\mu_0 \circ s=\mu_0 \circ t=\mu_1$. Similarly for other structure maps.

We define an $S$-groupoid $\h \ltimes \K$ as follows:

The space of \textbf{objects} of $\h \ltimes \K$ is $\K_0$. An \textbf{arrow} from $x$ to $y$ is a pair $\left(h,k\right)$ with $h \in \h_1$ and $k \in \K_1$ such that $k:hx \to y$ (which implicitly means that $s(h)=\mu_0(x)$). We denote such an arrow pictorially as

$$x \stackrel{h}{\dashrightarrow} hx \stackrel{k}{\rightarrow}y.$$
In other words, $\left(\h \ltimes \K\right)_1$ is the fibered product $\h_1 \times_{\h_0} \K_1$:

$$\xymatrix{\h_1 \times_{\h_0} \K_1 \ar[d]_-{pr_1} \ar[r]^-{pr_2} & \K_1 \ar[d]^{\mu_1} \\
\h_1 \ar[r]^{t} & \h_0,}\\$$
and the \textbf{source} and \textbf{target} maps are given by $$s\left(h,k\right)=h^{-1}s\left(k\right)$$ and $$t\left(h,k\right)=t\left(k\right).$$
We need to define \textbf{composition}. Suppose we have two composable arrows:

$$x \stackrel{h}{\dashrightarrow} hx \stackrel{k}{\rightarrow}t(k)\stackrel{h'}{\dashrightarrow}h't(k) \stackrel{k'}{\rightarrow} t(k').$$
Notice that $\mu_1\left(k\right)=\mu_0\left(t(k)\right)$ so that $h'$ can act on $k$. So we get an arrow $$h' \cdot k:\left(h'h\right)x \to h't(k).$$ We define the composition to be

$$x \stackrel{h'h}{\dashrightarrow} h'hx \stackrel{k'\left(h'\cdot k\right)}{-\!\!\!-\!\!\!-\!\!\!-\!\!\!-\!\!\!-\!\!\!\longrightarrow}y.$$
In other words $$\left(h',k'\right) \circ \left(h,k\right):=\left(h'h,k'\circ \left(h'\cdot k\right)\right).$$
The \textbf{unit} map $\K_0 \to \left(\h \ltimes \K\right)_1$ is given by $$x \mapsto \left(\mathbb{1}_{\mu_0\left(x\right)},\mathbb{1}_x\right).$$ And the \textbf{inverse} map is given by $$\left(h,k\right)^{-1}:=\left(h^{-1},h^{-1}\cdot k^{-1}\right).$$
Notice that if $\K$ is actually an $\h$-space $E$ considered as a groupoid object with only identity morphisms, then $\h \ltimes \K$ is the usual action groupoid $\h \ltimes E$. Hence, we call $\h \ltimes \K$ the \textbf{generalized action groupoid} of $\K$, or simply the action groupoid.
\end{dfn}

\begin{rmk}
This construction is known. It appears, for example, in \cite{Fol} under the name \emph{semi-direct product}.
\end{rmk}

Notice that each action groupoid $\h \ltimes \K$ comes equipped with a canonical morphism $\theta_\K:\h \ltimes \K \to \h$ given by $$\left(\theta_{\K}\right)_0=\mu_0:\K_0 \to \h_0$$ and $$\left(\theta_{\K}\right)_{1}=pr_1:\left(\h \ltimes \K\right)_1=\h_1 \times_{\h_0} \K_1 \to \h_1.$$
The following proposition is immediate:

\begin{prop}\label{prop:shver}
If $\h$ is \'etale and $\K$ is in fact a groupoid object in $\h$-equivariant sheaves, then $\h \ltimes \K$ is \'etale and the components of $\theta_\K$ are local homeomorphisms.
\end{prop}

\begin{rmk}
Each groupoid object $\K$ in $\h$-spaces has an underlying $S$-groupoid $\underline \K$ and there is a canonical map $\tau_\K:\underline \K \to \h \ltimes \K$ given by the identity morphism on $\K_0$ and on arrows by $$k \mapsto \left(\mathbb{1}_{\mu_1(k)},k\right).$$
\end{rmk}

Let $\left(S-Gpd\right)/\h$ denote the slice $2$-category of $S$-groupoids over $\h$. We will show that the action groupoid construction $$\K \mapsto \left(\left(\h \ltimes \K\right) \stackrel{\theta_\K}{\searrow} \h\right)$$ extends to a $2$-functor $$\h \ltimes:Gpd\left(\h-spaces\right) \to \left(S-Gpd\right)/\h.$$
Suppose $\varphi:\K \to \Ll$ is a homomorphism of groupoid objects in $\h$-spaces. Then we can define $\h \ltimes\left(\varphi\right):\h \ltimes \K \to \h \ltimes \Ll$ \textbf{on objects} as $\varphi_0$ and \textbf{on arrows} by $$\left(h,k\right) \mapsto \left(h,\varphi\left(k\right)\right),$$ which strictly commutes over $\h$. Finally, for $2$-\textbf{cells}, given an internal natural transformation $$\alpha:\varphi \Rightarrow \psi$$ between two homomorphisms $$\K \to \Ll,$$ $\alpha$ is in particular a map of $\h$-spaces $\alpha:\K_0 \to \Ll_1$. It is easily checked that $\left(\tau_\Ll\right)_1 \circ \alpha:\K_0 \to \left(\h \ltimes \Ll\right)_1$ encodes a $2$-cell $$\h \ltimes\left(\alpha\right):\h \ltimes\left(\varphi\right) \Rightarrow \h \ltimes\left(\psi\right),$$ where $\tau$ is as in the remark directly proceeding Proposition \ref{prop:shver}
We leave it to the reader to check that this is a strict $2$-functor.

\begin{rmk}
This restricts to a $2$-functor

$$\h \ltimes:Gpd\left(\B\h\right) \to \left(S^{et}-Gpd\right)/\h,$$
where $S^{et}$ denotes the category whose objects are spaces and arrows are all local homeomorphisms.
\end{rmk}

Let us now define a strict $2$-functor in the other direction, $$P:\left(S-Gpd\right)/\h \to Gpd\left(\h-spaces\right).$$
\textbf{On objects}:

Let $\varphi:\G \to \h$ be a map of $S$-groupoids. Consider the associated principal $\h$-bundle over $\G$. Its total space is $\h_1 \times_{\h_0} \G_0$, where

$$\xymatrix{\h_1 \times_{\h_0} \G_0 \ar[d]_-{pr_1} \ar[r]^-{pr_2} & \G_0 \ar[d]^-{\varphi_0} \\
\h_1 \ar[r]^-{s} & \h_0\\}$$
is a pullback diagram. Together with its projection $pr_2:\h_1 \times_{\h_0} \G_0 \to \G_0,$ it is a right $\G$-space with action given by $$\left(h,x\right)g:=\left(h\varphi\left(g\right),s(g)\right).$$ We define $$\underline P\left(\varphi\right):=\left(\h_1 \times_{\h_0} \G_0\right) \rtimes \G,$$
that is, the right action groupoid of the underlying $\G$-space of the associated principal bundle of $\varphi$. Since the left $\h$-action and right $\G$-action on $\h_1 \times_{\h_0} \G_0$ commute, this becomes a groupoid object in $\h$-spaces. Explicitly, the objects of $P\left(\varphi\right)$ are $\h_1 \times_{\h_0} \G_0$ equipped with the obvious left $\h$-action along $s\circ pr_1$ given by $$h'\left(h,x\right)=\left(h'h,x\right).$$ The arrows are the fibered product

$$\xymatrix{\h_1 \times_{\h_0} \G_1 \ar[d]_-{pr_1} \ar[r]^-{pr_2} & \G_1 \ar[d]^-{\varphi_0\circ t} \\
\h_1 \ar[r]^-{s} & \h_0,\\}$$
equipped with an analogously defined left $\h$-action along $s \circ pr_1$. The source and target maps are defined by $$s\left(h,g\right)=\left(h\varphi(g),s(g)\right),$$ and $$t\left(h,g\right)=\left(h,t(g)\right).$$ Composition and units are defined in the obvious way.

The following proposition is immediate:
\begin{prop}\label{prop:shver2}
If $\h$ is \'etale and $\varphi_0$ is a local homeomprhism (which implies that so is $\varphi_1$), then $P\left(\varphi\right)$ is a groupoid object in $\B\h$.
\end{prop}

We will now define $P$ \textbf{on arrows}:

Suppose we are given an arrow $$\left(f,\alpha\right):\left(\G\stackrel{\varphi}{\searrow}\h\right) \to \left(\Ll\stackrel{\psi}{\searrow}\h\right).$$

We wish now to define an internal functor $P\left(\left(f,\alpha\right)\right)$. On objects define it by:

$$P\left(\left(f,\alpha\right)\right)\left(h,x\right)=\left(h\alpha(x)^{-1},f(x)\right).$$
On arrows define it by
$$P\left(\left(f,\alpha\right)\right)\left(h,g\right)=\left(h\varphi(g)\alpha\left(s(g)\right)^{-1}\psi\left(f(g)\right)^{-1},f(g)\right).$$  It is routine to verify that this defines an internal functor.

We now define $P$ \textbf{on $2$-cells}:

Suppose we are given a $2$-cell $\omega:\left(f,\alpha\right) \Rightarrow \left(f',\alpha'\right)$ between two maps $$\left(\G\stackrel{\varphi}{\searrow}\h\right) \to \left(\Ll\stackrel{\psi}{\searrow}\h\right).$$
Define an internal natural transformation $$P\left(\omega\right):P\left(\left(f,\alpha\right)\right) \Rightarrow P\left(\left(f',\alpha'\right)\right)$$ by $$P\left(\omega\right)\left(h,x\right)=\left(h\alpha(x),\omega(x)\right).$$ We leave it to the reader to check that $P$ is indeed a strict $2$-functor.

\begin{lem}\label{lem:eps}
There exists a natural transformation $\varepsilon:\h \ltimes P\Rightarrow id_{\left(S-Gpd\right)/\h}$ whose components are equivalences.
\end{lem}

\begin{proof}
Given $\varphi:\G \to \h$, consider the left-action of $\h \times \G$ on $$\h_1 \times_{\h_0} \G_0=P\left(\varphi\right)_0$$ along $$\left(h,x\right) \mapsto \left(t\left(h\right),x\right)$$ defined by

$$\left(l,g\right)\cdot\left(h,x\right):=\left(lh\varphi(g)^{-1},t(g)\right).$$
Consider $$\theta_{P\left(\varphi\right)}:\left(\h \times \G\right) \ltimes \left(\h_1 \times_{\h_0} \G_0\right) \to \left(\h \times \G\right)$$ where $\theta_{P\left(\varphi\right)}$ is the canonical morphism.

By direct inspection, we see that $\h \ltimes P\left(\varphi\right)$ is canonically isomorphic to $$\tilde\theta_{P\left(\varphi\right)}:=pr_1 \circ \theta_{P\left(\varphi\right)}.$$ Consider the map $$\tilde \epsilon_{\varphi}:=pr_2 \circ \theta_{P\left(\varphi\right)}:\left(\h \times \G\right) \ltimes \left(\h_1 \times_{\h_0} \G_0\right) \to \G.$$
Let $\xi_{\varphi}:\h_1 \times_{\h_0} \G_0 \to \h_1$ be the obvious projection map. Then $\xi_{\varphi}$ is a natural isomorphism from $\varphi \circ \tilde \epsilon_{\varphi}$ to $\tilde\theta_{P\left(\varphi\right)}.$ Hence $\left(\tilde \epsilon_{\varphi},\xi^{-1}_{\varphi}\right)$ is a morphism in $\left(S-Gpd\right)/\h$ from $\tilde\theta_{P\left(\varphi\right)}$ to $\varphi$. It is easy to check that $$\epsilon:\h \ltimes P \circ \Rightarrow id_{\left(S-Gpd\right)/\h}$$ defined by $$\varepsilon\left(\varphi\right)=\left(\tilde \varepsilon_{\varphi}, \xi^{-1}_{\varphi}\right),$$ is a strict natural transformations of $2$-functors. It remains to see that its components consist of equivalences.

Define $\chi_{\varphi}:\G \to \left(\h \times \G\right) \ltimes \left(\h_1 \times_{\h_0} \G_0\right)$ on objects by $$\chi_{\varphi}\left(x\right)=\left(\mathbb{1}_{\varphi(x)},x\right),$$ and on arrows by $$\chi_{\varphi}\left(g\right)=\left(\left(\mathbb{1}_{\varphi(s(g))},s(g)\right),\left(\varphi(g),g\right)\right).$$ Then $$\tilde \varepsilon_{\varphi} \circ \chi_{\varphi}=id_\G.$$
Note that $\tilde\theta_{P\left(\varphi\right)} \circ \chi_{\varphi}=\varphi$ so that $\chi_{\varphi}$ is a morphism in $\left(S-Gpd\right)/\h$.

Define $$\lambda_{\varphi}:\h_1 \times_{\h_0} \G \to \left(\h \times \G\right) \ltimes \left(\h_1 \times_{\h_0} \G_0\right)_1$$ by

$$\lambda_{\varphi}\left(h,x\right)=\left(\left(\mathbb{1}_{\varphi(x)},x\right),\left(h,\mathbb{1}_x\right)\right).$$
Then $\lambda_{\varphi}$ encodes a $2$-cell $id_{\h \ltimes P\left(\varphi\right)} \Rightarrow \chi_{\varphi} \circ \varepsilon_{\varphi}$.
\end{proof}

\begin{cor}\label{cor:im1}
The $2$-functors $$\h \ltimes:Gpd\left(\h-spaces\right) \to \left(S-Gpd\right)/\h$$ and $$\h \ltimes:Gpd\left(\B\h\right) \to \left(S^{et}-Gpd\right)/\h$$ are bicategorically essentially surjective.
\end{cor}

\subsection{Action groupoids are \'etal\'e realizations}

Let $\h$ be an \'etale groupoid and $\X$ its associated \'etale stack, $\left[\h\right].$ Let $$Y:\left(S^{et}-Gpd\right)/\h \to \St\left(S\right)/\X$$ be the $2$-functor which sends a groupoid $\varphi:\G \to \h$ over $\h$ to $$\left[\varphi\right]:\left[\G\right] \to \left[\h\right]=\X.$$ Consider furthermore the canonical $2$-functor
$$\left[ \mspace{3mu} \cdot \mspace{3mu}\right]_{\B\h}:Gpd\left(\B\h\right) \to St\left(\sit\left(\h\right)\right)$$
which associates a groupoid object $\K$ in in $\B\h$ with its stack completion.

\begin{thm}\label{thm:finn}
The $2$-functor $\left[ \mspace{3mu} \cdot \mspace{3mu}\right]_{\B\h}$ is essentially surjective and faithful (but not in general full), and the $2$-functors $\bar L \circ \left[ \mspace{3mu} \cdot \mspace{3mu}\right]_{\B\h}$ and $Y \circ \h \ltimes$ are equivalent.
\end{thm}

The proof of this theorem is quite involved, so it is delayed to Appendix \ref{sec:lem}.

\begin{rmk}
In particular, this implies that if $\Z$ is a small stack over $\X$ represented by a groupoid object $\K$ in $\B\h$, then $L\left(\Z\right) \simeq Y\left(\h \ltimes \K\right)$.
\end{rmk}

\begin{dfn}
A morphism $\Y \to \X$ of \'etale stacks is said to be a \textbf{local homeomorphism} if it can be represented by a map $\varphi:\G \to \h$ of $S$-groupoids such that $\varphi_0$ (and  hence $\varphi_1$) is a local homeomorphisms of spaces. Denote the full sub-2-category of $St\left(S\right)/\X$ spanned by local homeomorphisms over $\X$ by $Et\left(\X\right)$.
\end{dfn}

In light of Theorem \ref{thm:finn} and Proposition \ref{relp}, the essential image of $L$ is precisely the local homeomorphisms over $\X$. Moreover, with Corollary \ref{cor:rest}, this implies:

\begin{cor}\label{cor:real}
$$\xymatrix{\St\left(\X\right)  \ar@<-0.5ex>[r]_-{\bar L}  & Et\left(\X\right)\ar@<-0.5ex>[l]_-{\bar \Gamma}},$$
is an adjoint-equivalence between $\St\left(\X\right)$ and local homeomorphisms over $\X.$
\end{cor}

\begin{rmk}
Note that there is a small error on the top of page 44 of \cite{Metzler}; the construction $P_1$, which assigns a stack $\Z$ over a space $X$ an \'etale groupoid over $X$ via a local homeomorphism, is not functorial with respect to all maps of stacks. It is only functorial with respect to \emph{strict} natural transformations of stacks, but in general, one must consider also pseudo-natural transformations. The above corollary may be seen as a corrected version of this construction, in the case that $\X$ is a space $X$. Note that this error also makes Theorem 94 of \cite{Metzler} incorrect. The corrected version of Theorem 94 is explained in section \ref{sec:groth} of this paper.
\end{rmk}

\subsection{The inverse image functor}\label{sec:inverseimage}

Suppose $f:\Y \to \X$ is a morphism of \'etale stacks. This induces a geometric morphism of $2$-topoi $\St\left(\Y\right) \to \St\left(\X\right)$, where by this we mean a pair of adjoint $2$-functors $f^* \rt \mspace{2mu} f_*$, such that $f^*$ preserves finite (weak) limits \cite{htt}. To see this, note that there is a canonical trifunctor $$\Top \to 2-\Top,$$ from topoi to $2$-topoi, which sends a topos $\E$ to the $2$-topos of stacks over $\E$ with the canonical topology. Since, $$\Sh:\St\left(S\right) \to \Top$$ is a $2$-functor, so we get an induced geometric morphism $$\Sh\left(f\right):\Sh\left(\Y\right) \to \Sh\left(\X\right),$$ which in turn gives rise to a geometric morphism $$\St\left(f\right):\St\left(\Y\right) \to \St\left(\X\right),$$ after applying the trifunctor $\Top \to 2-\Top$. We denote the direct and inverse image $2$-functors by $f_*$ and $f^{*}$.

We also get an induced geometric morphism between the $2$-topoi of \emph{large} stacks, $$\St\left(f\right):\St\left(S/\Y\right) \to \St\left(S/\X\right).$$ This arises as the adjoint pair of slice $2$-categories
$$\xymatrix{\St\left(S\right)/\Y  \ar@<-0.5ex>[r]_-{f_*} & \St\left(S\right)/\X \ar@<-0.5ex>[l]_-{f^*}},$$ induced by $f$. The inverse image $2$-functor $f^*$ is given by pullbacks:\\
If $\Z \to \X$ is in $\St\left(S\right)/\X$, then $f^*\left(\Z \to \X\right)$ is given by  $\Y \times_{\X} \Z \to \Y$.

\begin{thm}\label{thm:inv}

The following diagram $2$-commutes:

$$\xymatrix{\St\left(\X\right) \ar[r]^-{\bar L} \ar[d]_{f^*} & \St\left(S\right)/\X \ar[d]^-{f^*}\\
 \St\left(\Y\right) \ar[r]^-{\bar L} &  \St\left(S\right)/\Y,}$$
where $\bar L$ is as in Corollary \ref{cor:rest}.

\end{thm}

\begin{proof}
As both composites $f^* \circ \bar L$ and $\bar L \circ f^*$ are weak colimit preserving, it suffices to show that they agree on representables. We fix an \'etale $S$-groupoid $\h$ such that $\left[\h\right] \simeq \X$ and \emph{choose} a particular $\G$ such that $$\left[\G\right] \simeq \Y$$ and $f=\left[\varphi\right]$ with $\varphi:\G \to \h$ an internal functor. Choose a representable sheaf $m_U \in \B\h$. From \cite{cont}, for any equivariant sheaf $$\h \acts E \stackrel{\mu}{\longrightarrow} \h_0,$$ $\varphi^*\left(E\right)$ as a sheaf over $\G_0$ is given by $$\G_0\times_{\h_0} E \to \G_0$$ and has the $\G$-action
$$g \cdot \left(x,e\right)=\left(t\left(g\right),\varphi\left(g\right) \cdot e\right).$$
Hence $\bar L\left(f^*m_u\right)$ is given by $Y \left(\G \ltimes \left(\G_0\times_{\h_0} s^{-1}\left(U\right)\right)\right).$
Explicitly, the arrows may be described by pairs $\left(g,h\right) \in \G_1 \times s^{-1}\left(U\right)$ such that $$s\varphi\left(g\right)=t\left(h\right).$$

The other composite, $$f^*\bar L\left(m_u\right)$$ is given by $$\left[\G\right] \times_{\left[\h\right]} \left[\h \ltimes s^{-1}\left(U\right)\right] \to \left[\G\right].$$ Since the extended Yoneda $2$-functor preserves all weak limits, and stackification preserves finite ones, this pullback may be computed in $S$-groupoids. Its objects are triples $$\left(z,h,\alpha\right) \in \G_0 \times s^{-1}\left(U\right) \times \h_1$$ such that $$\varphi_0\left(z\right) \stackrel{\alpha}{\longrightarrow} t\left(h\right).$$
Its arrows are quadruples $$\left(g,h,h',\alpha\right) \in \G_1 \times \h_1 \times s^{-1}\left(U\right) \times \h_1$$ such that
$$s\left(\varphi\left(g\right)\right)=s\left(\alpha\right)$$ and
$$t\left(\alpha\right)=s\left(h'\right)=t\left(h\right).$$
Such a quadruple is an arrow from $\left(s\left(g\right),h,\alpha\right)$ to $\left(t\left(g\right),h'h,h'\alpha\varphi\left(g\right)^{-1}\right).$
The projections are defined  by

\begin{eqnarray*}
pr_1:\G \times_{\h}\left(\h \ltimes s^{-1}\left(U\right)\right) &\to& \G\\
\left(z,h,\alpha\right) &\mapsto& z\\
\left(g,h',h,\alpha\right) &\mapsto& g\\
\end{eqnarray*}
and
\begin{eqnarray*}
pr_2:\G \times_{\h}\left(\h \ltimes s^{-1}\left(U\right)\right) &\to& \h \ltimes s^{-1}\left(U\right)\\
\left(z,h,\alpha\right) &\mapsto& h\\
\left(g,h',h,\alpha\right) &\mapsto& \left(h',h\right).\\
\end{eqnarray*}
We now define an internal functor $$\zeta:\G \times_{\h}\left(\h \ltimes s^{-1}\left(U\right)\right) \to \G \ltimes \left(\G_0 \times_{\h_0}  s^{-1}\left(U\right)\right)$$ on objects by $$\left(z,h,\alpha\right) \mapsto \left(z,\alpha^{-1}h\right)$$ and on arrows by $$\left(g,h',h,\alpha\right) \mapsto \left(g,\alpha^{-1}h\right).$$
We define another internal functor $$\psi:\G \ltimes \left(\G_0 \times_{\h_0}  s^{-1}\left(U\right)\right) \to \G \times_{\h}\left(\h \ltimes s^{-1}\left(U\right)\right)$$ on objects as
$$\left(z,h\right) \mapsto \left(z,\mathbb{1}_{s(h)},h^{-1}\right)$$ and on arrows as $$\left(g,h\right)=\left(g,\mathbb{1}_{s(h)},\mathbb{1}_{s(h)},h^{-1}\right).$$
Note that $\psi$ is a left inverse for $\zeta$. We define an internal natural isomorphism $$\omega:\psi \circ \zeta \Rightarrow id_{\G \times_{\h}\left(\h \ltimes s^{-1}\left(U\right)\right)}$$ by
$$\omega\left(z,h,\alpha\right)=\left(\mathbb{1}_z,h^{-1},h,\alpha\right):\left(z,h,\alpha\right) \to \left(z,\mathbb{1}_{s(h)},h^{-1}\alpha\right)=\psi\zeta\left(z,h,\alpha\right).$$
As both $\zeta$ and $\psi$ commute strictly over $\G$, this establishes our claim.
\end{proof}

\begin{dfn}
For $\Z$ a small stack over an \'etale stack $\X$, and $$x:* \to \X$$ a point of $\X,$ the \textbf{stalk} of $\Z$ at $x$ is the groupoid $x^*\left(\Z\right),$ where we have made the identification $\St\left(*\right)\simeq Gpd.$ We denote this stalk by $\Z_x$.
\end{dfn}

As we have just seen, this stalk may be computed as the fiber of $$\bar L\left(\Z\right) \to \X$$ over $x,$ i.e. the weak pullback $* \times_{\X} \bar L\left(\Z\right),$ which is a constant stack with value $x^*\left(\Z\right)$. This stalk can also be computed analogously to stalks of sheaves:

\begin{lem}
Let $x \in X$ be a point of a space, and let $\Z$ be a small stack over $X$. Then the stalk at $x$ of $\Z$ can be computed by $$\Z_x\simeq  \underset{x \in U} \hc\Z\left(U\right),$$ where the weak colimit is taken over the open neighborhoods of $x$ regarded as a full subcategory of $\mathcal{O}\left(X\right).$
\end{lem}
\begin{proof}
The $2$-functor
\begin{eqnarray*}
\St\left(X\right) &\to& Gpd\\
\Z &\mapsto& \underset{x \in U}\hc  \Z\left(U\right),\\
\end{eqnarray*}
is clearly weak colimit preserving. If $\Z=V \subseteq X$ is a representable sheaf, i.e., an open subset of $X$, then $$\Z_x\simeq \underset{x \in U} \hc \Hom\left(U,V\right) \simeq \underset{x \in U} \varinjlim \Hom\left(U,V\right),$$ and the latter expression is equivalent to the singleton set if $x \in V$ and the empty set otherwise. This set is the same as the fiber of $V$ over $x,$ i.e. the stalk $V_x \cong x^*\left(V\right).$ So $$\Z \mapsto \underset{x \in U} \hc \Z\left(U\right)$$ is weak colimit preserving and agrees with $x^*$ on representables, hence is equivalent to $x^*$.
\end{proof}

\begin{cor}\label{cor:stalk}
Let $x:* \to \X$ be a point of an \'etale stack $\X\simeq \left[\h\right],$ with $\h$ an \'etale groupoid. Pick a point $\tilde x \in \h_0$ such that $x\cong p \circ \tilde x$ where $$p:\h_0 \to \X$$ is the atlas associated to $\h$. Let $\Z$ be a small stack over $X$. Then the stalk at $x$ of $\Z$ can be computed by $$\Z_x\simeq  \underset{\tilde x \in U} \hc\Z\left(U\right),$$ where the weak colimit is taken over the open neighborhoods of $\tilde x$ in $\h_0$ regarded as a full subcategory of $\mathcal{O}\left(\h_0\right).$
\end{cor}

\begin{proof}
Since $x\cong p \circ \tilde x,$ it follows that $$\Sh\left(x\right) \simeq \Sh\left(p\right) \circ \Sh\left(\tilde x\right):\Sh\left(*\right) \to \Sh\left(\X\right),$$ and hence $$x^* \simeq \tilde x^* \circ p^*.$$ By definition, for $U$ an open subset of $\h_0,$ $$p^*\left(\Z\right)\left(U\right)\simeq\Z\left(U\right).$$ Hence,

\begin{eqnarray*}
\Z_x&=& x^*\Z\\
 &\simeq& \tilde x^*\left(p^*\Z\right)\\
 &\simeq& \left(p^*\Z\right)_{\tilde x}\\
 &\simeq& \underset{\tilde x \in U} \hc \left(p^*\Z\right)\left(U\right)\\
 &\simeq& \underset{\tilde x \in U} \hc \Z\left(U\right).\\
\end{eqnarray*}
\end{proof}

\subsection{A classification of sheaves}

From Corollary \ref{cor:real}, we know that the for an \'etale stack $\X$, the $2$-category of local homeomorphisms over $\X$ is equivalent to the $2$-category of small stacks over $\X$. A natural question is which objects in $Et\left(\X\right)$ are actually \emph{sheaves} over $\X$, as opposed to stacks, i.e., what are the $0$-truncated objects?

\begin{thm}
A local homeomorphism $f:\Z \to \X$ over an \'etale stack $\X$ is a equivalent to $\bar L\left(F\right)$ for a small sheaf $F$ over $\X$ if and only if it is a representable map.
\end{thm}

\begin{proof}
Suppose $F$ is a small sheaf over $\X\simeq\left[\h\right]$ with $\h$ an \'etale $S$-groupoid. Denote by  $$\underline{\bar L \left(F\right)} \to \X$$ the map $\bar L\left(F\right)$. We wish to show that $$\underline{\bar L \left(F\right)} \to \X$$ is representable. It suffices to show that the $2$-pullback

$$\xymatrix{\h_0 \times_{\X} \underline{\bar L \left(F\right)} \ar[d] \ar[r] & \underline{\bar L \left(F\right)} \ar[d]\\
\h_0 \ar[r]^-{a} & \X},$$
is (equivalent to) a space, where $a:\h_0 \to \X$ is the atlas associated to $\h$. By Theorem \ref{thm:inv}, this pullback is equivalent to the total space of the \'etal\'e space of the sheaf $a^*\left(F\right)$ over $\h_0$.
Conversely, suppose $\Z \to \X$ is a representable local homeomorphism equivalent to $\bar L\left(\W\right)$ for some small stack $\W$. Then the pullback  $$\h_0 \times_{\X} \underline{\bar L \left(F\right)}$$ is equivalent to a space. This implies that $a^*\left(W\right)$ is a sheaf of sets over $\h_0$. By definition $a^*\left(\W\right)$ assigns to each open subset $U$ of $\h_0$ the groupoid $\W\left(m_U\right)$. It follows that $\W$ must be a sheaf.
\end{proof}

\begin{cor}
For an \'etale stack $\X$, the category of small sheaves over $\X$ is equivalent to the $2$-category of representable local homeomorphisms over $\X$.
\end{cor}

\begin{rmk}
This implies that the $2$-category of representable local homeomorphisms over $\X$ is equivalent to its $1$-truncation.
\end{rmk}

\begin{rmk}
This gives a purely intrinsic definition of the topos of sheaves $\Sh\left(\X\right)$. In particular, a posteriori, we could define a small stack over $\X$ to be a stack over this topos. We note for completeness that a site of definition of this topos is the category of local homeomorphisms $T \to \X$ from $T$ a space, with the induced open cover topology. This is equivalent to the category of principal $\h$-bundles whose moment map is a local homeomorphism.
\end{rmk}

\section{A Groupoid Description of the Stack of Sections}\label{sec:section}
\sectionmark{The Stack of Sections}
Now that we have a concrete description of $\bar L$ in terms of groupoids, it is natural to desire a similar description for $\bar \Gamma$ (where $\bar L$ and $\bar \Gamma$ are as in Corollary \ref{cor:rest}).

\begin{lem}\label{lem:sec1}
Suppose that $\varphi:T \to \h$ is a local homeomorphism from a space $T$, with $\h$ an \'etale groupoid. Then $\bar \Gamma\left(\left[\varphi\right]\right)$ is the equivariant sheaf $P\left(\varphi\right)\in \B\h$, where $P$ is as in Section \ref{subsec:action}.
\end{lem}

\begin{proof}
Let $m_U$ be a representable sheaf in $\B\h$. Then

\begin{eqnarray*}
\Gamma\left(\left[\varphi\right]\right)\left(U\right) &\simeq&  \Hom\left(\bar L\left(m_U\right),\left[\varphi\right]\right)\\
&\simeq& \Hom\left(\left[\h \ltimes s^{-1}\left(U\right)\right],\left[\varphi\right]\right).\\
\end{eqnarray*}
Since $T$ is a sheaf, the later is in turn equivalent to $$\Hom_{Gpd/\h}\left(\h \ltimes s^{-1}\left(U\right),\varphi\right).$$ This follows from the canonical equivalence $$\Hom\left(\tilde y\left(\h \ltimes s^{-1}\left(U\right)\right),T\right) \simeq \Hom\left(\left[\h \ltimes s^{-1}\left(U\right)\right],T\right).$$ In fact, this is a set, since $T$ has no arrows, so there are no natural transformations. An element of this set is the data of a groupoid homomorphism $$\psi:\h \ltimes s^{-1}\left(U\right) \to T$$ together with an internal natural transformation
$$\beta:\theta_{m_U} \Rightarrow \varphi \circ \psi.$$
To ease notation, let $\alpha:=\beta^{-1}.$ Since $T$ is a space, $\psi_1$ is determined by $\psi_0$ by the formula $$\psi_1\left(\left(h,\gamma\right)\right)=\psi_0\left(\gamma\right)=\psi_0\left(h\gamma\right).$$ Notice that this also imposes conditions on $\psi_0$, namely that it is constant on orbits. The internal natural transformation is a map of spaces $$\alpha:s^{-1}\left(U\right) \to \h_1$$ such that for all $\gamma \in s^{-1}\left(U\right),$ $$\alpha\left(\gamma\right):\varphi\psi_0\left(\gamma\right) \to t\left(\gamma\right).$$ Because of the constraints on $\psi$, the naturality condition is equivalent to $$\alpha\left(h\gamma\right)=h\alpha\left(\gamma\right).$$  This data defines a map $$m_U \to P\left(\varphi\right)$$ by
\begin{eqnarray*}
s^{-1}\left(U\right) &\to& \h_1\times_{\h_0}T\\
\gamma &\mapsto&\left(\alpha\left(\gamma\right),\psi_0\left(\gamma\right)\right).\\
\end{eqnarray*}
Conversely, any map $f:m_U \to P\left(\varphi\right)$ defines a morphism $$\hat{f}:\h \ltimes s^{-1}\left(U\right) \to T$$ on objects by $pr_2 \circ f$ (and hence determines it on arrows), and since $f$ is $\h$-equivariant, and the $\h$-action on $\h_1\times_{\h_0}T$ does not affect $T$, this map is constant on orbits. The map $f$ induces an internal natural transformation $$\alpha_f:\varphi \circ \hat{f} \to \theta_{m_U}$$ by $\alpha_f=pr_1 \circ f.$ This establishes a bijection
$$\Hom_{Gpd/\h}\left(\h \ltimes s^{-1}\left(U\right),\varphi\right) \cong \Hom_{\B\h}\left(m_U,P\left(\varphi\right)\right).$$
Hence $$\Gamma\left(\left[\varphi\right]\right)\left(U\right) \simeq  \Hom_{\B\h}\left(m_U,P\left(\varphi\right)\right),$$ so we are done by the Yoneda Lemma.
\end{proof}

\begin{thm}\label{thm:secs}
Suppose that $\varphi:\G \to \h$ is a homomorphism of \'etale $S$-groupoids with $\varphi_0$ a local homeomorphism. Then $\bar \Gamma\left(\left[\varphi\right]\right)$ is equivalent to the stack associated to the groupoid object $P\left(\varphi\right)$ in $\B\h$.
\end{thm}

\begin{proof}
Let $a:\G_0 \to \left[\G\right]$ denote the atlas of the stack $\left[\G\right]$. There is a canonical map $$p:\bar \Gamma\left(\left[\varphi\right] \circ a\right) \to \bar\Gamma\left(\left[\varphi\right]\right),$$ and since $a$ is an epimorphism, it follows that $p$ is an epimorphism as well. Since $p$ is an epimorphism from a sheaf to a stack, it follows that
$$\bar \Gamma\left(\left[\varphi\right] \circ a\right) \times_{\bar\Gamma\left(\left[\varphi\right]\right)}\bar \Gamma\left(\left[\varphi\right] \circ a\right)\rightrightarrows \bar \Gamma\left(\left[\varphi\right] \circ a\right),$$ is a groupoid object in sheaves (i.e. the classifying topos $\B\h$) whose stackification is equivalent to $\bar \Gamma\left(\left[\varphi\right]\right)$. We will show that this groupoid is isomorphic to $P\left(\varphi\right)$. This isomorphism is clear on objects from the previous lemma.

Since pullbacks are computed object-wise, as a sheaf, $$\bar \Gamma\left(\left[\varphi\right] \circ a\right) \times_{\bar\Gamma\left(\left[\varphi\right]\right)}\bar \Gamma\left(\left[\varphi\right] \circ a\right)$$ assigns $U \in \sit\left(\h\right)$ the pullback groupoid

$$\bar \Gamma\left(\left[\varphi\right] \circ a\right)\left(U\right) \times_{\bar\Gamma\left(\left[\varphi\right]\right)\left(U\right)}\bar \Gamma\left(\left[\varphi\right] \circ a\right)\left(U\right),$$
which is indeed (equivalent to) a set. It is the set of pairs of objects in $\Gamma\left(\left[\varphi\right] \circ a\right)\left(U\right)$ together with a morphism in $\bar\Gamma\left(\left[\varphi\right]\right)\left(U\right)$ between their images under $p\left(U\right)$.

Since for all $S$-groupoids, the induced map $$\Hom\left(\Ll,\K\right) \to \Hom\left(\left[\Ll\right],\left[\K\right]\right)$$ is full and faithful, we may describe this set in terms of maps of groupoids. It has the following description:

An element of $\bar \Gamma\left(\left[\varphi\right] \circ a\right)\left(U\right) \times_{\bar\Gamma\left(\left[\varphi\right]\right)\left(U\right)}\bar \Gamma\left(\left[\varphi\right] \circ a\right)\left(U\right),$ can be represented by two pairs $\left(\sigma_0,\alpha_0\right)$  and $\left(\sigma_1,\alpha_1\right),$ such that for $i=0,1$,
$$\xymatrix{\h \ltimes s^{-1}\left(U\right) \ar@/_1.65pc/_-{\theta_{m_U}}[rrrdd] \ar@{} @<-15pt>  [rrrdd]| (0.4) {}="a" \ar[r]^-{\sigma_i} & \G_0  \ar[rd]^-{\tilde a} \ar@{} @<-8pt>  [rd]| (0.4) {}="b"  \ar @{=>}^{\alpha_i}  "b";"a" \\
& & \G\ar[rd]^-{\varphi}\\
& & & \h,}$$
where $\tilde a:\G_0 \to \G$ is the obvious map such that $\left[\tilde a\right]=a,$
together with a $2$-cell $$\beta:\tilde a \circ \sigma_0 \Rightarrow \tilde a \circ \sigma_1,$$ such that the following diagram commutes:
\begin{equation}\label{diag:tri}
\xymatrix{\varphi\circ \tilde a \circ \sigma_0 \ar@{=>}[rr]^-{\varphi\beta} \ar@{=>}[rd]_-{\alpha_0}& & \varphi\circ \tilde a \circ \sigma_1 \ar@{=>}[ld]^-{\alpha_1}\\
& \theta_{m_U} &\\}.
\end{equation}
Each pair $\left(\sigma_i,\alpha_i\right)$ represents
$$\xymatrix{\left[\h \ltimes s^{-1}\left(U\right)\right] \ar@/_1.65pc/_-{\left[\theta_{m_U}\right]}[rrrdd] \ar@{} @<-15pt>  [rrrdd]| (0.4) {}="a" \ar[r]^-{\left[\sigma_i\right]} & \G_0  \ar[rd]^-{a} \ar@{} @<-8pt>  [rd]| (0.4) {}="b"  \ar @{=>}^{\left[\alpha_i\right]}  "b";"a" \\
& & \left[\G\right]\ar[rd]^-{\left[\varphi\right]}\\
& & & \left[\h\right],}$$
i.e. the element $\left(\left[\sigma_i\right],\left[\alpha_i\right]\right)$ of the set $\bar \Gamma\left(\left[\varphi\right] \circ a\right)\left(U\right).$
The groupoid structure on $$\bar \Gamma\left(\left[\varphi\right] \circ a\right) \times_{\bar\Gamma\left(\left[\varphi\right]\right)}\bar \Gamma\left(\left[\varphi\right] \circ a\right)\rightrightarrows \bar \Gamma\left(\left[\varphi\right] \circ a\right),$$ is such that the data $$\left(\left(\sigma_0,\alpha_0\right),\left(\sigma_1,\alpha_1\right),\beta\right)$$ is an arrow from $\left(\left[\sigma_0\right],\left[\alpha_0\right]\right)$ to $\left(\left[\sigma_0\right],\left[\alpha_0\right]\right).$

Recall that the arrows of $P\left(\varphi\right)$ are the equivariant sheaf described as the fibered product
$$\xymatrix{\h_1 \times_{\h_0} \G_1 \ar[d]_-{pr_1} \ar[r]^-{pr_2} & \G_1 \ar[d]^-{\varphi_0\circ t} \\
\h_1 \ar[r]^-{s} & \h_0,\\}$$
equipped with the left $\h$-action along $s \circ pr_1$ given by $$h \cdot \left(\gamma,g\right)=\left(h\gamma,g\right),$$
and that the source and target maps are given by $$s\left(h,g\right)=\left(h\varphi\left(g\right),s\left(g\right)\right),$$ and $$t\left(h,g\right)=\left(h,t\left(g\right)\right).$$
Viewing the arrows of $P\left(\varphi\right)$ as a sheaf, they assign $U$ the set $$\Hom\left(m_U,P\left(\varphi\right)_1\right).$$
Let $$\pi\left(U\right):\bar \Gamma\left(\left[\varphi\right] \circ a\right)\left(U\right) \times_{\bar\Gamma\left(\left[\varphi\right]\right)\left(U\right)}\bar \Gamma\left(\left[\varphi\right] \circ a\right)\left(U\right) \to \Hom\left(m_U,P\left(\varphi\right)_1\right)$$ be the map that sends $$\zeta:=\left(\left(\sigma_0,\alpha_0\right),\left(\sigma_1,\alpha_1\right),\beta\right)$$ to the morphism
\begin{eqnarray*}
\theta\left(\zeta\right):m_U &\to& \h_1 \times_{\h_0} \G_1\\
\gamma &\mapsto& \left(\alpha_1\left(\gamma\right),\beta\left(\gamma\right)\right).\\
\end{eqnarray*}
It is easy to check that this morphism is $\h$-equivariant, hence is a map in $\B\h$. We will show that under the identification $$\bar \Gamma\left(\left[\varphi\right] \circ a\right)\left(U\right) \cong P\left(\varphi \circ \tilde a\right)\left(U\right)=\Hom\left(m_U,P\left(\varphi \circ \tilde a\right)\right),$$ $\pi\left(U\right)$ respects source and targets. Indeed, suppose we start with a triple $$\zeta:=\left(\left(\sigma_0,\alpha_0\right),\left(\sigma_1,\alpha_1\right),\beta\right).$$ By Lemma \ref{lem:sec1}, each $\left(\sigma_i,\alpha_i\right)$ corresponds to an element of $$\bar \Gamma\left(\left[\varphi\right] \circ a\right)\left(U\right),$$ which in turn corresponds to a morphism

\begin{eqnarray}\label{eq:choochootrain}
m_U=s^{-1}\left(U\right) &\to& \h_1\times_{\h_0}T\nonumber\\
\gamma &\mapsto&\left(\alpha_i\left(\gamma\right),\sigma_i\left(\gamma\right)\right)\nonumber\\
\end{eqnarray}
in $\B\h$. Now $\pi\left(U\right)\left(\zeta\right)$ is a map from $d_0\pi\left(U\right)\left(\zeta\right)$ to $d_1\pi\left(U\right)\left(\zeta\right),$ where we have used simplicial notation for the source and target. For each $i$, we have a map $$m_U \stackrel{\pi\left(U\right)\left(\zeta\right)}{-\!\!\!-\!\!\!-\!\!\!-\!\!\!-\!\!\!\longrightarrow} \h_1 \times_{\h_0} \G_1 \stackrel{d_i}{\longrightarrow} \h_1 \times_{\h_0} \G_0,$$ which we may interpret as an element of $$P\left(\varphi\right)_0\left(U\right)=P\left(\varphi \circ \tilde a\right)\left(U\right).$$ From equation (\ref{eq:choochootrain}) and the definition of the source and target map, it follows that

\begin{eqnarray*}
s\pi\left(U\right)\left(\zeta\right)&=& \gamma \mapsto s\left(\alpha_1\left(\gamma\right),\beta\left(\gamma\right)\right)\\
&=&\gamma \mapsto \left(\alpha_1\varphi\beta\left(\gamma\right),s\beta\left(\gamma\right)\right)\\
&=&\gamma \mapsto \left(\alpha_0\left(\gamma\right),\sigma_0\left(\gamma\right)\right),\\
\end{eqnarray*}
and
\begin{eqnarray*}
t\pi\left(U\right)\left(\zeta\right)&=& \gamma \mapsto t\left(\alpha_1\left(\gamma\right),\beta\left(\gamma\right)\right)\\
&=&\gamma \mapsto \left(\alpha_1\left(\gamma\right),t\beta\left(\gamma\right)\right)\\
&=&\gamma \mapsto \left(\alpha_1\left(\gamma\right),\sigma_1\left(\gamma\right)\right).\\
\end{eqnarray*}
Hence $\pi\left(U\right)$ respects the source and target. We will now show it is an isomorphism. Suppose we are given an arbitrary equivariant map $$\theta:m_U \to \h_1 \times_{\h_0} \G_1.$$ Denote its components by $$\theta\left(\gamma\right)=\left(h\left(\gamma\right),g\left(\gamma\right)\right).$$ Since $\theta$ is $\h$-equivariant, it follows that $h$ is $\h$-equivariant and $g$ is $\h$-invariant. Now
\begin{eqnarray*}
s \circ \theta:m_u &\to& P\left(\varphi\right)_0\\
s\theta\left(\gamma\right)&=&\left(h\left(\gamma\right)\varphi\left(g\left(\gamma\right)\right),s\left(g\left(\gamma\right)\right)\right)\\
\end{eqnarray*}
and
\begin{eqnarray*}
t \circ \theta:m_u &\to& P\left(\varphi\right)_0\\
t\theta\left(\gamma\right)&=&\left(h\left(\gamma\right),t\left(g\left(\gamma\right)\right)\right).\\
\end{eqnarray*}
Each of these maps correspond to an element in $$P\left(\varphi\right)_0\left(U\right)\cong \bar \Gamma\left(\left[\varphi\right] \circ a\right)\left(U\right).$$ By Lemma \ref{lem:sec1}, we know that $s \circ \theta$ corresponds to the morphism of groupoids
\begin{equation*}
\widehat{s\circ\theta}:\h \ltimes s^{-1}\left(U\right)=s^{-1}\left(U\right) \to \G_0
\end{equation*}
given on objects as $$\gamma \mapsto s\left(g\left(\gamma\right)\right),$$ together with a $2$-cell $$\alpha_{s\theta}:\left[\varphi\right] \circ a \circ \widehat{s\circ\theta} \Rightarrow \theta_{m_U},$$ given by $$\alpha_{s\theta}=pr_1 \circ s \circ \theta.$$ Explicitly we have:
$$\alpha_{s\theta}\left(\gamma\right)=h\left(\gamma\right)\varphi\left(g\left(\gamma\right)\right).$$
Similarly, we know that $t \circ \theta$ corresponds to the morphism
\begin{equation*}
\widehat{t\circ\theta}:\h \ltimes s^{-1}\left(U\right)=s^{-1}\left(U\right) \to \G_0
\end{equation*}
given on objects as $$\gamma \mapsto t\left(g\left(\gamma\right)\right),$$ together with a $2$-cell $$\alpha_{t\theta}:\left[\varphi\right] \circ a \circ \widehat{t\circ\theta} \Rightarrow \theta_{m_U},$$ given by $$\alpha_{t\theta}=pr_1 \circ t \circ \theta,$$ and we have:
$$\alpha_{t\theta}\left(\gamma\right)=h\left(\gamma\right).$$
The map $$\beta\left(\theta\right):=pr_2\circ \theta:s^{-1}\left(U\right) \to \G_1$$ which assigns $\gamma \mapsto g\left(\gamma\right)$ encodes a natural transformation $$\beta\left(\theta\right):\widehat{s\circ\theta} \Rightarrow \widehat{s\circ\theta}.$$ Moreover, we have that

\begin{eqnarray*}
\alpha_{t\theta}\varphi\beta\left(\gamma\right)&=&h\left(\gamma\right)\circ\varphi\left(g\left(\gamma\right)\right)\\
&=&\alpha{s\theta}\left(\gamma\right),\\
\end{eqnarray*}
which implies the diagram (\ref{diag:tri}) commutes.

Define a map
$$\Xi\left(U\right):\Hom\left(m_U,P\left(\varphi\right)_1\right) \to \bar \Gamma\left(\left[\varphi\right] \circ a\right)\left(U\right) \times_{\bar\Gamma\left(\left[\varphi\right]\right)\left(U\right)}\bar \Gamma\left(\left[\varphi\right] \circ a\right)\left(U\right)$$
which assigns the morphism $\theta:m_U \to P\left(\varphi\right)_1$ the triple $$\left(\left(\widehat{s\circ\theta},\alpha_{s\theta}\right),\left(\widehat{t\circ\theta},\alpha_{t\theta}\right),\beta\left(\theta\right)\right).$$
This map is clearly inverse to $\pi$. We leave it to the reader to check that $\pi\left(U\right)$ respects composition. It then follows that the groupoids in sheaves $$\bar \Gamma\left(\left[\varphi\right] \circ a\right) \times_{\bar\Gamma\left(\left[\varphi\right]\right)}\bar \Gamma\left(\left[\varphi\right] \circ a\right)$$ and $P\left(\varphi\right)$ are isomorphic.
\end{proof}

\section{Effective Stacks}\label{sec:effective}
\subsection{Basic definitions}
In this section, we recall a special class of \'etale stacks, called effective \'etale stacks. This is a summary of results well known in the groupoid literature, expressed in a more stack-oriented language. We claim no originality for the ideas. We start with defining effectiveness for orbifolds, as this definition is more intuitive. This will also make the general definition for an arbitrary \'etale stack more clear.

\begin{rmk}
For the rest of the paper, we shall assume that the category $S$ of spaces is either manifolds or topological spaces, as it will be convenient to work point-set theoretically in many of the following proofs. The results proven do hold for locales as well, once phrased in a point-free way, but require a different proof.
\end{rmk}

\begin{dfn}
A differentiable stack $\X$ is called an \textbf{differentiable orbifold} if it is \'etale and the diagonal map $\Delta:\X \to \X \times \X$ is proper. If $\X$ is instead a topological stack, we call $\X$ a \textbf{topological orbifold}. To simplify things, we will refer to differentiable orbifolds and topological orbifolds, simply as orbifolds.
\end{dfn}

\begin{rmk}
We should explain what we mean in saying that the diagonal map is proper. In the differentiable setting, this map is not representable, even for manifolds. We say that a map of $f:\X \to \Z$ between differentiable stacks is proper if and only if for any \emph{representable} map $M \to \Z$ from a manifold, the induced map $M \times_{\Z} \X \to M$ is a proper map of manifolds. Equivalently, as properness is a topological property, and the diagonal map of any topological stack is representable \cite{NoohiF}, (and proper maps are invariant under restriction and local on the target) stating that the diagonal of a differentiable stack is proper in the above sense is equivalent to saying that the diagonal of the underlying topological stack is a representable proper map. Yet another characterization is viewing $\X$ and $\X \times \X$ as etendue and asking the map to be a proper map of topoi in the sense of \cite{elephant2}.
\end{rmk}

\begin{dfn}
An $S$-groupoid is an \textbf{orbifold groupoid} if it is \'etale and \textbf{proper}, i.e. the map $$\left(s,t\right):\h_1 \to \h_0 \times \h_0$$ is proper.
\end{dfn}

\begin{prop}
$\X$ is an orbifold if and only if there exists an orbifold groupoid $\h$ such that $\X \simeq \left[\h\right]$.
\end{prop}

\begin{proof}
For any \'etale $\h$ such that $\left[\h\right] \simeq \X$,
$$\xymatrix{\h_1 \ar[d]_{\left(s,t\right)} \ar[r] & \X \ar[d]^{\Delta}\\
\h_0 \times \h_0 \ar[r]^{a \times a} & \X \times \X}$$
is a weak pullback diagram, where $a:\h_0 \to \X$ is the atlas associated to $\h$.
\end{proof}

Recall the following definition:
\begin{dfn}
If $G$ is a $S$-group acting on a space $X$, the action is \textbf{effective} (or faithful), if $\bigcap\limits_{x \in X} G_x=e$, i.e., for all non-identity elements $g \in G$, there exists a point $x \in X$ such that $g \cdot x \neq x.$ Equivalently, the induced homomorphism $$\rho:G \to \mathit{Diff}\left(X\right),$$ where $\mathit{Diff}\left(X\right)$ is the group of diffeomorphisms (homeomorphisms) of $X$ is a monomorphism. (These two definitions are equivalence since $$Ker\left(\rho\right)=\bigcap\limits_{x \in X} G_x=e.)$$
\end{dfn}
If $\rho$ above has a non-trivial kernel $K,$ then there is an inclusion of $K$ into each isotropy group of the action, or equivalently into each automorphism group of the quotient stack (the stack associated to the action groupoid). In this case $K$ is ``tagged-along'' as extra data in each automorphism group. Each of these copies of $K$ is the kernel of the induced homomorphism $$\left(\rho\right)_x:Aut\left(\left[x\right]\right) \to \mathit{Diff}\left(M\right)_x,$$ where $\mathit{Diff}\left(M\right)_x$ is the group of diffeomorphisms of $M$ which fix $x$. In the \emph{differentiable} setting, when $G$ is \emph{finite}, these kernel are called the \emph{ineffective isotropy groups} of the associated \'etale stack. In this case, the \emph{effective part} of this stack is the stacky-quotient of $X$ by the induced action of $G/K$. This latter stack has only trivial ineffective isotropy groups.

\begin{rmk}
If $G$ is not finite, this notion of ineffective isotropy group may not agree with Definition \ref{dfn:ineffgp}, since non-identity elements can induce the germ of the identity around a point. It the topological setting, this problem can occur even when $G$ is finite.
\end{rmk}

\begin{prop}\label{prop:orbeff}
Suppose $\X$ is an orbifold and $x:* \to \X$ is a point. Then there exists a local homeomorphism $p:V_x \to \X$ from a space $V_x$ such that:
\begin{itemize}
\item[i)] the point $x$ factors (up to isomorphism) as $* \stackrel{\tilde x}{\longrightarrow} V_x \stackrel{p}{\longrightarrow} \X$
\item[ii)] the automorphism group $Aut\left(x\right)$ acts on $V_x$.
\end{itemize}
\end{prop}

\begin{proof}
The crux of this proof comes from \cite{orbcon}. Recall that for a point $x$ of a topological or differentiable stack $\X$, $Aut\left(x\right)$ fits into the $2$-Cartesian diagram \cite{NoohiF}

$$\xymatrix{Aut\left(x\right) \ar[d] \ar[r] & \mathrm{*} \ar[d]^-{x}\\
\mathrm{*} \ar[r]^-{x} & \X,}$$
and is a group object in spaces. If $\X \simeq \left[\h\right]$ for an $S$-groupoid $\h$, there is a point $\tilde x \in \h_0$ such that $x \cong a \circ \tilde x,$ where $a:\h_0 \to \X$ is the atlas associated to the groupoid $\h,$ and moreover, $\h_{\tilde x} \cong Aut\left(\tilde x\right)$, where $\h_{\tilde x}=s^{-1}\left(\tilde x\right) \cap t^{-1}\left(\tilde x\right)$ is the $S$-group of automorphisms of $\tilde x$. (In particular, this implies that if $\X$ is \'etale, then $Aut\left(x\right)$ is discrete for all $x$.) Suppose now that $\X$ and $\h$ are \'etale. Then for each $h \in \h_{\tilde x},$ there exists an open neighborhood $U_h$ such that the two maps $$s:U_h \to s\left(U_h\right)$$ $$t:U_h \to t\left(U_h\right)$$ are homeomorphisms. Now, suppose that $\X$ is in fact an orbifold (so that $\h$ is an orbifold groupoid). Then, it follows that $\h_{\tilde x}$ is finite. Given $f$ and $g$ in $\h_{\tilde x}$, we can find a small enough neighborhood $W$ of $\tilde x$ in $\h_0$ such that for all $z$ in $W$, $$t \circ s^{-1}|_{U_g}\left(z\right) \in s\left(U_f\right),$$ $$t \circ s^{-1}|_{U_g} \left(t \circ s^{-1}|_{U_f}\left(z\right) \right) \in W,$$ and

\begin{equation}\label{eq:391}
s^{-1}|_{U_g}\left(z\right) \cdot s^{-1}|_{U_f}\left(z\right) \in U_{gf}.
\end{equation}
In this case, by plugging in $z=\tilde x$ in (\ref{eq:391}), we see that (\ref{eq:391}) as a function of $z$ must be the same as $$s^{-1}|_{U_{gf}}.$$ Therefore, on $W$, the following equation holds
\begin{equation}\label{eq:lnof5}
t \circ s^{-1}|_{U_g} \left(t \circ s^{-1}|_{U_f}\left(z\right) \right)= t \circ s^{-1}|_{U_{gf}}.
\end{equation}
Since $\h_{\tilde x}$ is finite, we may shrink $W$ so that equation (\ref{eq:lnof5}) holds for all composable arrows in $\h_{\tilde x}$. Let $$V_x:=\bigcap\limits_{h \in \h_{\tilde x}} \left(t \circ s^{-1}|_{U_f}\left(W\right)\right).$$ Then, for all $h \in h_{\tilde x}$, $$t \circ s^{-1}|_{U_h}\left(V_x\right)=V_x.$$ So $t \circ s^{-1}|_{U_h}$ is a homeomorphism from $V_x$ to itself for all $x$, and since equation (\ref{eq:lnof5}) holds, this determines an action of $\h_{\tilde x}\cong Aut\left(x\right)$ on $V_x$. Finally, define $p$ to be the atlas $a$ composed with the inclusion $V_x \hookrightarrow \h_0$
\end{proof}

\begin{dfn}
An orbifold $\X$ is an \textbf{effective} orbifold, if the actions of $Aut\left(x\right)$ on $V_x$ as in the previous lemma can be chosen to be effective.
\end{dfn}

The finiteness of the stabilizer groups played a crucial role in the proof of Proposition \ref{prop:orbeff}. Without this finiteness, one cannot arrange (in general) for even a single arrow in an \'etale groupoid to induce a self-diffeomorphism of an open subset of the object space. Additionally, even if each arrow had such an action, there is no guarantee that the (infinite) intersection running over all arrows in the stabilizing group of these neighborhoods will be open. Hence, for a general \'etale groupoid, the best we can get is a germ of a locally defined diffeomorphism. It is using these germs that we shall extend the definition of effectiveness to arbitrary \'etale groupoids and stacks.

Given a space $X$ and a point $x \in X$, let $\mathit{Diff}_x\left(X\right)$ denote the group of germs of (locally defined) diffeomorphisms (homeomorphisms if $X$ is a topological space) that fix $x$.

\begin{prop}
Let $\X$ be an \'etale stack and pick an \'etale atlas $$V \to \X.$$ Then for each point $x:* \to \X$,
\begin{itemize}
\item[i)] the point $x$ factors (up to isomorphism) as $* \stackrel{\tilde x}{\longrightarrow} V\stackrel{p}{\longrightarrow} \X,$ and
\item[ii)] there is a canonical homomorphism $Aut\left(x\right) \to \mathit{Diff}_{\tilde x}\left(V\right)$.
\end{itemize}
\end{prop}

\begin{proof}
Following the proof of Proposition \ref{prop:orbeff}, let $V=\h_0$ and let the homomorphism send each $h \in \h_x$ to the germ of $t \circ s^{-1}|_{U_h}$, which is a locally defined diffeomorphism of $V$ fixing $\tilde x$.
\end{proof}

\begin{cor}
For $\h$ an \'etale $S$-groupoid, for each $x \in \h_0$, there exists a canonical homomorphism of groups $\h_x \to \mathit{Diff}_{x}\left(\h_0\right)$.
\end{cor}

\begin{dfn}\label{dfn:ineffgp}
Let $x$ be a point of an \'etale stack $\X$. The \textbf{ineffective isotropy group} of $x$ is the kernel of the induced homomorphism $$Aut\left(x\right) \to \mathit{Diff}_{\tilde x}\left(V\right).$$ Similarly for $\h$ an \'etale groupoid.
\end{dfn}

\begin{dfn}
An \'etale stack $\X$ is \textbf{effective} if the ineffective isotropy group of each of its points is trivial. Similarly for $\h$ an \'etale groupoid.
\end{dfn}

\begin{prop}
An orbifold $\X$ is an effective orbifold if and only if it is effective when considered as an \'etale stack.
\end{prop}

\begin{proof}
This follows from \cite{Fol}, Lemma 2.11.
\end{proof}

\begin{dfn}
Let $X$ be a space. Consider the presheaf $$\mathit{Emb}:\mathcal{O}\left(X\right)^{op} \to \Set,$$ which assigns an open subset $U$ the set of embeddings of $U$ into $X$. Denote by $\Ha\left(X\right)_1$ the total space  of the \'etal\'e space of the associated sheaf. Denote the map to $X$ by $s$. The stalk at $x$ is the set of germs of locally defined diffeomorphisms (which no longer need to fix $x$). If $germ_x\left(f\right)$ is one such germ, the element $f\left(x\right) \in X$ is well-defined. We assemble this into a map $$t:\Ha\left(X\right)_1 \to X.$$ This extends to a natural structure of an \'etale $S$-groupoid $\Ha\left(X\right)$ with objects $X$, called the \textbf{Haefliger groupoid} of $X$.
\end{dfn}

\begin{rmk}
In literature, the Haefliger groupoid is usually denoted by $\Gamma\left(X\right)$, but, we wish to avoid the clash of notation with the stack of sections $2$-functor.
\end{rmk}

\begin{prop}
For $\h$ an \'etale $S$-groupoid, there is a canonical map $$\tilde \iota_{\h}:	\h \to \Ha\left(\h_0\right).$$
\end{prop}
\begin{proof}
For each $h \in \h_1$, choose a neighborhood $U$ such that $s$ and $t$ restrict to embeddings. Then $h$ induces a homeomorphism $$s\left(h\right) \in s\left(U\right) \to t\left(U\right) \ni t\left(h\right),$$ namely $t\circ s^{-1}|_{U}$. Define $\tilde \iota_{\h}$ by having it be the identity on objects and having it send an arrow $h$ to the germ at $s\left(h\right)$ of $t\circ s^{-1}|_{U}$. This germ clearly does not depend on the choice of $U$.
\end{proof}

The following proposition is immediate:

\begin{prop}
An \'etale $S$-groupoid $\h$ is effective if and only if $\tilde \iota_{\h}$ is faithful.
\end{prop}

\begin{dfn}
Let $\h$ be an \'etale $S$-groupoid. The \textbf{effective part} of $\h$ is the image in $\Ha\left(\h_0\right)$ of $\tilde \iota_{\h}$. It is denoted by  $\Ef\left(\h\right)$. This is an open subgroupoid, so it is clearly effective and \'etale. We will denote the canonical map $\h \to \Ef \h$ by $\iota_{\h}$.
\end{dfn}

\begin{rmk}
$\h$ is effective if and only if $\iota_{\h}$ is an isomorphism.
\end{rmk}

\subsection{\'Etale invariance}
Unfortunately, the assignment $\h \mapsto \Ef\left(\h\right)$ is not functorial with respect to all maps, that is, a morphism of \'etale $S$-groupoids need not induce a morphism between their effective parts. However, there are classes of maps for which this assignment is indeed functorial. In this subsection, we shall explore this functoriality.
\begin{dfn}
Let $P$ be a property of a map of spaces which forms a subcategory of $S$. We say that $P$ is \textbf{\'etale invariant} if the following two properties are satisfied:
\begin{itemize}
\item[i)] $P$ is stable under pre-composition with local homeomorphisms
\item[ii)] $P$ is stable under pullbacks along local homeomorphisms.
\end{itemize}
If in addition, every morphism in $P$ is open, we say that $P$ is a class of \textbf{open \'etale invariant} maps. Examples of such open \'etale invariant maps are open maps, local homeomorphisms, or, in the smooth setting, submersions. We say a map $\psi:\G \to \h$ of \'etale $S$-groupoids has property $P$ if both $\psi_0$ and $\psi_1$ do. We denote corresponding $2$-category of $S$-groupoids as $\left(S-Gpd\right)^{et}_{P}$. We say a morphism $$\varphi:\Y \to \X$$ has property $P$ is there exists a homomorphism of \'etale $S$-groupoids $$\psi:\G \to \h$$ with property $P$, such that $$\varphi \cong \left[\psi\right].$$
\end{dfn}

Warning: Do not confuse notation with $\left(S^{et}-Gpd\right)$, the $2$-category of \'etale $S$-groupoids and only local homeomorphisms. This only agrees with $\left(S-Gpd\right)^{et}_{P}$ when $P$ is local homeomorphisms.

\begin{rmk}
This agrees with our previous definition of a local homeomorphism of \'etale stacks in the case $P$ is local homeomorphisms. When $P$ is open maps, under the correspondence between \'etale stacks and \'etendues, this agrees with the notion of an open map of topoi in the sense of \cite{elephant2}. When $P$ is submersions, this is equivalent to the definition of a submersion of smooth \'etendues given in \cite{Ie}.
\end{rmk}

\begin{rmk}
Notice that being \'etale invariant implies being invariant under restriction and local on the target, as in Definition \ref{dfn:local}.
\end{rmk}

\begin{prop}
Let $P$ be a property of a map of spaces which forms a subcategory of $S$. $P$ is \'etale invariant if and only if the following conditions are satisfied:
\begin{itemize}
\item[i)] every local homeomorphism is in $P$
\item[ii)] for any commutative diagram
$$\xymatrix{W \ar[d]_-{g} \ar[r]^-{f'} & Y \ar[d]^-{g'}\\
X \ar[r]^-{f} & Z,}$$
with both $g$ and $g'$ local homeomorphisms, if $f$ has property $P$, then so does $f'$.
\end{itemize}
\end{prop}

\begin{proof}
Suppose that $P$ is \'etale invariant. Then, as $P$ is a subcategory, it contains all the identity arrows, and since it is stable under pre-composition with local homeomorphisms, this implies that every local homeomorphism is in $P$. Now suppose that $f \in P$, and $$\xymatrix{W \ar[d]_-{g} \ar[r]^-{f'} & Y \ar[d]^-{g'}\\
X \ar[r]^-{f} & Z,}$$ is commutative with both $g$ and $g'$ local homeomorphisms. Then as $P$ is stable under pullbacks along local homeomorphisms, the induced map $X \times_{Z} Y \to Y$ has property $P$. Moreover, as local homeomorphisms are invariant under change of base (Definition \ref{dfn:local}), the induced map $X \times_{Z} Y \to X$ is a local homeomorphism. It follows that the induced map $W \to X \times_{Z} Y$ is a local homeomorphism, and since $f'$ can be factored as $$W \to X \times_{Z} Y \to Y,$$ and $P$ is stable under pre-composition with local homeomorphisms, it follows that $f'$ has property $P$.
Conversely, suppose that the conditions of the proposition are satisfied. Condition $ii)$ clearly implies that $P$ is stable under pullbacks along local homeomorphisms. Suppose that $e:W \to X$ is a local homeomorphism and $f:X \to Z$ is in $P$. Then as
$$\xymatrix{W \ar[d]_-{e} \ar[r]^-{f \circ e} & Z \ar[d]^-{id_Z}\\
X \ar[r]^-{f} & Z,}$$
commutes, it follows  that the $f \circ e$ has property $P$.
\end{proof}

\begin{lem}\label{lem:etin1}
For any open \'etale invariant $P$, the assignment $$\h \mapsto \Ef\left(\h\right)$$ extends to a $2$-functor

$$\Ef_P:\left(S-Gpd\right)^{et}_{P} \to \Ef\left(S-Gpd\right)^{et}_{P}$$
from \'etale $S$-groupoids and $P$-morphisms to effective \'etale $S$-groupoids and $P$-morphisms.
\end{lem}

\begin{proof}
Suppose $\varphi:\G \to \h$ has property $P$. Since $\Ef$ does not affect objects, we define $$\Ef\left(\varphi\right)_0=\varphi_0.$$ Given $g \in \G_1$, denote its image in $\Ef\left(\G\right)_1$ by $[g]$. Define $$\Ef\left(\varphi\right)_1\left([g]\right)=\left[\varphi\left(g\right)\right].$$ We need to show that this is well defined. Suppose that $[g]=[g']$. Let $V_g$ and $V_{g'}$ be neighborhoods of $g$ and $g'$ respectively, on which both $s$ and $t$ restrict to embeddings. Denote by $x$ the source of $g$ and $g'$. Then there exists a neighborhood $W$ of $x$ over which $$t \circ s^{-1}|_{g}$$ and $$t \circ s^{-1}|_{g'}$$ agree. Since $\varphi_1$ has property $P$, it is open, so $\varphi_1\left(V_g\right)$ is a neighborhood of $\varphi_1\left(g\right)$, and similarly for $g'$. By making $V_g$ and $V_g'$ smaller if necessary, we may assume that $s$ and $t$ restrict to embeddings on $\varphi_1\left(V_g\right)$ and $\varphi_1\left(V_g'\right)$. Since $\varphi$ is a groupoid homomorphism, it follows that $$t \circ s^{-1}|_{\varphi_1\left(V_g\right)}$$ and  $$\varphi_0 \circ t \circ s^{-1}|_{V_g}$$ agree on $W$, and similarly for $g'$. Hence, if $g$ and $g'$ induce the same germ of a locally defined homeomorphism, so do $\varphi_1\left(g\right)$ and $\varphi_1\left(g'\right)$. It is easy to check that $\Ef\left(\varphi\right)$ as defined is a homomorphism of $S$-groupoids. In particular, the following diagram commutes:

$$\xymatrix@C=1.5cm{\Ef\left(\G\right)_1 \ar[d]_-{s} \ar[r]^-{\Ef\left(\varphi\right)_1} & \Ef\left(\h\right) \ar[d]^-{s}\\
\G_0 \ar[r]^-{\varphi_0} & \h_0.}$$
Since $P$ is \'etale invariant and the source maps are local homeomorphisms, it implies that $\Ef\left(\varphi\right)_1$ has property $P$. The rest is proven similarly.
\end{proof}

\begin{thm}\label{thm:bungalo}
Let $j_P:\Ef\left(S-Gpd\right)^{et}_{P} \hookrightarrow \left(S-Gpd\right)^{et}_{P}$ be the inclusion. Then $\Ef_P$ is left-adjoint to $j_P$.
\end{thm}

\begin{proof}
There is a canonical natural isomorphism $$\Ef_P \circ j_P \Rightarrow id_{\Ef\left(S-Gpd\right)^{et}_{P}}$$ since any effective \'etale groupoid is canonically isomorphic to its effective part. Furthermore, the maps $\iota_{\h}$ assemble into a natural transformation $$\iota:id_{\left(S-Gpd\right)^{et}_{P}} \Rightarrow j_P \circ \Ef_P.$$ It is easy to check that these define a $2$-adjunction.
\end{proof}

\begin{thm}
$\Ef_P$ sends Morita equivalences to Morita equivalences.
\end{thm}
\begin{proof}
Suppose $\varphi:\G \to \h$ is a Morita equivalence. Since $\G$ and $\h$ are \'etale, this implies $\varphi$ is a local homeomorphism. Hence, in the pullback diagram
$$\xymatrix{\h_1\times_{\h_0}\G_0 \ar[d]_-{pr_1} \ar[r]^-{pr_2} & \G_0 \ar[d]^-{\varphi_0}\\
\h_1 \ar[r]^-{s} & \h_0,}$$
$pr_1$ is a local homeomorphism, and hence the map $$t \circ pr_1:\h_1 \times_{\h_0} \G_0 \to \h_0$$ is as well. We have a commutative diagram

$$\xymatrix{\h_1 \times_{\h_0} \G_0 \ar[r]^{t \circ pr_1} \ar[d] & \h_0\\
\Ef\left(\h\right)_1 \times_{\h_0} \G_0. \ar[ur]_-{t \circ pr_1}  &   }$$
The map $$\Ef\left(\h\right)_1 \times_{\h_0} \G_0 \to \h_1\times_{\h_0} \G_0 $$ is the pullback of a local homeomorphism, hence one itself, and the upper arrow $t \circ pr_1$ is a local homeomorphism. This implies $$t \circ pr_1:\Ef\left(\h\right)_1 \times_{\h_0} \G_0 \to \h_0$$ is a local homeomorphism as well. In particular, it admits local sections, and if $S$ is manifolds, is a surjective submersion. Therefore $\Ef\left(\varphi\right)$ is essentially surjective.
Now suppose that $$\left[h\right]:\varphi\left(x\right) \to \varphi\left(y\right).$$ Then  $$h:\varphi\left(x\right) \to \varphi\left(y\right).$$ So there is a unique $g:x \to y$ such that $\varphi\left(g\right)=h$. Now suppose $$\left[h\right]=\left[h'\right].$$ We can again choose a unique $g'$ such that $\varphi\left(g'\right)=h'.$ We need to show that $$\left[g\right]=\left[g'\right].$$ Let $V_g$ and $V_{g'}$ be neighborhoods of $g$ and $g'$ respectively chosen so small that $s$ and $t$ of $\G_1$ restrict to embeddings on them, $s$ and $t$ of $\h_1$ restrict to embeddings on $\varphi_1\left(V_g\right)$ and $\varphi_1\left(V_g'\right),$ $\varphi_0$ restricts to an embedding on $s\left(V_g\right),$ which is possible since $\varphi_0$ is a local homeomorphism, and $$t \circ s^{-1}|_{\varphi_1\left(V_g\right)}$$ and $$t\circ s^{-1}|_{\varphi_1\left(V_g'\right)}$$ agree on $\varphi_0\left(s\left(V_g\right)\right),$ which is possible since $$\left[\varphi\left(g\right)\right]=\left[\varphi\left(g'\right)\right].$$ Then by the proof of Lemma \ref{lem:etin1}, $$t \circ s^{-1}|_{\varphi_1\left(V_g\right)}$$ and  $$\varphi_0 \circ t \circ s^{-1}|_{V_g}$$ agree on $s\left(V_g\right)$, and similarly for $g'$. Hence $$\varphi_0 \circ t \circ s^{-1}|_{V_g}$$ and $$\varphi_0 \circ t \circ s^{-1}|_{V_g'}$$ agree on $W$, but $\varphi_0$ is an embedding when restricted to $W$, hence $$t \circ s^{-1}|_{V_g}$$ and $$t \circ s^{-1}|_{V_g'}$$ agree on $W$ so $\left[g\right]=\left[g'\right].$
\end{proof}

\begin{lem}
Let $\mathcal{U}$ be an \'etale cover of $\h_0$, with $\h$ an \'etale $S$-groupoid. Then there is a canonical isomorphism between $\Ef\left(\h_{\mathcal{U}}\right)$ and $\left(\Ef\left(\h\right)\right)_{\mathcal{U}}$ (See Definition \ref{dfn:cech}).
\end{lem}
\begin{proof}
Both of these groupoids have the same object space. It suffices to show that their arrow spaces are isomorphic (and that this determines an internal functor). Suppose the cover $\mathcal{U}$ is given by a local homeomorphism $e:U \to \h_0$. An arrow in $\h_{\mathcal{U}}$ is a triple $$\left(h,p,q\right)$$ with $$h:e\left(p\right) \to e\left(p\right).$$ An arrow in $\left(\Ef\left(\h\right)\right)_{\mathcal{U}}$ is a triple $$\left(\left[h\right],p,q\right)$$ such that $\left[h\right]$ is the image of an arrow $h \in \h_1$ under $\iota_\h$ such that $$h:e\left(p\right) \to e\left(p\right).$$ Define a map
\begin{align}
\left(\h_{\mathcal{U}}\right)_1 &\to  \left(\left(\Ef\left(\h\right)\right)_{\mathcal{U}}\right)_1 \nonumber\\
\left(h,p,q\right) &\mapsto \left(\left[h\right],p,q\right)\label{eq:cake}.\end{align}
This map is clearly surjective.

We make the following claim:

$$\left[h\right]=\left[h'\right]$$
if and only if
$$\left[\left(h,p,q\right)\right]=\left[\left(h',p,q\right)\right].$$
Suppose that $$\left[h\right]=\left[h'\right].$$
Pick a neighborhood $U_h$ of $h$ in $\h_1$ such that both $s$ and $t$ are injective over it, and $U_h'$ an analogous neighborhood of $h'$. Let $W$ be a neighborhood of $s\left(h\right)=s\left(h'\right)$ over which
\begin{equation}\label{eq:yikes}
t\circ s|_{U_h}^{-1}=t\circ s|_{U_h'}^{-1}.
\end{equation}
Pick neighborhoods $V_p$ and $V_q$ of $p$ and $q$ respectively so small that $e$ is injective over them, and for all $a \in V_p$, $$e\left(a\right) \in W$$ and $$t\circ s|_{U_h}^{-1}\left(e\left(a\right)\right) \in e\left(V_q\right).$$
As the arrow space $\left(\h_{\mathcal{U}}\right)_1$ fits into the pullback diagram $$\xymatrix{\left(\h_{\mathcal{U}}\right)_1 \ar[r] \ar[d]_-{\left(s,t\right)} & \h_1 \ar[d]^-{\left(s,t\right)} \\
U \times U \ar[r] & \h_0 \times \h_0,}$$
$\left(V_p \times V_q \times U_h \right) \cap \left(\h_{\mathcal{U}}\right)_1$ is a neighborhood of $\left(h,p,q\right)$ over which both the source and target maps are injective. The set $\left(V_p \times V_q \times U_h' \right) \cap \left(\h_{\mathcal{U}}\right)_1$ is an analogous neighborhood of $\left(h',p,q\right)$. The local inverse of $s$ through $\left(h,p,q\right)$ is then given by
$$a \mapsto \left(s|_{U_h}^{-1}\left(e\left(a\right)\right),a,e|_{V_p}^{-1}\left(t\circ s|_{U_h}^{-1}\left(e\left(a\right)\right)\right)\right).$$
Hence, the germ associated to $\left(h,p,q\right)$ is the germ of $$a \mapsto e|_{V_p}^{-1}\left(t\circ s|_{U_h}^{-1}\left(e\left(a\right)\right)\right).$$ Similarly the germ associated to $\left(h',p,q\right)$ is the germ of $$a \mapsto e|_{V_p}^{-1}\left(t\circ s|_{U_h'}^{-1}\left(e\left(a\right)\right)\right).$$ From equation (\ref{eq:yikes}), it follows that these maps are identical. Moreover, supposing instead that $$\left[\left(h,p,q\right)\right]=\left[\left(h',p,q\right)\right],$$ by the above argument, it follows that $\left[h\right]=\left[h'\right]$ since $e$ is injective over $V_q$.

Hence the assignment (\ref{eq:cake}) depends only  on the  image of $\left(h,p,q\right)$ in $\Ef\left(\h_{\mathcal{U}}\right)$. So there is an induced well defined and surjective map
\begin{equation}\label{eq:turtle}
\left(\Ef\left(\h_{\mathcal{U}}\right)\right)_1 \to  \left(\left(\Ef\left(\h\right)\right)_{\mathcal{U}}\right)_1.
\end{equation}
Since, $\left[h\right]=\left[h'\right]$ implies $\left[\left(h,p,q\right)\right]=\left[\left(h',p,q\right)\right],$ it follows  that this map is also injective, hence bijective. It is easy to check that it is moreover a homeomorphism. It clearly defines a groupoid homomorphism
\end{proof}

\begin{cor}\label{cor:etadj}
There is an induced $2$-adjunction
$$\xymatrix{\Eft_{P}  \ar@<-0.5ex>[r]_-{j_P}  & \Et_{P}\ar@<-0.5ex>[l]_-{\Ef_P},}$$
between \'etale stacks with $P$-morphisms and effective \'etale stacks with $P$-morphisms, where $j_P$ is the canonical inclusion.
\end{cor}
\begin{proof}
Let $\mathcal{U}$ be an \'etale cover of $\h_0$, with $\h$ an \'etale $S$-groupoid. From the previous lemma, there is a canonical isomorphism between $\Ef\left(\h_{\mathcal{U}}\right)$ and $\left(\Ef\left(\h\right)\right)_{\mathcal{U}}.$ Let $\G$ be an effective \'etale $S$-groupoid. Then
\begin{eqnarray*}
\Hom\left(\left[\Ef\left(\h\right)\right],\left[\G\right]\right)  &\simeq& \underset{\mathcal{U}} \hc \Hom\left(\left(\Ef\left(\h\right)\right)_{\mathcal{U}}, \G\right)\\
&\simeq& \underset{\mathcal{U}} \hc \Hom\left(\Ef\left(\h_{\mathcal{U}}\right), \G \right) \\
&\simeq& \underset{\mathcal{U}} \hc \Hom\left(\h_{\mathcal{U}},j_P\G\right)\\
&\simeq& \Hom\left(\left[\h\right],j_P\left[\G\right]\right).
\end{eqnarray*}
\end{proof}
Note that this implies that $\Eft_{P}$ is a localization of $\Et_{P}$ with respect to those morphisms whose image under $\Ef_P$ become equivalences. When $P$ is local homeomorphisms, denote $P=et$. We make the following definition for later:

\begin{dfn}
A morphism $\varphi:\Y \to \X$ between \'etale stacks is called an \textbf{effective local equivalence} if $\varphi$ is a local homeomorphism and $\Ef_{et}\left(\varphi\right)$ is an equivalence.
\end{dfn}

\section{Small Gerbes}\label{sec:gerbe}

\subsection{Gerbes}
Gerbes are a special type of stack. Gerbes were first introduced by Jean Giraud in \cite{Giraud}. Intuitively, gerbes are to stacks what groups are to groupoids. In some sense, a gerbe is ``locally'' a sheaf of groups. The most concise definition of a gerbe is:
\begin{dfn}\label{dfn:gerbe}
A \textbf{gerbe} over a Grothendieck site $\left(\C,J\right)$ is a stack $\g$ over $\C$ such that
\begin{itemize}
\item[i)] the unique map $\g \to *$ to the terminal sheaf is an epimorphism, and
\item[ii)] the diagonal map $\g \to \g \times \g$ is an epimorphism.
\end{itemize}
\end{dfn}
The first condition means that for any object $C \in \C_0$, the unique map $C \to *$ \emph{locally} factors through $\g \to *$, up to isomorphism. Spelling this out means that there exists a cover $\left(f_\alpha:C_\alpha \to C\right)$ of $C$ such that each groupoid $\g\left(C_\alpha\right)$ is non-empty. This condition is often phrased by saying $\g$ is \emph{locally non-empty}.

The second condition means that for all $C$, any map $C \to \g \times \g$ \emph{locally} factors through the diagonal $\g \to \g \times \g$ up to isomorphism. Spelling this out, any map $C \to \g \times \g$, by Yoneda, corresponds to objects $x$ and $y$ of the groupoid $\g\left(C\right)$. The fact that this map locally factors through the diagonal means that, given two such objects $x$ and $y$, there exists a cover $\left(g_\beta:C_\beta \to C\right)$ of $C$ such that for all $\beta$, $$\g\left(f_\beta\right)\left(x\right) \cong \g\left(f_\beta\right)\left(y\right)$$ in $\g\left(C_\beta\right)$. This condition is often phrased by saying $\g$ is \emph{locally connected}.

If it were not for the locality of these properties, then this would mean that each $\g\left(C\right)$ would be a non-empty and connected groupoid, hence, equivalent to a group.

\begin{dfn}
The full sub-2-categeory of $\St\left(\C\right)$ on all gerbes, is called the $2$-category of gerbes and is denoted by $\scr{Gerbe}\left(\C\right)$.
\end{dfn}

\begin{dfn}\label{dfn:bouqet}
Let $\left(\C,J\right)$ be a Grothendieck site, then a \textbf{bouquet} over $\left(\C,J\right)$ is a groupoid object in sheaves, $\bq,$ such that
\begin{itemize}
\item[i)] the canonical map $\bq_0 \to *$ to the terminal sheaf is an epimorphism, and
\item[ii)] the canonical map $\left(s,t\right):\bq_1 \to \bq_0 \times \bq_0$ is an epimorphism.
\end{itemize}
\end{dfn}

Notice the similarity of this definition with that of Definition \ref{dfn:gerbe}.

\begin{thm}\label{thm:bq}
A stack $\Z$ over $\left(\C,J\right)$ is a gerbe if and only if it is equivalent to the stack associated to a bouquet $\bq \in Gpd\left(\Sh\left(\C\right)\right)$. \cite{Duskin}
\end{thm}

\subsection{Small gerbes over an \'etale stack}
\begin{dfn}\label{dfn:smgb}
A \textbf{small gerbe} over an \'etale stack $\X$ is a small stack $\g$ over $\X$ which is a gerbe. To be more concrete, a small gerbe over $\left[\h\right]$ is a gerbe over the site $\sit\left(\h\right)$.
\end{dfn}

\begin{rmk}
Under the correspondence  between \'etale stacks and etendues, we could also define a small gerbe over $\X$ as a gerbe over the topos $\Sh\left(\X\right)$.
\end{rmk}

\begin{lem}
Let $\X$ be an \'etale  stack and let $f:\Z \to \X$ be a local homeomorphism. Then $f$ is an epimorphism in $\St\left(S\right)$ if and only if $f$ is an epimorphism when considered as a map from $\Z \to \X$ in $Et\left(\X\right)$ to the terminal object $\X \to \X$, where $Et\left(\X\right)$ is the $2$-category of local homeomorphisms over $\X$.
\end{lem}
\begin{proof}
Fix an \'etale $S$-groupoid $\h$ such that $\X \simeq \left[\h\right].$ If $f$ is an epimorphism in $\St\left(S\right)$, then any map $T \to \X$ from a space $T$ locally factors through $f$ up to isomorphism. In particular, this holds for every local homeomorphism $T \to \X$ from a space. Hence, $f$ is an epimorphism in $Et\left(\X\right)$. Conversely, suppose that $f$ is an epimorphism in $Et\left(\X\right)$. Then the atlas $a:\h_0 \to \X$ locally factors through $f$ up to isomorphism. However, every map $T \to \X$ from a space locally factors through $a$ up to isomorphism as well. It follows that $f$ is an epimorphism in $\St\left(S\right)$.
\end{proof}

\begin{cor}
Let $f:\Z \to \X$ be a local homeomorphism of \'etale stacks. Then the stack in $\St\left(\X\right)$ represented by $f$ is a gerbe over $\X$ if and only if
\begin{itemize}
\item[i)] $f$ is an epimorphism in $\St\left(S\right),$ and
\item[ii)] the induced map $\Z \to \Z \times_{\X} \Z$ is an epimorphism in $\St\left(S\right)$.
\end{itemize}
In other words, when identifying $f$ with an object of $\St\left(S/\X\right)$, it is a gerbe.
\end{cor}
\begin{proof}
It suffices to show that if $q:\Z \to \Z \times_{\X} \Z$ is an epimorphism in $Et\left(\X\right)$, then it is an epimorphism in $\St\left(S\right)$. Choose an \'etale atlas $$Y \to \Z \times_{\X} \Z$$ for $\Z \times_{\X} \Z$. Then this atlas locally factors through $q$ up to isomorphism. However, any map $T  \to \Z \times_{\X} \Z$ locally factors  through $Y$ up to isomorphism.
\end{proof}

\begin{rmk}
In the differentiable setting, this implies that the \'etal\'e realization of a small gerbe $\g$ over an \'etale stack $\X$, in particular, is a differentiable gerbe over $\X$ in the sense of \cite{diffg}, Definition 4.7.
\end{rmk}

\begin{dfn}
Let $\h$ be an \'etale $S$-groupoid. By a bouquet over $\h$, we mean a bouquet over $\sit\left(\h\right)$. Explicitly, this is a groupoid object $\bq$ in $\B\h$ such that
\begin{itemize}
\item[i)]$\mu_0:\bq_0 \to \h_0$ is surjective, and
\item[ii)]$\left(s,t\right):\bq_1 \to \bq_0\times_{\h_0}\bq_0$ is surjective.
\end{itemize}
\end{dfn}

In light of Theorem \ref{thm:bq} and the adjoint-equivalence
$$\xymatrix{\St\left(\X\right)  \ar@<-0.5ex>[r]_-{\bar L}  & Et\left(\X\right)\ar@<-0.5ex>[l]_-{\bar \Gamma}},$$
of Corollary \ref{cor:real}, we have the following corollary:

\begin{cor}
For an \'etale stack $\X \simeq \left[\h\right]$, the $2$-category of small gerbes over $\X$, $\gb,$ is equivalent to the full sub-2-category of those local homeomorphisms $\Z \to \X$ in $Et\left(\X\right)$ which are of the form $\bar L\left(\bq\right)$ for a bouquet $\bq$ over $\h$.
\end{cor}

\subsection{Characterizing gerbes by their stalks}
In this subsection, we will show that small gerbes over an \'etale stack have a simple characterization in terms of their stalks:

\begin{thm}
Let $\X$ be an \'etale stack. A small stack $\Z$ over $\X$ is a small gerbe if and only if for every point $$x:* \to \X,$$ the stalk $\Z_x$ of $\Z$ at $x$ is equivalent to a group.
\end{thm}
\begin{proof}
Fix $\h$ an \'etale groupoid such that $\X \simeq \left[\h\right]$ and $\tilde x \in \h_0$ a point such  that $x \cong p \circ \tilde x,$ where $p:\h_0 \to \X$ is the atlas associated to $\h$. Suppose that $\g$ is a small gerbe over $\X$. Then, since $\g$ is locally non-empty, $$\Z_x\simeq  \underset{\tilde x \in U} \hc\Z\left(U\right),$$ is a non-empty groupoid. Furthermore, since $\g$ is locally connected, it follows that $$\underset{\tilde x \in U} \hc\Z\left(U\right)$$ is also connected, hence, equivalent to a group.

Conversely, suppose that $\Z$ is a small stack and that $$\Z_x\simeq  \underset{\tilde x \in U} \hc\Z\left(U\right)$$ is equivalent to a group. This means it is a non-empty and connected groupoid. It follows that $\Z$ is locally non-empty and locally connected, hence a gerbe.
\end{proof}

The significance of this theorem is the following:

Suppose we are given an effective \'etale stack $\X$ and a small gerbe $\g$ over it. By taking stalks, we get an assignment to each point $x$ of $\X$ a group $\g_x.$ From this data, we can build a new \'etale stack by taking the \'etale realization of $\g$. Denote this new \'etale stack by $\Y$. As it will turn out, if $\g$ is non-trivial, $\Y$ will not be effective, but it will have $\X$ as its effective part and, for each point $x$ of $\X$, the stalk $\g_x$ will be equivalent to the ineffective isotropy group of $x$ in $\Y$. In particular, if $\X$ is a space $X$, $\Y$ will be an \'etale stack which ``looks like $X$'' except that each point $x \in X,$ instead of having a trivial automorphism group, will have a group equivalent to $\g_x$ as an automorphism group. In this case, every automorphism group will consist entirely of purely ineffective automorphisms.

\subsection{Gerbes are full effective local equivalences}

In this subsection, we will characterize which small stacks $\Z$ over an \'etale stack $\X$ are gerbes in terms of their \'etale realization. In particular, we will show that when $\X$ is effective, gerbes over $\X$ are the same as \'etale stacks $\Y$ whose effective part are equivalent to $\X$.

Let $\h$ be an \'etale $S$-groupoid and let $\K$, with $$\mu_i:\K_i \to \h_0$$ for $i=0,1$, be a groupoid object in $\B\h$. Then the map $$\theta_{\K}:\h \ltimes \K \to \h$$ factors through the canonical map $$p_{\K}:\h_{\mu_0} \to \h.$$ Recall (Definition \ref{dfn:cech}) that $\h_{\mu_0}$ has $\K_0$ as object space, and an arrow from $x$ to $y$ is an arrow $$h:\mu_0\left(x\right) \to \mu_0\left(y\right)$$ in $\h$, and we write such an arrow as $\left(h,x,y\right)$. Define $\theta'_{\K}$ on objects to be the identity, and on arrows by sending an arrow $\left(h,k\right)$ in $\h \ltimes \K$ to $$\left(h,s\left(k\right),t\left(k\right)\right).$$ Then $\theta_{\K}=p_{\K} \circ \theta'_{\K}.$

\begin{lem}\label{lem:situation}
In the situation above, let $\h$ be effective. Then for two arrows $\left(h_i,k_i\right)$ $i=1,2$, in $\h \ltimes \K$,
\begin{equation}\label{eq:1}
\theta'_{\K}\left(h_1,k_1\right)=\theta'_{\K}\left(h_2,k_2\right)
\end{equation}\label{eq:2}
if and only if
\begin{equation}
\left[\left(h_1,k_1\right)\right]=\left[\left(h_1,k_1\right)\right],
\end{equation}
where the bracket denotes the image in $\Ef\left(\h \ltimes \K\right).$
\end{lem}

\begin{proof}
Suppose that (\ref{eq:1}) holds, with $k_i:h_i \cdot x_i \to y_i$. Then $$h_1=h_2=:h,$$ $$x_1 = x_2:=x,$$ and $$y_1=y_2=:y.$$ Let $V$ be a neighborhood of $hx$ in $\K_0$ over which $\mu_0$ restricts to an embedding. Let $U$ be a neighborhood of $h$ in $\h_1$ over which $s$ and $t$ restrict to embeddings, and $W_i$ be analogous neighborhoods of $k_1$ and $k_2$ in $\K_1$. For all $i$, let $$\mathcal{O}_i:=\left(\left(W_i \cap \mu_1^{-1}\left(U\right)\right) \times U\right) \cap \left(\h \ltimes \K\right)_1 \subset \left(\h \ltimes \K\right)_1.$$ Let $$M:=\bigcap\limits_{i=1,2}\left(t\left(\mathcal{O}_i \cap s \circ t|_{W_i}^{-1}\left(V\right)\right)\right) \subset \K_0,$$ and $$f_i:M \to \K_1$$ be given by $f_i:= \left(t|_{W_i}^{-1}\right)|_{M}$. Then the target map of $\h \ltimes \K$ restricts to an embedding over $\mathcal{O}_i$, and letting $$\sigma_i:=\left(t|_{\mathcal{O}_i}^{-1}\right)|_{M},$$ we have $$\left[\left(h_i,k_i\right)^{-1}\right]= s \circ \sigma_i.$$ Moreover, for each $x \in \mathcal{O}_i$, $$\sigma_i\left(x\right)=\left(t|_{U}^{-1}\left(\mu_1\left(f_i\left(x\right)\right)\right),f_i\left(x\right)\right).$$ Since for all $i$, $$\mu_1\left(f_i\left(x\right)\right)=\mu_0\left(t \circ f_i\left(x\right)\right)=\mu_0\left(x\right),$$ this simplifies to $$\sigma_i\left(x\right)=\left(t|_{U}^{-1}\left(\mu_0\left(x\right)\right),f_i\left(x\right)\right).$$ So, for all $i$, $$s\circ \sigma_i\left(x\right)=\left(t|_{U}^{-1}\left(\mu_0\left(x\right)\right)\right)^{-1} \cdot s\left(f_i\left(x\right)\right).$$ But $$\mu_0\left(x\right)=\mu_0\left( s \circ f_1\left(x\right)\right) = \mu_0\left( s \circ f_2\left(x\right)\right),$$ and $s \circ f_i\left(x\right) \in V$ for all $i$, so, $$s \circ f_1\left(x\right)=s \circ f_2\left(x\right).$$ This implies $$s \circ \sigma_1\left(x\right)=s \circ \sigma_2\left(x\right),$$ i.e. $$\left[\left(h_1,k_1\right)^{-1}\right]=\left[\left(h_2,k_2\right)^{-1}\right].$$ Hence $$\left[\left(h_1,k_1\right)\right]=\left[\left(h_2,k_2\right)\right].$$
Conversely, suppose $$\left[\left(h_1,k_1\right)\right]=\left[\left(h_2,k_2\right)\right].$$ Then
\begin{equation}\label{eq:3}
\left(t|_{U_1}^{-1}\left(\mu_0\left(x\right)\right)\right)^{-1} \cdot s\left(f_1\left(x\right)\right)=\left(t|_{U_2}^{-1}\left(\mu_0\left(x\right)\right)\right)^{-1} \cdot s\left(f_2\left(x\right)\right)
\end{equation}
on some neighborhood $\Omega$ of $x$, which we may assume maps homeomorphically onto its image under $\mu_0$. Applying $\mu_0$ to (\ref{eq:3}) yields $$s \circ t|_{U_1}^{-1}\left(y\right)=s \circ t|_{U_2}^{-1}\left(y\right),$$ for all $y \in \mu_0\left(\Omega\right)$. This implies that $h_1$ and $h_2$ have the same germ. As $\h$ is effective, this implies $h_1=h_2$. This in turn implies that $s \circ f_1$ and $s \circ f_2$ agree on $\Omega$, hence $k_1$ and $k_2$ have the same germ. In particular, they have the same source and target, hence $$\theta_{\K}'\left(h_1,k_1\right)=\theta_{\K}'\left(h_2,k_2\right).$$
\end{proof}

\begin{thm}
Let $\h$ be an effective \'etale $S$-groupoid and let $\bq$ be a bouquet over $\h$. Then $$\Ef\left(\h \ltimes \bq\right) \cong \h_{\mu_0}.$$
\end{thm}
\begin{proof}
Define a map $\varphi:\h_{\mu_0} \to \Ef\left(\h \ltimes \bq\right)$ as follows. On objects define it as the identity. Notice that the arrows of $\h_{\mu_0}$ are triples $\left(h,x,y\right)$ such that $$h:\mu_0\left(x\right) \to \mu_0\left(y\right).$$ For such a triple,  $$\left(hx,y\right) \in \bq_0 \times_{\h_0} \bq_0.$$ Recall that $$\left(s,t\right):\bq_1 \to \bq_0 \times_{\h_0} \bq_0$$ is surjective. For each $\left(hx,y\right) \in \bq_1 \to \bq_0 \times_{\h_0} \bq_0$, choose a $\gamma \in \bq_0$ such that $$\gamma:hx \to y,$$ and define $$\kappa\left(h,x,y\right):=\left[\left(h,\gamma\right)\right].$$ From the previous lemma, $\kappa$ does not depend on our choice. Moreover, $\left(s,t\right)$ admits continuous local sections, so it follows that $\kappa$ is continuous. Suppose that $$\kappa\left(h,x,y\right)=\kappa\left(h',x',y'\right).$$ Then $$\theta\left(h,\gamma\right)=\theta\left(h',\gamma'\right)$$ which in turn implies $$\left(h,x,y\right)=\kappa\left(h',x',y'\right).$$ Hence $\kappa$ is injective. Now let $\left[\left(h,\gamma\right)\right] \in \Ef\left(\h \ltimes \bq\right)_1$ be arbitrary. Then $$\left[\left(h,\gamma\right)\right]=\kappa\left(h,s\left(\gamma\right),t\left(\gamma\right)\right),$$ so $\kappa$ is bijective, hence an isomorphism. Moreover, identifying $\h_{\mu_0}$ with its effective part, $$\kappa^{-1} = \Ef\left(\theta'_{\bq}\right),$$ and in particular, is continuous.
\end{proof}

\begin{cor}
For $\bq$ a bouquet over an effective \'etale $S$-groupoid $\h$, $$L\left(\bq\right)=\g \to \X$$ is an effective local equivalence over $\X\simeq\left[\h\right]$.
\end{cor}
\begin{proof}
As $\theta_{\bq}=p_{\bq} \circ \theta'_{\bq}$ and $p_{\bq}$ is a Morita equivalence, it suffices to show that $\Ef\left(\theta'_{\bq}\right)$ is an equivalence, but this is clear as $\kappa$ is its inverse, by construction.
\end{proof}

\begin{cor}\label{cor:gg}
If $\g=\left(\rho: \underline{\g} \to \X\right) \in \gb$ is a small gerbe over an effective \'etale stack $\X$, $\underline{\g} \to \X$ is an effective local equivalence.
\end{cor}

\begin{thm}\label{thm:gerbsec}
Consider the map of \'etale $S$-groupoids $$\iota_{\G}:\G \to \Ef\left(\G\right)=:\h.$$ Then $$\bar\Gamma\left(\left[\iota_{\G}\right]\right) \in \St\left(\left[\h\right]\right)$$ is a gerbe.
\end{thm}
\begin{proof}
By Theorem \ref{thm:secs}, it suffices to show that $P\left(\iota_{\G}\right)$ is a bouquet over $\h$. The map
$$t \circ pr_1:\h_1 \times_{\h_0} \G_0 \to \h_0$$ is clearly surjective, as we may identify it with the map $$t:\h_1 \to \h_0.$$
It suffices to show that the map
\begin{eqnarray*}
\h_1 \times_{\h_0} \G_1 &\to& \h_1 \times_{\h_0} \h_1\\
\left(h,g\right) &\mapsto& \left(h[g],h\right)\\
\end{eqnarray*}
is surjective, where $\h_1 \times_{\h_0} \h_1$ is the pullback
$$\xymatrix{\h_1 \times_{\h_0} \h_1 \ar[d] \ar[r] & \h_1 \ar[d]^-{t}\\
\h_1 \ar[r]^-{t} & \h_0.}$$
Given $l$ and $l'$ in $\h_1$ with common target, choose $g$ such that $[g]=l'^{-1}l.$ Then $\left(l',g\right)$ gets sent to $\left(l,l'\right).$
\end{proof}

\begin{cor}\label{cor:effchar}
$\g=\left(\rho: \underline{\g} \to \X\right) \in Et\left(\X\right)$ is a small gerbe over an effective \'etale stack $\X$ if and only if  $\rho:\underline{\g} \to \X$ is an effective local equivalence.
\end{cor}

\begin{cor}
If $\X$ is an orbifold, it encodes a small gerbe over its effective part $\Ef\left(\X\right)$ via $$\iota_\X:\X \to \Ef\left(\X\right),$$ where $\iota$ is the unit of the adjunction in Theorem \ref{thm:bungalo}.
\end{cor}

\begin{thm}
Let $\X$ be an effective \'etale stack and $\g$ a small gerbe over it. Denote by $\Y$ the underlying \'etale stack of the \'etale realization of $\g$. Then, under the natural bijection between the points of $\X$ and the points of $\Y,$ for each point $x,$ the stalk $\g_x$ is equivalent to the ineffective isotropy group of $x$ in $\Y$, as defined in Definition \ref{dfn:ineffgp}.
\end{thm}

\begin{proof}
Represent $\X$ by an \'etale groupoid $\h$ and $\g$ by a bouquet $\bq$ over $\h$. Denote the objects of the bouquet by $$\mu_0:\bq_0 \to \h_0.$$ Then the \'etale realization of $\g$ is induced by the map of groupoids $$\theta'_{\bq}:\h\ltimes \bq \to \h_{\mu_0},$$ where $\h_{\mu_0}$ is the \v{C}ech groupoid with respect to the \'etale cover $\mu_0$ and $\theta'_{\bq}$ is as defined in the beginning of this subsection. To the \'etale cover $\mu_0$, there is an associated atlas $$p':\bq_0 \to \X.$$ Let $x$ be a point of $\X$. Then there exists a point $\tilde x \in \bq_0$ such that $x \cong p' \circ \tilde x.$ On one hand, from Section \ref{sec:inverseimage}, it follows that the stalk $\g_x$ is equivalent to the weak pullback in $S$-groupoids
$$\xymatrix{{*} \times_{\h_{\mu_0}} \left(\h \ltimes \bq\right) \ar[rr] \ar[d] & & \h \ltimes \bq \ar[d]^-{\theta'_{\bq}}\\
{*} \ar[r]^-{\tilde x} & \bq_0 \ar[r]^-{p'} & \h_{\mu_0},}$$
which is necessarily a connected groupoid.

An \textbf{object} of this groupoid can be described by a pair of the form $\left(z,h\right)$ with $z \in \bq_0$ and $h \in \h_1$ such that $$h:\mu_0\left(\tilde x\right) \to \mu_0\left(z\right).$$

An \textbf{arrow} from $\left(z,h\right)$ to $\left(z',h'\right),$ can be described simply as an arrow $$\gamma:z \to z'$$ in $\bq_1$.

Since this groupoid is equivalent to a group, it must be equivalent to the isotropy group of any object. Consider the object $\left(\tilde x,\mathbb{1}_{\mu_0\left(\tilde x\right)}\right).$ Then its isotropy group is canonically isomorphic to $\bq_{\tilde x},$ the isotropy group of $\tilde x$ in $\bq$.

On the other hand, the ineffective isotropy group of $x$ is isomorphic to the kernel of the homomorphism
$$\left(\h \ltimes \bq\right)_{\tilde  x} \to \mathit{Diff}_{\tilde x}\left(\bq_0\right)$$ induced from the canonical map $$\left(\h \ltimes \bq\right) \to \Ha\left(\bq_0\right).$$ From Lemma \ref{lem:situation}, it follows that the this is the same as the kernel of the map
\begin{eqnarray*}
\left(\h \ltimes \bq\right)_{\tilde  x} &\to& \h_{\mu_0\left(\tilde x\right)}\\
\left(h,l\right) &\mapsto& h\\
\end{eqnarray*}
which is induced from $\theta'_{\bq}$ and the canonical identification $$\left(\h_{\mu_0}\right)_{\tilde x} \cong \h_{\mu_0\left(\tilde x\right)}.$$ This kernel is clearly isomorphic to $\bq_{\tilde x}$ as well.
\end{proof}

Hence, we can use the data of a small gerbe over an effective \'etale stack to add ineffective isotropy groups to its points, as claimed. The rest of this subsection will be devoted to characterizing gerbes over general \'etale stacks which need not be effective.

\begin{dfn}
A map of stacks $\X \to \Y$ is \textbf{full} if for every space $T$, the induced map of groupoids $\X\left(T\right) \to \Y\left(T\right)$ if full as a functor.
\end{dfn}

\begin{prop}\label{prop:full}
Let $\rho:\underline{\g} \to \X$ be a small gerbe over an \'etale stack. Then $\rho$ is a full epimorphism.
\end{prop}
\begin{proof}
The fact that it is an epimorphism is clear. To see that it is full, we may assume it is of the form $\bar L\left(\bq\right)$ for a bouquet $\bq$. This means it is the stackification of the map
$$\theta_{\bq}:\h \ltimes \bq \to \h,$$ where $\X \simeq \left[\h\right].$ Such a map is clearly full.
\end{proof}

\begin{thm}\label{thm:dec}
Let $\rho:\underline{\g} \to \Y$ be a local homeomorphism of \'etale stacks. If $\rho:\underline{\g} \to \Y$ is a small gerbe over $\Y,$ then $$\iota_{\Y} \circ \rho:\underline{\g} \to \Ef\left(\Y\right)$$ is a small gerbe over $\Ef\left(\Y\right)$. Conversely, $\rho:\underline{\g} \to \Y$ is a small gerbe over $\Y$ if and only if  $$\iota_{\Y} \circ \rho:\underline{\g} \to \Ef\left(\Y\right)$$ is a small gerbe over $\Ef\left(\Y\right)$ and $\rho$ is full.
\end{thm}
\begin{proof}
Suppose that $\rho:\underline{\g} \to \Y$ is a small gerbe over $\Y$. In particular, this implies $\rho:\underline{\g} \to \Y$ is an epimorphism. From Theorem \ref{thm:gerbsec}, $\iota_{\Y}$ is a gerbe, hence also an epimorphism. This implies $\iota_{\Y} \circ \rho$ is an epimorphism. Let $\X:=\Ef\left(\Y\right).$ The following diagram is a $2$-pullback:
$$\xymatrix{\underline{\g} \times_{\Y} \underline{\g} \ar[r] \ar[d] & \Y \ar[d]\\
\underline{\g} \times_{\X} \underline{\g} \ar[r] & \Y \times_{\X} \Y.}$$
Since $\iota_{\Y}$ is a gerbe, the map $\Y \to \Y \times_{\X} \Y$ is an epimorphism, hence so is $$\underline{\g} \times_{\Y} \underline{\g} \to \underline{\g} \times_{\X} \underline{\g}.$$ The composite, $$\underline{\g} \to \underline{\g} \times_{\Y} \underline{\g} \to \underline{\g} \times_{\X} \underline{\g}$$ is an epimorphism, since $\g$ is a gerbe over $\Y$. Hence $\g$ is a gerbe over $\X$.

Conversely, suppose that $\iota_{\Y} \circ \rho$ is a gerbe  over $\X$ and that $\rho$ is full. In particular, it is an epimorphism. Let $\varphi:\G \to \K$ be a map of $S$-groupoids such that $$\left[\varphi\right] \cong \rho.$$ Let $\h=\Ef\left(\K\right).$ Then the map $$t \circ pr_1:\h_1 \times_{\h_0} \G_0 \to \h_0$$ is a surjective local homeomorphism. To show that $\rho$ is an epimorphism, we want to show that the induced map $$t \circ pr_1:\K_1 \times_{\K_0} \G_0 \to \K_0=\h_0$$ is a surjective local homeomorphism. It is automatically a local homeomorphism as $pr_1$ is the pullback of one and $t$ is one. It suffices to show that it is surjective. However, it can be factored as $$\K_1 \times_{\K_0} \G_0 \to \h_1 \times_{\h_0} \G_0 \to \h_0.$$ To show that $\g$ is in fact a gerbe over $\Y$, we need  to show that $\underline{\g} \to \underline{\g} \times_{\Y} \underline{\g}$ is an epimorphism. In terms of groupoids, this is showing that the map $$t\circ pr_1:\left(\G\times_{\K} \G\right)_1 \times_{\left(\G\times_{\K} \G\right)_0} \G_0 \to \left(\G\times_{\K} \G\right)_0$$ is a surjective local homeomorphism, where $\left(\G\times_{\K} \G\right)$ is a weak pullback of $S$-groupoids. To see that it is a local homeomorphism, note that we have the following commutative diagram:

$$\xymatrix{\left(\G\times_{\h} \G\right)_1 \times_{\left(\G\times_{\h} \G\right)_0} \G_0 \ar[r]^-{et} \ar[d]_-{et} & \left(\G\times_{\h} \G\right)_0 \ar[d]^-{et}\\
\left(\G\times_{\K} \G\right)_1 \times_{\left(\G\times_{\K} \G\right)_0} \G_0 \ar[r] & \left(\G\times_{\K} \G\right)_0,}$$
where the maps marked as $et$ are local homeomorphisms. Since $\underline{\g} \to \X$ is a gerbe, we know the map $$\left(\G\times_{\h} \G\right)_1 \times_{\left(\G\times_{\h} \G\right)_0} \G_0 \to \left(\G\times_{\h} \G\right)_0$$ is a surjective local homeomorphism. This implies that for every $$\left[\gamma\right]:\varphi\left(x_1\right) \to \varphi\left(x_2\right)$$ in $\left(\G \times_{\h} \G\right)_0$, there exists $g_1$ and $g_2$ in $\G_1$ such that $$\left[\gamma\right] \circ \left[\varphi\left(g_1\right)\right] = \left[\varphi\left(g_2\right)\right].$$ Suppose instead we are given $$\gamma:\varphi\left(x_1\right) \to \varphi\left(x_2\right)$$ in $\left(\G\times_{\K} \G\right)_0.$ Then as $\rho$ is full, so is $\varphi$, hence $$\gamma=\varphi\left(a\right)$$ for some $a \in \G_1$. Now, there exists $g_1$ and $g_2$ in $\G_1$ such that $$\left[\gamma\right] \circ \left[\varphi\left(g_1\right)\right] = \left[\varphi\left(g_2\right)\right].$$ Let $g'_2:=\left(a \circ g_1\right)^{-1}.$ Then $\left(g_1,g_2\right)$ is an arrow in $\left(\G \times_{\K} \G\right)$ from $$\left(x,x,\mathbb{1}_{\varphi(x)}\right)$$ to  $$\left(x_1,x_2,\gamma\right).$$ Hence $$\left(\G\times_{\K} \G\right)_1 \times_{\left(\G\times_{\K} \G\right)_0} \G_0 \to \left(\G\times_{\K} \G\right)_0$$ is a surjection.
\end{proof}

\begin{cor}
Let $\g=\left(\rho:\underline{\g} \to \X\right)$ be a local homeomorphism of \'etale stacks. Then $\g$ is a small gerbe over $\X$ if and only if $\rho$ is a full, effective local equivalence.
\end{cor}

\begin{proof}
Suppose that $\g$ is a gerbe. From Theorem \ref{thm:gerbsec}, $\iota_{\X}:\X \to \Ef\left(\X\right)$ is a small gerbe over $\Ef\left(\X\right).$ Hence the composite $\iota_{\X} \circ \rho$ is a gerbe over $\Ef\left(\X\right)$ by Theorem \ref{thm:dec}. By Corollary \ref{cor:gg}, this implies that it is an effective local equivalence, i.e. $$\Ef\left(\iota_{\X} \circ \rho\right) = \Ef\left(\iota_{\X}\right) \circ \Ef\left(\rho\right)$$ is an equivalence. But $\Ef\left(\iota_{\X}\right)$ is an isomorphism, hence $\ \Ef\left(\rho\right)$ is an equivalence. So $\rho$ is an effective local equivalence. It is full by Proposition \ref{prop:full}.

Conversely, suppose that $\rho$ is full and an effective local equivalence. It follows that $\iota_{\X} \circ \rho$ is an effective local equivalence over $\Ef\left(\X\right)$, hence a gerbe by Corollary \ref{cor:effchar}. The result now follows from Theorem \ref{thm:dec}.
\end{proof}

\section{The 2-Category of Gerbed Effective \'Etale Stacks}\label{sec:groth}
\sectionmark{Gerbed Effective \'Etale Stacks}
In this section, we will treat topological and differentiable stacks as fibered categories (categories fibered in groupoids over $S$). The Grothendieck construction provides an equivalence of $2$-categories between this description, and the one in terms of groupoid valued weak $2$-functors \cite{elephant2}. We will assume the reader is familiar with the language of fibered categories and this equivalence. For a quick introduction to fibered categories, see for instance \cite{FGA}. For a fibered category $\X$, we shall denote the structure map which makes it a fibered category over $S$ by $p_\X$.

For a diagram of fibered categories:

$$\xymatrix{& \Z \ar[d]^-{\rho} \\
\X \ar[r]^-{g} & \Y,}$$
we choose the explicit weak pullback described as follows. The objects of $g^*\left(\Y\right)$ are triples $\left(x,z,r\right)$ in $\X_0 \times \Z_0 \times \times \Y_1$ such that $$p_\X\left(x\right)=p_\Z\left(z\right)=T$$ and $$r:g\left(x\right) \to \rho\left(z\right),$$ with $r \in \Y\left(T\right).$ An arrow between a triple $\left(z_1,x_1,r_1\right)$ and a triple $\left(z_2,x_2,r_2\right)$ is a pair $\left(u,v\right) \in \Z_1 \times \X_1$ such that $$p_\Z\left(u\right)=p_\X\left(v\right),$$ making the following diagram commute:
$$\xymatrix{g\left(x_1\right) \ar[d]_-{r_1} \ar[r]^-{g\left(u\right)} & g\left(x_2\right) \ar[d]^-{r_2}\\
\rho\left(z_1\right) \ar[r]^-{\rho\left(v\right)} & \rho\left(z_2\right).}$$
It has structure map $$p_{g^*\Y}\left(x,z,r\right)=p_\Z\left(z\right)=p_\X\left(x\right),$$ $$p_{g^*\Y}\left(u,v\right)=p_\Z\left(u\right)=p_\X\left(v\right).$$
We denote the canonical projections as $pr_1:g^{*}\rho \to \Z$ and $pr_2:g^*\rho \to \Y$. We define $g^*\rho$ as the map $pr_1:g^*\Y \to \Z.$

Given $\alpha:f \Rightarrow g$ with $g:\X \to \Y$, there is a canonical map $$\alpha^*:g^*\rho \to f^*\rho$$ given on objects as $$\left(z,x,r\right) \mapsto \left(z,x,r \circ \alpha(z)\right),$$ and given as the identity on arrows. This strictly commutes over $\X$. We denote the associated map in $St\left(S\right)/\X$ as $\alpha^*\rho$.

Given a composable sequence of arrows, $$\W \stackrel{f}{\longrightarrow} \X \stackrel{g}{\longrightarrow} \Y,$$ there is a canonical isomorphism $\chi_{g,f}:f^*g^*\Z \to \left(gf\right)^*\Z$ given on objects as $$\left(w,\left(x,z,r\right),q\right) \mapsto \left(w,z,r \circ g\left(q\right)\right),$$ and on arrows as $$\left(u,\left(a,b\right)\right) \mapsto \left(u,b\right).$$ This strictly commutes over $\W$. We denote the associated map in $St\left(S\right)/\W$ by the same name.

In a similar spirit, given $\tau:\W \to \Y$ with $m:\rho \to \tau$ in $\St\left(S\right)/\Y$,  and $$f:\X \to \Y,$$ there is a canonical map $$f^*m:f^*\rho \to f^*\tau$$ in $\St\left(S\right)/\X$, and given $\phi:m \Rightarrow n$, with $g:\rho \to \tau$, there is a canonical $2$-cell $$f^*\phi:f^*m \Rightarrow f^*n.$$ We invite the reader to work out the details.

Finally, we note that if $$f:\X \to \Y,$$ $$\rho:\Z \to \Y,$$ $$\lambda:\W \to \X,$$ $$\zeta:\W \to \Z,$$ and $$\omega: \rho \circ \zeta \Rightarrow f \circ \lambda,$$ there is a canonical map

\begin{eqnarray*}
\left(\lambda,\zeta,\omega\right):\W &\to& f^*\Z\\
w &\mapsto& \left(\lambda(w),\zeta(w),\omega(w)^{-1}\right)\\
l &\mapsto& \left(\lambda(l),\zeta(l)\right).\\
\end{eqnarray*}

This data provides us with coherent choices of pullbacks. We will now use this data to construct a $2$-category we will call the $2$-category of \textbf{gerbed effective \'etale stacks}. We will denote it by $\gets$.

Its \textbf{objects} are pairs $\left(\X,\sigma\right)$ with $\X$ an effective \'etale stack and $\sigma \to \X$ an effective local equivalence. Of course, this is the same data as a small gerbe over $\X$.

An \textbf{arrow} from $\left(\X,\sigma\right)$ to $\left(\Y,\tau\right)$ is a pair $\left(f,m\right)$ where $f:\X \to \Y$ and $m:\sigma \to f^*\tau$ in $\St\left(S\right)/\X$. Note that this is equivalent data to a map in $\St\left(\X\right)$ from $\sigma$ to $f^*\tau$ viewed as gerbes.

A $2$-\textbf{cell} between such an $\left(f,m\right)$ and a $$\left(g,n\right):\left(\X,\sigma\right) \to \left(\Y,\tau\right),$$ is a pair $\left(\alpha,\phi\right)$ with $\alpha:f \Rightarrow g$ a $2$-cell in $\Eft$, and $\phi$ a $2$-cell in $\St\left(S\right)/\X$ such that

$$\xymatrix{\sigma \ar[rr]^-{n} \ar[rrdd]_{m} & & g^*\tau \ar[dd]^{\alpha^*\tau} \ar@2{->}[dl]|{\phi}  \\
& & \\
& & f^*\tau.}$$

\textbf{Composition of $1$-morphisms} is given as follows:

If $$\left(\X,\sigma\right) \stackrel{\left(f,m\right)}{\longlongrightarrow} \left(\Y,\tau\right) \stackrel{\left(g,n\right)}{\longlongrightarrow} \left(\Z,\rho\right),$$ is a pair of composable $1$-morphisms, define their composition as $\left(gf,n\ast m\right),$ where $n\ast m$ is defined as the composite
$$\sigma \stackrel{m}{\longlongrightarrow} f^*\tau \stackrel{f^*\left(n\right)}{\longlongrightarrow} f^*g^*\rho \stackrel{\chi_{g,f}}{\longlongrightarrow} \left(gf\right)^*\rho.$$

\textbf{Vertical composition of $2$-cells} is defined in the obvious way.

Suppose $$\left(\alpha,\phi\right):\left(f,m\right) \Rightarrow \left(k,p\right)$$ and
$$\left(\beta,\psi\right):\left(g,n\right) \Rightarrow \left(l,p\right),$$
with $$\left(f,m\right):\left(\X,\sigma\right) \to \left(\Y,\tau\right)$$ and $$\left(g,n\right):\left(\Y,\tau\right) \to \left(\Z,\rho\right).$$
Denote the horizontal composition of $\beta$ with $\alpha$ by $\beta \circ \alpha$. Then we define the \textbf{horizontal composition of 2-cells}
$$\left(\beta,\psi\right) \circ \left(\alpha,\phi\right):\left(g,n\right) \circ \left(f,m\right) \Rightarrow \left(l,p\right) \circ \left(k,o\right),$$
by $$\left(\beta,\psi\right) \circ \left(\alpha,\phi\right)=\left(\beta \circ \alpha, \psi \ast \phi\right),$$where $\psi \ast \phi$ is defined by the pasting diagram:

$$\xymatrix@C=2.5cm{ & & k^*l^*\rho \ar[r]^-{\chi_{l,k}} \ar[d]^{\left(k^* \circ \beta^*\right)\left(\rho\right)} \ar@{} @<-4pt>[d]| (0.4) {}="c" & \left(lk\right)^*\left(\rho\right) \ar[d]^-{\left(\beta\circ k\right)^*\left(\rho\right)}  \\
& k^*\tau  \ar[d]^-{\alpha^*\left(\tau\right)} \ar@{} @<-4pt>[d]| (0.4) {}="a" \ar[ru]^-{k^*\left(\rho\right)} \ar[r]_-{k^*(n)} \ar@{} @<4pt> [r]| (0.4) {}="d" \ar @{=>}|{k^*\left(\psi\right)}  "c";"d" & k^*g^*\rho \ar[d]^{\left(\alpha^* \circ g^*\right)\left(\rho\right)} \ar[r]_-{\chi_{g,k}} & \left(gk\right)^*\left(\rho\right) \ar[d]^-{\left(g\alpha\right)^*\left(\rho\right)}\\
\sigma \ar[ur]^-{p} \ar[r]_-{m} \ar@{} @<4pt> [r]| (0.4) {}="b" \ar @{=>}|{\phi}  "a";"b" & f^*\tau \ar[r]_-{f^*(n)} & f^*g^*\rho \ar[r]_-{\chi_{g,f}} & \left(gf^*\right)^*\left(\rho\right).}$$
\begin{rmk}
What we have actually done is applied the Grothendieck construction for bicategories \cite{bicat} to the trifunctor which associates to each effective \'etale stack, the $2$-category of effective local equivalences over $\X$ (which we know to be equivalent to the $2$-category $\gb$ of small gerbes over $\X$).
\end{rmk}

If $P$ is an \'etale invariant subcategory of spaces, we can similarly define the $2$-category $\gets_P$ in which each underlying $1$-morphism in $\Eft$ must lie in $\Eft_P$.

\begin{thm}\label{thm:etin}
$P$ is an open \'etale invariant subcategory of spaces. Then the $2$-category $\gets_P$ of gerbed effective \'etale stacks and $P$-morphisms is equivalent to the $2$-category $\Et_P$ of \'etale stacks and $P$-morphisms (See Corollary \ref{cor:etadj}).
\end{thm}

\begin{proof}
Define a $2$-functor $\Theta:\Et_P \to \gets_P$.

\textbf{On objects:}
$$\Theta\left(\X\right)=\left(\Ef\left(\X\right),\iota_\X\right),$$ where $\iota$ is the unit of the adjunction in Theorem \ref{thm:bungalo}. This associates $\X$ to the gerbe it induces over $\Ef\left(\X\right).$

\textbf{On arrows:}
Suppose $\varphi:\X \to \Y$ is a map in $\Et_P.$ Notice that the diagram
$$\xymatrix{\X \ar[r]^-{\varphi} \ar[d]_-{\iota_{\X}} & \Y \ar[d]^-{\iota_{\Y}} \\
\Ef\left(\X\right) \ar[r]^-{\Ef\left(\varphi\right)} & \Ef\left(\Y\right),}$$
commutes on the nose, so there is an associated map $$\left(\iota_\X,\varphi,id\right):\X \to \Ef\left(\varphi\right)^*\Y.$$
Define $\Theta\left(\varphi\right)=\left(\Ef\left(\varphi\right),\left(\iota_\X,\varphi,id\right)\right).$

\textbf{On $2$-cells:}
Suppose that $\varphi':\X \to \Y$ and $$\alpha:\varphi \Rightarrow \varphi'.$$ Then define $$\Theta\left(\alpha\right)=\left(\Ef\left(\alpha\right),\tilde \alpha\right),$$
where $$\tilde \alpha:\X_0 \to \left(\Ef\left(\varphi\right)^*\rho\right)_1$$ is defined by the equation
$$\tilde\alpha\left(x\right)=\left(id,\alpha(x)^{-1}\right).$$

We leave it to the reader to check that $\Theta$ is $2$-functor.

Define another $2$-functor $$\Xi:\gets_P \to \Et_P.$$
\textbf{On objects:}
If $\sigma:\g \to \X$ is an effective local equivalence, denote $\g$ by $\underline{\sigma}$. Let $$\Xi\left(\X,\sigma\right):=\underline{\sigma}.$$
\textbf{On arrows:}
Suppose $\left(f,m\right):\left(\X,\sigma\right) \to \left(\Y,\tau\right)$. Denote the underlying map of $m$ by $$\underline{m}:\underline{\sigma} \to \underline{f^*\tau}.$$ Define $$\Xi\left(f,m\right):=pr_2 \circ \underline{m}: \underline{\sigma} \to \underline{\tau},$$ where $$pr_2:f^*\tau \to \tau$$ is the canonical projection.\\
\textbf{On $2$-cells:}
Given $$\left(g,n\right):\left(\X,\sigma\right) \to \left(\Y,\tau\right)$$ and $$\left(\alpha,\phi\right):\left(f,m\right) \to \left(g,n\right),$$ define $\Xi\left(\left(\alpha,\phi\right)\right)$ by the following pasting diagram:
$$\xymatrix@C=2.5cm{\underline{\sigma} \ar@2{->}[drr]|{\phi^{-1}} \ar[rr]^-{\underline{m}} \ar[rrdd]_-{\underline{n}} & & f^*\Y \ar[r]^-{pr_2} & \Y. \\
& & &  \\
& & g^*\Y \ar[uu]_{\underline{\alpha^*\left(\tau\right)}} \ar[ruu]_-{pr_2} & }$$
By direct inspection, one can see that $$\Xi \circ \Theta = id_{\Eft_P}.$$ There is furthermore a canonical natural isomorphism $$\Theta \circ \Xi \Rightarrow id_{\gets_P}.$$ On objects $$\Theta \circ \Xi\left(\left(\X,\sigma\right)\right)=\left(\Ef\left(\underline{\sigma}\right),\iota_{\underline{\sigma}}\right).$$ By Corollary \ref{cor:gg} and Theorem \ref{thm:gerbsec}, this is canonically isomorphic to $\left(\X,\sigma\right)$. Moreover, if $$\left(f,m\right):\left(\X,\sigma\right) \to \left(\Y,\tau\right),$$ then $$\Theta \Xi \left(\left(f,m\right)\right)=\left(\Ef\left(pr_2 \circ \underline{m}\right),\left(\iota_{\underline{\sigma}},pr_2 \circ \underline{m},id\right)\right).$$ Consider the following diagram:

$$\xymatrix@C=2cm{\Ef\left(\underline{\sigma}\right) \ar[r]^-{\Ef\left(\underline{m}\right)} \ar[rd]_-{\Ef\left(\sigma\right)} & \Ef\left(pr_2\right) \ar[d]^-{\Ef\left(f^*\underline{\tau}\right)} \ar[r]^-{\Ef\left(f^*\tau\right)} & \Ef\left(\underline{\tau}\right) \ar[d]^-{\Ef\left(\tau\right)}\\
& \Ef\left(\X\right) \ar[r]^-{\Ef\left(f\right)} & \Ef\left(\Y\right)\\
& \X \ar[u]^{\iota_\X} \ar[r]^-{f} & \Y \ar[u]_-{\iota_\Y}.}$$
Since $\sigma$ and $f^*\tau$ are effective local equivalences, the triangle consists of all equivalences. The lower square likewise consists of all equivalences as $\X$ and $\Y$ are effective. We leave the rest of the details to the reader.
\end{proof}

\begin{cor}
There is an equivalence of $2$-categories between gerbed effective \'etale differentiable stacks and submersions, $\gets_{subm}$, and the $2$-category of \'etale differentiable stacks and submersions, $\Et_{subm}.$
\end{cor}

\begin{rmk}
Some variations of this are possible. For example, if we restrict to \'etale stacks whose effective parts are (equivalent to) spaces, so-called \textbf{purely ineffective} \'etale stacks, then the functor $\Ef$ extends to all maps. The proof of Theorem \ref{thm:etin} extends to this setting to show that purely ineffective \'etale stacks are equivalent to the $2$-category of gerbed spaces, a result claimed in \cite{pres2}. This theorem is a corrected version of theorem 94 of \cite{Metzler} (which is unfortunately incorrect since there is an error on the top of page 44, see the remark after Corollary \ref{cor:real}). Moreover, by results of \cite{pres2}, this restricts to an equivalence between purely ineffective orbifolds and gerbed manifolds whose gerbe has a locally constant band with finite stabilizers.
\end{rmk}

\appendix
\section{Sheaves in Groupoids vs. Stacks}\label{App:sheavess}

\begin{dfn}\label{strps}
Let $\C$ be a small category. A \textbf{strict presheaf in groupoids} over $\C$ is a strict 2-functor $F:\C^{op} \to Gpd$ to the $2$-category of (small) groupoids.  Notice that this is the same as a $1$-functor $\C^{op} \to \tau_1\left(Gpd\right)$, where the target is the 1-category of groupoids. A morphism of strict presheaves is a strict natural transformation (i.e. a natural transformation between their corresponding $1$-functors into $\tau_1\left(Gpd\right)$). A $2$-morphism between two natural transformations $\alpha_i:F \Rightarrow G$, $i=1,2$, is an assignment to each object $C$ of $\C$ a natural transformation $$w(C):\alpha_1(C) \Rightarrow \alpha_2(C)$$ subject to the following condition:

For all $f:D \to C$, we have two functors from $F(C)$ to $G(D)$, namely $$G(f)\alpha_1(C)=\alpha_1(D)F(f)$$ and $$G(f)\alpha_2(C)=\alpha_2(D)F(f).$$ Given our assignment $C \mapsto w(C)$, we have two different natural transformations between these functors: $G(f)w(C)$ and $w(D)F(f)$. $w$ is called a \textbf{modification} if these two natural transformations are equal. Modifications are the $2$-cells of strict presheaves. This yields a strict $2$-category of strict presheaves in groupoids $\mathit{Psh}(\C,Gpd)$.
\end{dfn}

\begin{prop}\label{prop:shfgpd}
The $2$-category $\mathit{Psh}(\C,Gpd)$ is equivalent to the $2$-category of groupoid objects in $\Set^{\C^{op}}$.
\end{prop}

\begin{proof}
Let $\left(\mspace{2mu} \cdot \mspace{2mu}\right)_i:\tau_1\left(Gpd\right) \to \Set$, $i=0,1,2$ be the functors which associate to a groupoid $\G$ its set of objects $\G_0$, its set of arrows $\G_1$, and its set $\G_2$ of composable arrows respectively. Let $F:\C^{op} \to \tau_1\left(Gpd\right)$ be a strict presheaf of groupoids. Then each $F_i$ is an ordinary presheaf of sets. Moreover, for each $C$, $F(C)$ is a groupoid, which we may write as demanding certain diagram involving each $F(C)_i$ to commute. These assemble to a corresponding diagram for the global $F_i$'s, showing they form a groupoid object in $\Set^{\C^{op}}$, $Q(F)$. Given $1$-morphism $\alpha:F \Rightarrow G$ in $\mathit{Psh}(\C,Gpd)$, let $Q(\alpha):Q(F) \to Q(F)$ be the internal functor with components $Q(\alpha)_i(C)=\alpha(C)_i$ for $i=0,1$. Finally, let $w$ be a modification from $\alpha$ to $\beta$. Then, in particular, for each $C$, $w(C):\alpha(C) \Rightarrow \beta(C)$ is a natural transformation, so is a map $w(C):F(C)_0 \to G(C)_1$ satisfying the obvious properties. It is easy to check that the conditions for $w$ to be a modification are precisely those for the family $\left(w(C):F(C)_0 \to G(D)_1\right)$ to assemble into a natural transformation $$Q(w):F_0 \Rightarrow G_1.$$ Since $w$ is point-wise a natural transformation, $Q(w)$ is an internal natural transformation. It is easy to check that this is indeed an equivalence of $2$-categories with an explicit inverse on objects given by $$\mathbb{G} \mapsto \Hom\left(\mspace{2mu} \cdot \mspace{2mu},\mathbb{G}\right).$$
\end{proof}

\begin{dfn}
Let $\left(\C,J\right)$ be a Grothendieck site. Then a \textbf{sheaf} of groupoids is a strict presheaf $F:\C^{op} \to \tau_1\left(Gpd\right)$ such that for any covering family $\left(C_i \to C\right)_i$, the induced morphism

$$F(C) \to \varprojlim \left[ \prod \limits_i {F(C_i)} \rrarrow \prod \limits_{i,j} {F(C_{ij})} \right]$$
is an isomorphism of groupoids. Sheaves of groupoids form a full sub-$2$-category $\Sh(\C,Gpd)$ of strict presheaves of groupoids.
\end{dfn}

The following proposition is easily checked:

\begin{prop}
The 2-functor $Q:\mathit{Psh}(\C,Gpd) \to Gpd\left(\Set^{\C^{op}}\right)$ restricts to an equivalence $Q:\Sh(\C,Gpd) \to Gpd\left(\Sh\left(\C\right)\right)$.
\end{prop}

Analogously to sheaves of sets, there is a $2$-adjunction

$$\xymatrix{\Sh(\C,Gpd)  \ar@<-0.5ex>[r]_-{i}  & \mathit{Psh}(\C,Gpd)\ar@<-0.5ex>[l]_-{sh}},$$
where $sh$ denotes sheafification.

Denote by $j:\mathit{Psh}(\C,Gpd) \to Gpd^{\C^{op}}$ the ``inclusion'' of strict presheaves into weak presheaves. We use quotations since this functor is not full. The following proposition is standard:

\begin{prop}\label{prop:sh}
Let $\Z$ be a strict presheaf of groupoids. Then $$a\circ j\left(\Z\right) \simeq a\circ j\circ i\circ sh\left(\Z\right),$$ where $a$ denotes stackification.
\end{prop}

In other words, if you start with a strict presheaf of groupoids, sheafify it to a sheaf of groupoids, and then stackify the result, this is equivalent to stackifying the original presheaf.

\begin{cor}\label{ref:stksh}
Every stack is equivalent to $a \circ j \circ i \left(\W\right)$ for some sheaf of groupoids $\W$.
\end{cor}

\section{Proof of Theorem \ref{thm:finn}}\label{sec:lem}
In this Appendix, we will prove that the generalized action groupoid construction described in Section \ref{subsec:action} yields a concrete description of \'etal\'e realization. For technical reasons, we start by fixing an ambient Grothendieck universe $\mathcal{U}$. Recall the following definition from \cite{sga4} (expos\'e ii):
\begin{dfn}
A locally $\mathcal{U}$-small Grothendieck site $\left(\mathscr{E},V\right)$ is called a $\mathcal{U}$-site if there exists a $\mathcal{U}$-small set of objects $G$, called topological generators, such that every object $E \in \mathscr{E}$ admits $V$-cover by a family of morphisms all of whose sources are in $G$.
\end{dfn}

\begin{thm}\cite{sga4} (expos\'e ii, th\'eor\`em 3.4)
If $\left(\mathscr{E},V\right)$ is a $\mathcal{U}$-site, then the category of presheaves of $\mathcal{U}$-small sets on $\E$ is locally $\mathcal{U}$-small. Moreover, the full subcategory thereof consisting of $\mathcal{U}$-small $V$-sheaves is reflective, and the reflector is left exact, hence this subcategory is a $\mathcal{U}$-topos.
\end{thm}

\begin{rmk}
In particular, a $\mathcal{U}$-site is \emph{not} necessarily $\mathcal{U}$-small. An important example of a $\mathcal{U}$-site which is not $\mathcal{U}$-small is the following:\\
\\
Suppose $\left(\C,J\right)$ is a $\mathcal{U}$-small site. Let $\E:=Sh^{\mathcal{U}}_{J}\left(\C\right)$ be the $\mathcal{U}$-topos of $J$-sheaves of $\mathcal{U}$-small sets. Equip $\E$ with the \emph{canonical topology} which is generated by jointly surjective epimorphic families. Denote this site by $\left(\E,can\right)$. This site is clearly not $\mathcal{U}$-small, but it is a $\mathcal{U}$-site, since the set of representable sheaves is $\mathcal{U}$-small (since it is a copy of $\C$) and topologically generates $\E$ (because of the Yoneda lemma). Moreover, the category of $\mathcal{U}$-small sheaves on this site, is canonically equivalent to $\E$ itself. More generally, if $\left(\C,J\right)$ is not $\mathcal{U}$-small, but just a $\mathcal{U}$-site, $\left(\E,can\right)$ is still a $\mathcal{U}$-site; its topological generators are the image of those of $\C$ under the Yoneda embedding.
\end{rmk}

\begin{lem}
Let $\E$ be a $\mathcal{U}$-small topos, $\left(\D,K\right)$ be a $\mathcal{U}$-small site, and $$G:\E \to \St^{\mathcal{U}}_{K}\left(\D\right)$$ be a $2$-functor from $\E$ into the $2$-topos (in $\mathcal{U}$) of $K$-stacks of (essentially) $\mathcal{U}$-small groupoids, which preserves coproducts and epimorphisms. Denote by $$\hat y_{\E}:\E \to \St^{\mathcal{U}}_{can}\left(\E\right)$$ the Yoneda embedding of $\E$ into the $2$-topos of $\mathcal{U}$-small stacks on $\E$. Then, there exists (an essentially unique) $\mathcal{U}$-small weak colimit preserving $2$-functor $$L^{\mathcal{U}}_{\E}:St^{\mathcal{U}}_{can}\left(\E\right) \to \St^{\mathcal{U}}_{K}\left(\D\right),$$ which when restricted to $\E$ along $y_{\E}$ agrees with $G$ (up to equivalence).
\end{lem}

\begin{proof}
Let $\mathcal{V}$ be a larger Grothendieck universe such that $\mathcal{U} \in \mathcal{V}$. Then $\left(\E,can\right)$ is a $\mathcal{V}$-small site. Consider the canonical inclusion $$i^{\mathcal{V}}_\D:St^{\mathcal{U}}_K\left(\D\right) \to St^{\mathcal{V}}_K\left(\D\right).$$ By \cite{htt} (Remark 6.3.5.17), this inclusion preserves $\mathcal{U}$-small weak colimits (apply $\tau_{\le 1}$ to functor in this proposition, and note that $\tau_{\le 1}$ is a colimit preserving functor of infinity categories). Let $\widehat{\E}$ denote the $2$-topos in $\mathcal{V}$ of $\mathcal{V}$-small weak presheaves of groupoids on $\E$. Let $L^{\mathcal{V}}_{\E}$ denote the weak left Kan extension of $i^{\mathcal{V}}_{\D} \circ G$ along the Yoneda embedding $\hat y^{\mathcal{V}}$ of $\E$ into $\widehat \E$. Denote its right adjoint by $R^{\mathcal{V}}_{\E}.$ Explicitly, for $\Z \in St^{\mathcal{V}}_K\left(\D\right)$, $R^{\mathcal{V}}_{\E}\left(\Z\right)$ is the weak presheaf that assigns $E \in \E$, the groupoid
\begin{eqnarray*}
R^{\mathcal{V}}_{\E}\left(\Z\right)\left(E\right)&\simeq& \Hom\left(L^{\mathcal{V}}_{\E}\hat y^{\mathcal{V}} \left(E\right),\Z\right)\\
&\simeq& \Hom\left(i^{\mathcal{V}}_{\D} G\left(E\right),\Z\right).\\
\end{eqnarray*}
Since $G$ and $i^{\mathcal{V}}_{\D}$ both preserves coproducts and epimorphisms, their composite preserves covers, and since $\Z$ is a stack, it follows that this presheaf satisfies descent for the canonical topology on $\E$, so is an object of $St^{\mathcal{V}}_{can}\left(\E\right).$ Hence, by abuse of notation, there is an induced $2$-adjunction
$$\xymatrix{St^{\mathcal{V}}_{can}\left(\E\right) \ar@<-0.5ex>[r]_-{L^{\mathcal{V}}_{\E}} & St^{\mathcal{V}}_K\left(\D\right) \ar@<-0.5ex>[l]_-{R^{\mathcal{V}}_{\E}}}.$$ The $2$-functor $L^{\mathcal{V}}_{\E}$ is uniquely determined up to equivalence by the fact that it is $\mathcal{V}$-small weak colimit preserving, and agrees up to equivalence with $i^{\mathcal{V}}_{\D} \circ G$ when restricted to $\E$ along $\hat y^{\mathcal{V}}$. Define $$L^{\mathcal{U}}_{\E}:=L^{\mathcal{V}}_{\E} \circ i^{\mathcal{V}}_{\E}$$ where $i^{\mathcal{V}}_{\E}$ is the canonical inclusion $$i^{\mathcal{V}}_\E:St^{\mathcal{U}}_{can}\left(\E\right) \to St^{\mathcal{V}}_{can}\left(\E\right),$$ which is $\mathcal{U}$-small weak colimit preserving. It follows that $L^{\mathcal{U}}_{\E}$ is $\mathcal{U}$-small weak colimit preserving, and by construction, when restricted to $\E$ along the Yoneda embedding $\hat y^{\mathcal{U}}$ of $\E$ into $\mathcal{U}$-small stacks over $\E$, it agrees up to equivalence with $G$. It suffices to show that the essential image of $L^{\mathcal{U}}_{\E}$ lies in the essential image of $i^{\mathcal{V}}_{\D},$ i.e. $\mathcal{U}$-small $K$-stacks. Let $\Y \in St^{\mathcal{U}}_{can}\left(\E\right).$ Then $\Y$ is a stack over the topos $\E$ in the sense of \cite{Giraud}, hence there exists groupoid object $\K \in Gpd\left(\E\right)$ such that $\Y \simeq \left[\K\right]$ is its stack completion. In particular, this implies that $\Y$ is the weak colimit of the truncated semi-simplicial diagram $$\K_2 \rrrarrow \K_1 \rrarrow \K_0,$$ when viewed as a diagram in $St^{\mathcal{U}}_{can}\left(\E\right),$ i.e. taking the weak colimit after applying $\hat y^{\mathcal{U}},$ the Yoneda embedding into $\mathcal{U}$-small stacks over $\E$. Hence

\begin{eqnarray*}
L^{\mathcal{U}}_{\E}\left(\Y\right) &=& L^{\mathcal{V}}_{\E}i^{\mathcal{V}}_{\E}\left(\Y\right)\\
&\simeq& L^{\mathcal{V}}_{\E}i^{\mathcal{V}}_{\E}\left( \hc \hat y^{\mathcal{U}} \K_n\right)\\
&\simeq& \hc L^{\mathcal{V}}_{\E} \hat y^{\mathcal{V}} \K_n\\
&\simeq& \hc i^{\mathcal{V}}_{\D} G\left(\K_n\right)\\
&\simeq& i^{\mathcal{V}}_{\D} \left(\hc G\left(\K_n\right)\right).\\
\end{eqnarray*}
Therefore, the essential image of $L^{\mathcal{U}}_{\E}$ consists entirely of $\mathcal{U}$-small $K$-stacks.
\end{proof}

Let $\X \simeq \left[\h\right]$ be an \'etale stack, with $\h$ an \'etale $S$-groupoid. Consider the $2$-functor
$$\h \ltimes:Gpd\left(\B\h\right) \to \left(S^{et}-Gpd\right)/\h,$$ from Section \ref{subsec:action}. Denote by $G$ the composition
$$\B\h \stackrel{q}{\hookrightarrow} Gpd\left(\B\h\right) \stackrel{\h \ltimes}{\longlongrightarrow} \left(S^{et}-Gpd\right)/\h \stackrel{Y}{\longrightarrow} \St\left(S/X\right),$$
where $q$ is the canonical inclusion.

\begin{prop}
The $2$-functor $G$ as defined above preserves coproducts and epimorphisms.
\end{prop}

\begin{proof}
The fact that $G$ preserves coproducts can be checked immediately. As far as epimorphisms, suppose that $$\varphi:E \to F$$ is an epimorphism in $\B \h$. To show that $G\left(\varphi\right)$ is an epimorphism in $\St\left(S/X\right)$, it suffices to show that the map $$t \circ pr_1:\left(\h \ltimes F\right)_1 \times_{F} E \to F$$ is an epimorphism in $\B\h$ where
$$\xymatrix{\left(\h \ltimes F\right)_1 \times_{F} E \ar[r]^-{pr_2} \ar[d]_-{pr_1} & E \ar[d]^-{\varphi} \\
\left(\h \ltimes F\right)_1 \ar[r]^-{s} & F}$$
is the pullback diagram in $\B\h$. However, $\left(\h \ltimes F\right)_1$ is itself the pullback
$$\xymatrix{\left(\h \ltimes F\right)_1 \ar[d] \ar[r]^-{s}& F \ar[d]^-{\nu}\\
\h_1 \ar[r]^-{s} & \h_0,}$$ where $\nu$ is the moment map of $F$. Since $\varphi$ is, in particular, a map in $S/\h_0,$ $\nu \circ \varphi=\mu,$ where $\mu$ is the moment map of $E$. Hence $\left(\h \ltimes F\right)_1 \times_{F} E$ is in fact the fibered product $$\h_1 \times_{\h_0} E = \left(\h \ltimes E\right)_1.$$ The map $t \circ pr_1$ can then be identified with $t \circ \h \ltimes \left(\varphi\right)$ which is an epimorphism since $t$ is and $\h \ltimes \left(\varphi\right)$ is the pullback of $\varphi$ along $$s:\left(\h \ltimes F\right)_1 \to F.$$
\end{proof}

\begin{cor}\label{cor:mk}
There exists a weak colimit preserving $2$-functor $$L_{\B\h}:St_{can}\left(\B\h\right) \to St\left(S/\X\right)$$ whose restriction to $\B\h$ along the Yoneda embedding $$\hat y_{\B\h}:\B\h \hookrightarrow St_{can}\left(\B\h\right)$$ agrees with $G$ up to equivalence.
\end{cor}

(We have suppressed the role of the Grothendieck universe $\mathcal{U}$ for simplicity.)

\begin{lem}\label{lem:mk}
Let $$Y:\left(S^{et}-Gpd\right)/\h \to \St\left(S\right)/\X$$ be the $2$-functor which sends a groupoid $\varphi:\G \to \h$ over $\h$ to $$\left[\varphi\right]:\left[\G\right] \to \left[\h\right]=\X.$$ Then for $U \subset \h_0$ an open subset, the stacks $y\left(U \hookrightarrow \h_0 \to \X\right)$ and $$Y\left(\theta_{m_U}\right)=\left[\theta_{m_U}\right]$$ are canonically equivalent in $\St\left(S\right)/\X$, where $m_U$ is the equivariant sheaf associated to the representable $U \in \sit\left(\h \right)_0$ (Definition \ref{dfn:mu}), and $$\theta_{m_U}:\h \ltimes m_U \to \h,$$ is as in the remark directly preceding Proposition \ref{prop:shver}
\end{lem}

\begin{proof}
$\h \ltimes m_U$ has objects $s^{-1}\left(U\right)$ and arrows are of the form $$\left(h,\gamma\right):\gamma \to h \circ \gamma.$$ Define an internal functor $$f_U:\h \ltimes m_U \to U^{id}$$ on objects as $$s^{-1}\left(U\right) \stackrel{s}{\longrightarrow} U$$ and on arrows by  $$\left(h,\gamma\right) \mapsto s\left(\gamma\right).$$ Define another internal functor $$g_U:U^{id} \to \h \ltimes m_U$$ on objects as $$x \mapsto \mathbb{1}_x$$ and on arrows as $$x \mapsto \left(\mathbb{1}_x,\mathbb{1}_x\right).$$ Clearly $$f_U \circ g_U = id_{U^{id}}.$$ Moreover, there is a canonical internal natural transformation $$\lambda_U:g_U  \circ f_U \Rightarrow id_{\h \ltimes m_U},$$ given by $$\lambda_U\left(\gamma\right)=\left(\gamma,\mathbb{1}_{s(\gamma)}\right).$$
Denote by $a_U:U^{id} \to \h,$ the composite $$U \hookrightarrow \h_0 \to \h.$$ Notice that $g_U$ extends to a morphism from $a_U:U^{id} \to \h$ to $\h \ltimes m_U$ in $\left(S^{et}-Gpd\right)/\h$ as $\theta_{m_U} \circ g_U=a_U$. Note that the formula $$\alpha_U\left(\gamma\right)=\gamma^{-1}$$ defines an internal natural transformation $$\alpha_U:\theta_{m_U} \Rightarrow a_U \circ f_U .$$ Hence $\left(f_U,\alpha_U\right)$ is morphism in $\left(S^{et}-Gpd\right)/\h$ from $\theta_{m_U}$ to $a_U$. It is easy to check that $\lambda_U$ is in fact a $2$-cell in $\left(S^{et}-Gpd\right)/\h$. Hence $a_U$ and $\theta_{m_U}$ are canonically equivalent, so the same is true of their images under $Y$.
\end{proof}

Consider the functor $$m:Site\left(\h\right) \to \B\h,$$ from Proposition \cite{pres}. Then, this is a morphism of sites, and in light of the aforementioned proposition, it induces an equivalence of bicategories
$$m_!:St\left(Site\left(\h\right)\right) \to St_{can}\left(\B\h\right).$$ The $2$-functor $m_!$ is the left Kan extension of $\hat y_{\B\h} \circ m$ (See Corollary \ref{cor:mk} for the notation) along the Yoneda embedding $$\hat y_{\h}:Site\left(\h\right) \hookrightarrow St\left(Site\left(\h\right)\right).$$ In other words, it is the unique weak colimit preserving $2$-functor whose restriction to $Site\left(\h\right)$ along $\hat y_{\h}$ agrees with $\hat y_{\B\h} \circ m$ up to equivalence.

\begin{cor}\label{cor:Lbh}
The \'etale realization $2$-functor $$\bar L:St\left(Site(\h\right)) \to \St\left(S/\X\right)$$ from Corollary \ref{cor:rest} is equivalent to $L_{\B\h} \circ m_!.$
\end{cor}
\begin{proof}
This follows from Lemma \ref{lem:mk} together with Corollary \ref{cor:mk}.
\end{proof}

\begin{lem}\label{lem:C4}
The composite $$Y \circ \h \ltimes:Gpd\left(\B\h\right) \to \St\left(S/\X\right)$$ preserves epimorphisms and weak pullbacks.
\end{lem}

\begin{proof}
Suppose that $\varphi:\K \to \Ll$ is an epimorphism in $Gpd\left(\B\h\right).$ This implies that the induced map $$t\circ pr_1:\Ll_1\times_{\Ll_0} \K_0 \to \Ll_0$$ is an epimorphism. In particular, this means that, when viewed as a map of underlying spaces, it is a surjective local homeomorphism, i.e. an \'etale cover. There is a canonical map $$\Ll_1\times_{\Ll_0} \K_0 \to \left(\h \ltimes \Ll\right)_1 \times_{\Ll_0} \K_0$$ induced by the canonical homomorphism $\tau_{\Ll}:\Ll \to \left(\h \ltimes \Ll\right)$ (See the remark directly preceding Proposition \ref{prop:shver}) such that the following diagram commutes
$$\xymatrix{& \left(\h \ltimes \Ll\right)_1 \times_{\Ll_0} \K_0 \ar[d]^-{t \circ pr_1}\\ \Ll_1\times_{\Ll_0} \K_0 \ar[ru] \ar[r]^-{t \circ pr_1} & \Ll_0.}$$ Hence $$t \circ pr_1:\left(\h \ltimes \Ll\right)_1 \times_{\Ll_0} \K_0 \to \Ll_0$$ is a surjective \'etale map. This implies that $Y \circ \h \ltimes\left(\varphi\right)$ is an epimorphism.

To show that $Y \circ \h \ltimes$ preserves weak pullbacks, it suffices to show that $\h \ltimes$ does, since $Y$ preserves finite weak limits. Suppose that $$\xymatrix{P \ar[r] \ar[d] & B \ar[d]^-{\beta}\\ A \ar[r]^-{\alpha} & C}$$ is a weak pullback diagram in $\B\h.$ Explicitly, we may describe $P$ by its objects being triples $$\left(a,b,l\right) \in A_0 \times B_0 \times C_1$$ such that $$l:\alpha\left(a\right) \to \beta\left(b\right),$$ with the obvious structure of an \'etale $\h$-space (we can take the moment map $\xi_0$ to be the projection onto $B_0$ followed by its moment map $\nu_0$). Its arrows from $$\left(a,b,l\right) \to \left(a',b',l'\right)$$ can be described by pairs $$\left(k_a,k_b\right)\in A_1 \times B_1$$ such that the following diagram commutes:
$$\xymatrix{\alpha\left(a\right) \ar[r]^-{l} \ar[d]_-{\alpha\left(k_a\right)} & \beta\left(b\right) \ar[d]^-{\beta\left(k_b\right)}\\ \alpha\left(a'\right) \ar[r]^-{l'} & \beta\left(b'\right).}$$
This condition on the arrows can be expressed as a pullback diagram, hence they also inherits the structure of an \'etale $\h$-space. Now, the objects of $\h \ltimes P$ are the same as $P$. The arrows $$\left(a,b,l\right) \to \left(a',b',l'\right)$$ in $\h \ltimes P$ can be described by triples $\left(h,k_{ha},k_{hb}\right)$ such that $$\left(k_{ha},k_{hb}\right):\left(ha,hb,hl\right) \to \left(a',b',l'\right)$$ is an arrow in $P.$ These of course assemble into a space which can be constructed via pullbacks as well, with moment map $\xi_1$. With the choice of moment maps $\xi,$ $\h \ltimes P$ becomes an $S$-groupoid over $\h$ by factoring it through $$\theta_B:\h \ltimes B \to \h.$$
Let $P'$ denote the weak pullback $$\xymatrix{P' \ar[r] \ar[d] & \h \ltimes A \ar[d]\\ \h \ltimes B \ar[r] & \h \ltimes C.}$$ Its objects can be described by quadruples $$\left(a,b,h,l\right) \in A_0 \times B_0 \times \h_1 \times C_1$$ such that $$k:h\alpha\left(a\right) \to \beta\left(b\right).$$ A quick calculation shows that its arrows $$\left(a,b,h,l\right) \to \left(a',b',h',l'\right)$$ can be described by quadruples $$\left(h_a,k_a,h_b,k_b\right) \in \h_1 \times A_1 \times \h_1 \times B_1$$ such that $$k_a:h_a a \to a'$$ and $$k_b:h_b b \to b',$$ and such that

\begin{equation}\label{eq:njn}
\left(h'h_a,l' \circ \left(h' \cdot \alpha\left(k_a\right)\right)\right)=\left(h_bh,\beta\left(k_b\right) \circ \left(h_b \cdot l\right)\right).
\end{equation}
We may regard $P'$ as an $S$-groupoid over $\h$ by factoring it through its canonical projection onto $\h \ltimes B.$

There is of course a canonically induced map $$F:\h \ltimes P \to P'$$ coming from the cone obtained by applying $\h \ltimes$ to the diagram expressing $P$ as a pullback. On objects, $F$ sends a triple $$\left(a,b,l\right) \mapsto \left(a,b,\mathbb{1}_{\mu_1\left(l\right)},l\right),$$ where $$\mu_1:C_1 \to \h_0$$ is the moment map. On arrows it sends $$\left(h,k_{ha},k_{hb}\right)\mapsto\left(h,k_{ha},h,k_{hb}\right).$$
Define a homomorphism $\Lambda:P' \to P$ on objects by $$\left(a,b,h,l\right) \mapsto \left(ha,b,\mathbb{1}_{\mu_1\left(l\right)},l\right)$$ and on arrows by sending quadruples $$\left(h_a,k_a,h_b,k_b\right):\left(a,b,h,l\right) \to \left(a',b',h',l'\right)$$ to triples $$\left(h'h_ah^{-1}=h_b,h'\cdot k_a,k_b\right).$$ Notice that $\Lambda$ strictly commutes over $\h$ and $\Lambda \circ F= id_{\h \ltimes P}.$ Moreover, consider the continuous map $$\omega:P'_0 \to P'_1$$ given by $$\left(a,b,h,l\right) \mapsto \left(h,\mathbb{1}_{ha},\mathbb{1}_{\nu_0\left(b\right)},\mathbb{1}_b\right),$$ where $\nu_0$ is the moment map of $B$. Notice that $$\omega\left(a,b,h,l\right):\left(a,b,h,l\right) \to \left(ha,b,\mathbb{1}_{\mu_1\left(l\right)},l\right).$$ It follows that $\omega$ is an internal natural transformation $$\omega:id_P' \Rightarrow F \circ \Lambda.$$ It is easy to check that it is indeed a $2$-morphism in $S-Gpd/\h,$ hence $F$ is an equivalence.
\end{proof}

\begin{thm}\label{thm:finally}
Consider the canonical $2$-functor
$$\left[ \mspace{3mu} \cdot \mspace{3mu}\right]_{\B\h}:Gpd\left(\B\h\right) \to St\left(\sit\left(\h\right)\right)$$
which associates a groupoid object $\K$ in in $\B\h$ with its stack completion. Then $\left[ \mspace{3mu} \cdot \mspace{3mu}\right]_{\B\h}$ is essentially surjective and faithful (but not in general full), and the $2$-functors $\bar L \circ \left[ \mspace{3mu} \cdot \mspace{3mu}\right]_{\B\h}$ and $Y \circ \h \ltimes$ are equivalent.
\end{thm}

\begin{proof}
The fact that $\left[ \mspace{3mu} \cdot \mspace{3mu}\right]_{\B\h}$ is essentially surjective follows from the fact that every stack is equivalent to a strict $2$-functor, which is in particular a sheaf of groupoids, i.e. a groupoid object in sheaves. The fact that is faithful follows from the fact that sheaves of groupoids considered as weak presheaves are separated (i.e. prestacks).
Let $\K$ be a groupoid object in $\B\h$. Then the following is a $2$-Cartesian cube all of whose edges are epimorphisms:
$$\xymatrix{
   &\K_2 \ar[rr]^{d_2}\ar'[d]^{d_1}[dd]
   \ar[ld]_{d_0}
   &&\K_1\ar[ld]^-{d_0}\ar[dd]^-{d_1}\\
   \K_1 \ar[rr]^(0.4){d_1}\ar[dd]^-{d_0}
   && \K_0 \ar[dd]^(0.4){p}\\
   &\K_1 \ar[ld]^-{d_0} \ar'[r]^-{d_1}[rr]
   &&\K_0,\ar[ld]^-{p}\\
   \K_0 \ar[rr]^-{p}&& \K
   }$$
where $p:\K_0 \to \K$ is the canonical map. From Lemma \ref{lem:C4}, $$Y\circ \h \ltimes \left(p\right):\left[\h \ltimes \K_0\right] \to \left[\h \ltimes \K\right]$$ is an epimorphism, and also, applying $Y\circ \h \ltimes$ to the above cube, results in another $2$-Cartesian cube, this time in the $2$-topos $St\left(S/\X\right),$ all of whose edges are again epimorphisms. The fact that this cube is Cartesian means that the diagram obtained by deleting the vertex $\left[\h\ltimes\K\right]$ and all edges into it, namely $$\left[\h \ltimes \K_2\right] \rrrarrow \left[\h \ltimes \K_1\right] \rrarrow \left[\h \ltimes \K_0\right],$$ is the truncated semi-simplicial C\v{e}ch nerve of the epimorphism $Y\circ \h \ltimes \left(p\right).$ From \cite{htt}, since $St\left(S/\X\right)$ is a $2$-topos, this implies that $$\left[\h \ltimes \K\right] \simeq \hc \left(\left[\h \ltimes \K_2\right] \rrrarrow \left[\h \ltimes \K_1\right] \rrarrow \left[\h \ltimes \K_0\right]\right).$$ Notice that in the $2$-topos $St\left(\sit(\h\right),$ $$\left[\K\right]_{\B\h} \simeq \hc \left(\K_2 \rrrarrow \K_1 \rrarrow \K_0\right).$$ From Corollary \ref{cor:Lbh}, this implies that $$\bar L \left[\K\right]_{\B\h} \simeq L_{\B\h}\circ m_!\left(\left[\K\right]_{\B\h}\right) \simeq \hc \left(\left[\h \ltimes \K_2\right] \rrrarrow \left[\h \ltimes \K_1\right] \rrarrow \left[\h \ltimes \K_0\right]\right).$$ Hence $$\bar L \left[\K\right]_{\B\h} \simeq \left[\h \ltimes \K\right].$$ We leave the rest of the details to the reader.

\end{proof}

\bibliographystyle{hplain}
\bibliography{main}

\begin{thebibliography}{10}

\bibitem{sga4}
{\em Th\'eorie des topos et cohomologie \'etale des sch\'emas. {T}ome 1:
  {T}h\'eorie des topos}.
\newblock Lecture Notes in Mathematics, Vol. 269. Springer-Verlag, Berlin,
  1972.
\newblock S{\'e}minaire de G{\'e}om{\'e}trie Alg{\'e}brique du Bois-Marie
  1963--1964 (SGA 4), Dirig{\'e} par M. Artin, A. Grothendieck, et J. L.
  Verdier. Avec la collaboration de N. Bourbaki, P. Deligne et B. Saint-Donat.

\bibitem{bicat}
Igor Bakovi\'c.
\newblock Grothendieck construction for bicategories.
\newblock Available at \url{http://www.irb.hr/users/ibakovic/sgc.pdf}.

\bibitem{diffg}
Kai Behrend and Ping Xu.
\newblock Differentiable stacks and gerbes, 2006.
\newblock \href{http://arxiv.org/abs/math/0605694}{arXiv:math/0605694}.

\bibitem{bunge}
Marta Bunge.
\newblock An application of descent to a classification theorem for toposes.
\newblock {\em Math. Proc. Cambridge Philos. Soc.}, 107(1):59--79, 1990.

\bibitem{IekeButz}
Carsten Butz and Ieke Moerdijk.
\newblock Representing topoi by topological groupoids.
\newblock {\em J. Pure Appl. Algebra}, 130(3):223--235, 1998.

\bibitem{bracket}
Marius Crainic and Rui~Loja Fernandes.
\newblock Integrability of {L}ie brackets.
\newblock {\em Ann. of Math. (2)}, 157(2):575--620, 2003.

\bibitem{Duskin}
John Duskin.
\newblock {An outline of non-Abelian cohomology in a topos. I: The theory of
  bouquets and gerbes.}
\newblock {\em Cahiers Topologie Geom. Differ\'entielle Cat\'egoriques},
  23:165--191, 1982.

\bibitem{FGA}
Barbara Fantechi, Lothar G{\"o}ttsche, Luc Illusie, Steven~L. Kleiman, Nitin
  Nitsure, and Angelo Vistoli.
\newblock {\em Fundamental algebraic geometry}, volume 123 of {\em Mathematical
  Surveys and Monographs}.
\newblock American Mathematical Society, Providence, RI, 2005.
\newblock Grothendieck's FGA explained.

\bibitem{Andre}
David Gepner and Andr\'e Henriques.
\newblock Homotopy theory of orbispaces.
\newblock \href{http://arxiv.org/abs/math/0701916}{arXiv:math/0701916}, 2007.

\bibitem{Giraud}
Jean Giraud.
\newblock {\em Cohomologie non ab\'elienne}.
\newblock Springer-Verlag, Berlin, 1971.
\newblock Die Grundlehren der mathematischen Wissenschaften, Band 179.

\bibitem{pres2}
Andre Henriques and David~S. Metzler.
\newblock Presentations of noneffective orbifolds.
\newblock {\em Trans. Amer. Math. Soc.}, 356(6):2481--2499 (electronic), 2004.

\bibitem{pointless}
Peter~T. Johnstone.
\newblock The point of pointless topology.
\newblock {\em Bull. Amer. Math. Soc. (N.S.)}, 8(1):41--53, 1983.

\bibitem{elephant2}
Peter~T. Johnstone.
\newblock {\em Sketches of an elephant: a topos theory compendium. {V}ol. 2},
  volume~44 of {\em Oxford Logic Guides}.
\newblock The Clarendon Press Oxford University Press, Oxford, 2002.

\bibitem{loc1}
Andr{\'e} Joyal and Myles Tierney.
\newblock An extension of the {G}alois theory of {G}rothendieck.
\newblock {\em Mem. Amer. Math. Soc.}, 51(309):vii+71, 1984.

\bibitem{pres}
Anders Kock and Ieke Moerdijk.
\newblock Presentations of 'etendues.
\newblock {\em Cahiers de Topologie et G\'eom\'etrie Diff\'erentielle
  Cat\'egoriques}, 32(2):145--164, 1991.

\bibitem{htt}
Jacob Lurie.
\newblock {\em Higher topos theory}, volume 170 of {\em Annals of Mathematics
  Studies}.
\newblock Princeton University Press, Princeton, NJ, 2009.

\bibitem{sheaves}
Saunders MacLane and Ieke Moerdijk.
\newblock {\em Sheaves in Geometry and Logic: A First Introduction to Topos
  Theory}.
\newblock Springer-Verlag, New York, 1st, edition, 1992.

\bibitem{Metzler}
David Metzler.
\newblock Topological and smooth stacks.
\newblock \href{http://arxiv.org/abs/math/0306176}{arXiv:math/0306176}, 2003.

\bibitem{cont}
Ieke Moerdijk.
\newblock The classifying topos of a continuous groupoid. i.
\newblock {\em Transactions of the American Mathematical Society},
  310(2):629--668, 1988.

\bibitem{Ie}
Ieke Moerdijk.
\newblock Foliations, groupoids and {G}rothendieck \'etendues.
\newblock {\em Rev. Acad. Cienc. Zaragoza (2)}, 48:5--33, 1993.

\bibitem{reg}
Ieke Moerdijk.
\newblock On the classification of regular groupoids.
\newblock \href{http://arxiv.org/abs/math/0203099}{arXiv:math/0203099}, 2002.

\bibitem{Fol}
Ieke. Moerdijk and Janez Mr{\v{c}}un.
\newblock {\em Introduction to foliations and {L}ie groupoids}, volume~91 of
  {\em Cambridge Studies in Advanced Mathematics}.
\newblock Cambridge University Press, Cambridge, 2003.

\bibitem{bunger1}
Michael~K. Murray.
\newblock Bundle gerbes.
\newblock {\em J. London Math. Soc (2).}, 54(2):403--416, 1996.

\bibitem{NoohiF}
Behrang Noohi.
\newblock Foundations of topological stacks i.
\newblock \href{http://arxiv.org/abs:math/0503247}{arXiv:math/0503247}, 2005.

\bibitem{Dorette}
Dorette Pronk.
\newblock Etendues and stacks as bicategories of fractions.
\newblock {\em Compositio Mathematica}, 102(3):243--303, 1996.

\bibitem{orbcon}
Joel~W. Robbin and Dietmar~A. Salamon.
\newblock Corrigendum: ``{A} construction of the {D}eligne-{M}umford orbifold''
  [{J}. {E}ur. {M}ath. {S}oc. ({JEMS}) {8} (2006), no. 4, 611--699;.
\newblock {\em J. Eur. Math. Soc. (JEMS)}, 9(4):901--905, 2007.

\bibitem{stacklie}
Hsian-Hua Tseng and Chenchang Zhu.
\newblock Integrating {L}ie algebroids via stacks.
\newblock {\em Compos. Math.}, 142(1):251--270, 2006.

\end{thebibliography}

\end{document}